\documentclass[10pt,a4paper,english]{article}

\usepackage{amsfonts,amssymb,stmaryrd,amsmath,amsthm,amssymb}
\usepackage{graphicx,color}
\usepackage{multimedia}
\usepackage{float}
\usepackage{placeins}

\usepackage{algorithm}

\usepackage[a4paper]{geometry}
\newtheorem{lemma}{Lemma}
\newtheorem{remark}{Remark}
\newtheorem{proposition}{Proposition}

\newcommand \bn {\boldsymbol{\mathrm{n}}}
\newcommand \bw {\boldsymbol{\mathrm{w}}}

\newcommand \bomega {\boldsymbol{\mathrm{\omega}}}

\newcommand \bu {\boldsymbol{\mathrm{u}}}
\newcommand \bv {\boldsymbol{\mathrm{v}}}
\newcommand \bg {\boldsymbol{\mathrm{g}}}
\newcommand \bff {\boldsymbol{\mathrm{f}}}
\newcommand \blambda {\boldsymbol{\mathrm{\lambda}}}
\newcommand \bLambda {\boldsymbol{\mathrm{\Lambda}}}
\newcommand \bmu {\boldsymbol{\mathrm{\mu}}}

\newcommand \bvarphi {\boldsymbol{\mathrm{\varphi}}}

\newcommand \bpsi {\boldsymbol{\mathrm{\psi}}}

\newcommand \p {\partial}

\newcommand \R {\mathbb{R}}
\newcommand \N {\mathbb{N}}
\renewcommand \L {\mathrm{L}}

\newcommand \I {\mathrm{I}}

\renewcommand \d {\mathrm{d}}

\renewcommand \div {\mathrm{div}}

\begingroup\makeatletter\ifx\SetFigFont\undefined%
\gdef\SetFigFont#1#2#3#4#5{%
  \reset@font\fontsize{#1}{#2pt}%
  \fontfamily{#3}\fontseries{#4}\fontshape{#5}%
  \selectfont}%
\fi\endgroup%

\title{XFEM based fictitious domain method for linear elasticity model with crack\thanks{This work is supported by...}}

\author{Olivier Bodart, Val\'erie Cayol, S\'ebastien Court\thanks{Laboratoire de Math\'ematiques UMR CNRS 6620, Universit\'e Blaise Pascal, Campus des C\'ezeaux, F-63177 Aubi\`ere Cedex, France, email: {\tt sebastien.court@math.univ-bpclermont.fr}.}, Jonas Koko}

\begin{document}

\maketitle

\begin{abstract}
Reduction of computational cost of solutions is a key issue to crack identification or crack propagation problems. One of the solution is to avoid re-meshing the domain when the crack position changes or when the crack extends. To avoid re-meshing, we propose a new finite element approach for the numerical simulation of discontinuities of displacements generated by cracks inside elastic media. The approach is based on a fictitious domain method originally developed for Dirichlet conditions for the Poisson problem and for the Stokes problem, which is adapted to the Neumann boundary conditions of crack problems. The crack is represented by level-set functions. Numerical tests are made with a mixed formulation to emphasize the accuracy of the method, as well as its robustness with respect to the geometry enforced by a stabilization technique. In particular an inf-sup condition is theoretically proven for the latter. A realistic simulation with a uniformly pressurized fracture inside a volcano is given for illustrating the applicability of the method.
\end{abstract}

\hfill 


\section{Introduction}
Recovering information on cracks located inside elastic media and predicting their propagation is a key topic for several geophysics and engineering problems. For instance, the analysis of ground deformation is used to search for the shape, position, and stress changes of magma-filled fractures at volcanoes (see \cite{Wauthier}) and of seismogenic faults \cite{Ozawa} in order to assess the associated hazards. Studying the propagation of fluid filled fractures is central to the study of hydraulic fracturing of hydrocarbon reservoirs or the formation of ore deposits.\\
These problems usually involve multiple computations for which only part of the boundaries is modified: Determination of crack characteristics from ground deformation usually requires hundreds of computations with different crack configurations and study of fracture propagation requires computation of incrementally larger crack surfaces.\\
For these problems to be numerically tractable, one possibility is to use methods in which the domain does not have to be meshed, such as boundary element methods \cite{Crouch}, for which meshing is limited to the boundaries so that modification of the system is only required for the modified boundaries. An other possibility is to use domain methods, such as fictitious domain methods in which meshes are non-conforming so that re-assembly of the whole system is avoided.\\
Boundary element methods involve full non-symmetric matrices. These methods \cite{Crouch} can take anisotropic media into account, but treating heterogeneous media requires the definition of new boundary, which increases matrices dimensions and involves very long computation times. As a consequence, heterogeneous media considered using boundary element methods only have two different material properties. On the other hand, domain methods, such as finite element methods, can deal with anisotropic and heterogeneous media with no increase in the system dimension and no extra numerical cost. In order to avoid re-meshing when the crack is updated and to treat heterogeneous media appropriately, we have chosen to develop a finite element method in which cracks are taken into account with a fictitious domain method. Besides accuracy, we require the method to be robust with respect to the problem geometry: The same results must be obtained whatever the way the crack is intersecting the mesh.\\

Our method is based on an artificial extension of the considered crack (Figure~\ref{figsplit}) in order to split the domain into two sub-domains. It is different from the way the crack is extended in \cite{Sokolowski}, where the extension is made such that the crack withdraws into itself. The aim is to double properly the degrees of freedom around the crack. Discontinuities of the displacement field generated by traction forces (Neumann-type boundary conditions) imposed on both sides of the crack have to be simulated. We can impose other Neumann-type conditions on this interface. The fictitious domain approach we implement is inspired by Xfem \cite{MoesD}, since it consists partially in cutting the basis functions near or around the interfaces; But, unlike Xfem \cite{Chahine, Laborde, Liu2014, reviewXfem}, we do not enrich our finite element spaces functions with singular functions, as we intend to avoid these enrichments by minimizing the number of updates when the position of the crack is modified. This approach has been first introduced in \cite{HaslR} for the Poisson problem, and in \cite{Court} in the context of Fluid-Structure Interactions, for Dirichlet conditions.\\

Across the artificial extension considered for the crack, there is no discontinuity, so we impose a homogeneous displacement jump condition with a Lagrange multiplier. This boundary is also taken into account with a fictitious domain approach, and the multiplier aforementioned is a dual variable for which we want a good approximation. This point is crucial for two reasons: First this quantity represents the constraints at stake in this region, whose the knowledge is required for crack propagation problems. Secondly this quantity takes part in some inversion algorithms (like gradient algorithms based on the computation of a direct adjoint system). Getting a good approximation of this multiplier is not guaranteed {\it a priori}, because the degrees of freedom considered on the underlying mesh do not match the crack and its extension, and it is the price to pay when we do not provide enrichment of basis functions like it is done with Xfem-type methods, where the enrichment of the basis functions has to be done according to the geometry of the crack, which makes part of the updates we want to avoid. For circumventing this drawback, we carry out a stabilization technique of augmented Lagrangian-type \`a la Barbosa-Hughes~\cite{Barbosa1, Barbosa2}. This technique theoretically ensures an unconditional optimal convergence for the multiplier, but improvement is mainly observed when we perform numerical tests related to the robustness with respect to the geometry.\\

The paper is organized as follows: In Section~\ref{secsetting} we introduce the theoretical problem, we explain why and how we adopt the extension method for considering the crack and we give the continuous abstract weak formulation. In Section~\ref{sec3} we first detail the discrete formulation, in particular the principles of the fictitious domain method developed for the interfaces associated with the crack. Then, in Section~\ref{sectha0}, we provide the theoretical analysis without the augmented Lagrangian technique. Next in Section~\ref{secaug} we introduce the stabilization technique which forces the multiplier to reach the desired value. Lemma~\ref{lemmainfsup} shows that an Inf-Sup condition is automatically satisfied while performing the stabilization technique. Section~\ref{secimplem} is devoted to practical explanations for the implementation. In Section~\ref{secexp0} we present some 2D numerical tests without stabilization which estimate the rates of convergence for the displacement as well as for the multiplier. First illustrations in 2D are also given. The choice of the stabilization parameter is discussed in Section~\ref{secgamma}, and this part is concluded with numerical tests in Section~\ref{secrates1} providing rates of convergence with the stabilization. In Section~\ref{secrobust} we show that the interest of the stabilization technique lies in the criteria of robustness with respect to the geometry. Last, a realistic 3D simulation of an over-pressured magma-filled fracture inside {\it Piton de la Fournaise} volcano is performed in Section~\ref{secPiton}. Conclusion is given in Section~\ref{seconc}.

\section{Setting of the problem} \label{secsetting}

\subsection{The original elastic problem}
 \label{sec2}
Given a domain $\Omega$ of $\R^d$ ($d = 2$ or $3$), and a crack $\Gamma_T \subset \subset \Omega$ represented by a curve or a surface parameterized by an injective mapping, we consider a static linear elasticity model governed by the following system:
\begin{eqnarray*}
\left\{ \begin{array}{rcl}
-\div \ \sigma_{L} (\bu) = \bff & & \text{in } \Omega , \\
\bu = 0 & & \text{on } \p \Omega, \\
\sigma_{L}(\bu)\bn  = p \bn & & \text{on } \Gamma_T.
\end{array}
\right.
\end{eqnarray*}
In this system the displacement of the solid is denoted by $\bu$, body forces (like the gravity) by $\bff$, and $\sigma_L(\bu) = 2\mu_L \varepsilon(\bu) + \lambda_L(\div \ \bu)\I_{\R^d}$ denotes the Lam\'e stress tensor, with
\begin{eqnarray*}
\varepsilon(\bu) & = & \frac{1}{2}\left( \nabla \bu + \nabla \bu ^T \right).
\end{eqnarray*}
The real numbers $\mu_L$ and $\lambda_L$ are the Lam\'e coefficients.
The pressure force of value $p > 0$ is applied on both sides of the crack $\Gamma_T$, so we have to specify the {\it outward} normal $\bn$ on $\Gamma_T$. In order to determine solutions on both sides of the crack, we relate displacements on each side of the fracture to the displacement discontinuity across the fracture.

\subsection{Extension of the fracture}
In order to give a sense to both sides of the crack, we have to be able to determine whether a point of the domain lies on one side of the fracture or the other. The most convenient way we have found consists in uncoupling the problem by setting two unknowns displacements instead of a global one. We extend the crack $\Gamma_T$ to $\Gamma$, as shown in Figure~\ref{figsplit} below:\\
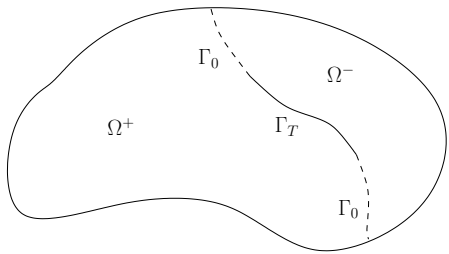
\begin{figure}[!h]
\begin{center}\hspace*{120pt}
\scalebox{0.40}{
\includegraphics{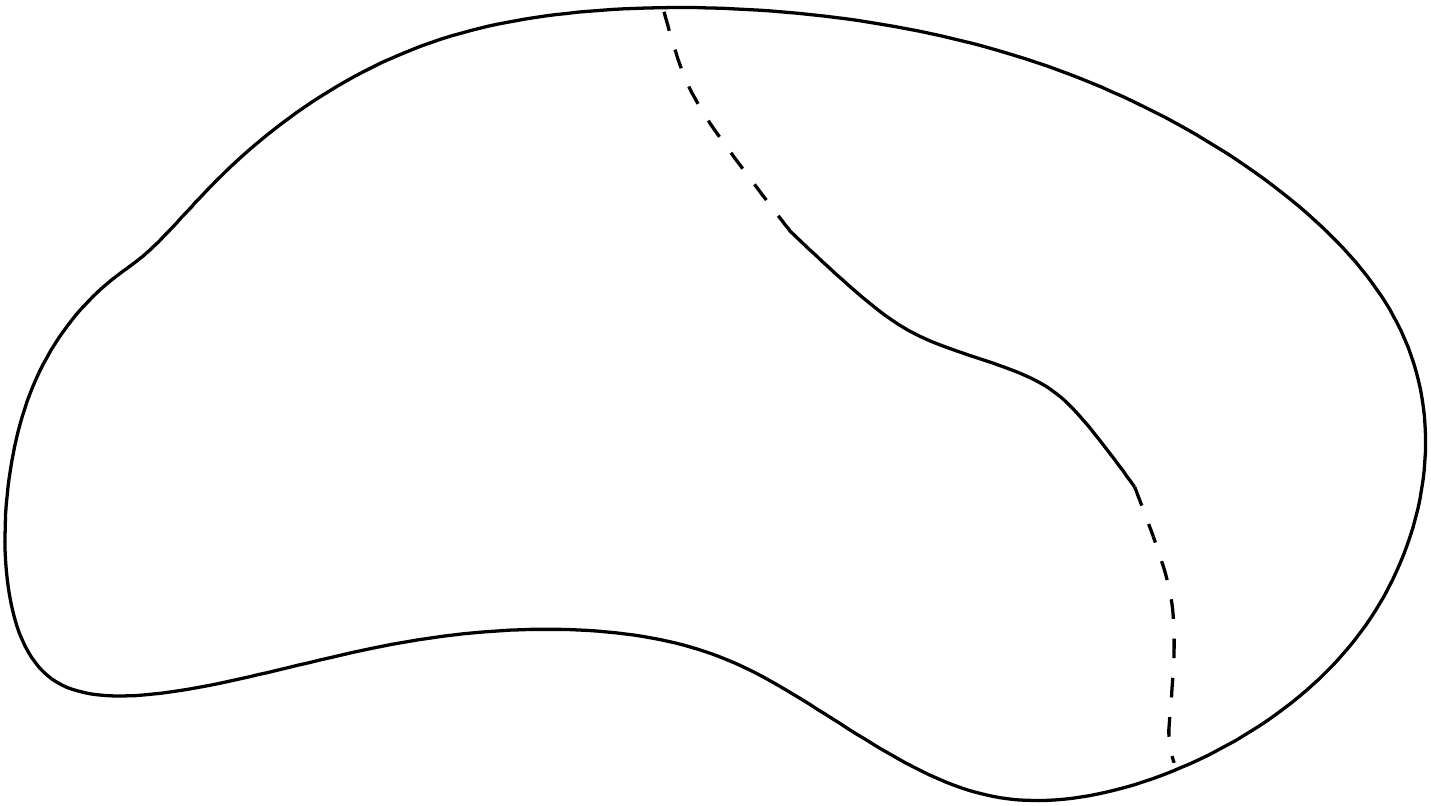}%
\setlength{\unitlength}{4144sp}%
\begin{picture}(6539,3682)(1473,-4801)
\put(-1086,-3011){\makebox(0,0)[lb]{\smash{{\SetFigFont{22}{24.4}{\familydefault}{\mddefault}{\updefault}{\color[rgb]{0,0,0}$\Gamma_T$}%
}}}}
\put(-3561,-3046){\makebox(0,0)[lb]{\smash{{\SetFigFont{22}{24.4}{\familydefault}{\mddefault}{\updefault}{\color[rgb]{0,0,0}$\Omega^+$}%
}}}}
\put(-2220,-1996){\makebox(0,0)[lb]{\smash{{\SetFigFont{22}{24.4}{\familydefault}{\mddefault}{\updefault}{\color[rgb]{0,0,0}$\Gamma_0$}%
}}}}
\put(-150,-4246){\makebox(0,0)[lb]{\smash{{\SetFigFont{22}{24.4}{\familydefault}{\mddefault}{\updefault}{\color[rgb]{0,0,0}$\Gamma_0$}%
}}}}
\put(-311,-2266){\makebox(0,0)[lb]{\smash{{\SetFigFont{22}{24.4}{\familydefault}{\mddefault}{\updefault}{\color[rgb]{0,0,0}$\Omega^-$}%
}}}}
\end{picture}
}
\end{center}
\centering \caption{Splitting of the domain according to the crack.\label{figsplit}}
\end{figure}
\FloatBarrier

Let us justify why we adopt this method of extension of the crack. First, in view of the boundary conditions we impose on the crack $\Gamma_T$, it is necessary to define two unit normal vectors on this interface, and thus to be able to split the domain as a function of the discontinuity induced by these boundary conditions. As aforementioned in the introduction, it is then more convenient to consider a configuration for which we are able to locate a point with respect to the crack.\\
The relevance of such an approach can be underlined by the theoretical context. Indeed, the Cauchy problem as well as the identification for crack-like boundaries is originally ill-posed. See for instance \cite{cauchycrack}. Splitting the domain as described above can enable us to recover existence and uniqueness of a (at least weak) solution.\\
The question of the possibility of making a crack extension has to be answered. Namely we seek the conditions under which a crack extension splitting domain $\Omega$ into two sub-domains exists. A first answer could be to only consider cracks for which an extension exists. Such cracks necessarily have to be represented as an injective curve/surface. An argument for this answer is more physical and results from the fact that, modeling crustal displacements, the domain considered includes cracks. In \cite{Fukushima2005} fractures inside a realistic 3D volcano were described by a specific parameterization, which can easily be extended. Moreover, this extension can be constructed and implemented in practice. For more details on this kind of extension, we refer to Section~\ref{secextension}.

\subsection{The transformed problem}

\noindent The global domain $\Omega$ is now split into two sub-domains $\Omega^+$ and $\Omega^-$. We have:
\begin{eqnarray*}
\Gamma = \Gamma_0 \cap \Gamma_T, & \quad & \Omega = \Omega^+ \cup \Gamma \cup \Omega^-.
\end{eqnarray*}
Let us now denote $\bu^+ = \bu_{| \Omega^+}$ and $\bu^- = \bu_{| \Omega^+}$. On the artificial boundary $\Gamma_0$ - which is not connected - we have to ensure the continuity of the displacement, namely $\bu^+ - \bu^- = 0$. The system becomes:\footnote{The symbol $\pm$ represents the fact that we consider both formulations involving the symbols $+$ and $-$, for the sake of concision. The outward normal of domain $\Omega^{\pm}$ is denoted by $\bn^{\pm}$.}
\begin{eqnarray}
\left\{ \begin{array}{rcl}
-\div \ \sigma_{L} (\bu^{\pm}) = \bff & & \text{in } \Omega^{\pm}, \\
\bu^{\pm} = 0 & & \text{on } \p \Omega \cap \p \Omega^{\pm}, \\
\left( \sigma_{L}(\bu)\bn \right)^{\pm} = p \bn^{\pm} & & \text{on } \Gamma_T, \\
\left[ \bu \right] = 0 & & \text{across } \Gamma_0 = \Gamma \setminus \Gamma_T,\\
\left[ \sigma_L(\bu) \right]\bn^+ = 0 & & \text{across } \Gamma_0.
\end{array}
\right. \label{pbdec}
\end{eqnarray}
Notation $\left[ \bvarphi \right] = \bvarphi^+ - \bvarphi^-$ refers to the jump of a function $\bvarphi$ across $\Gamma_0$. We assume that $\bff \in \mathbf{L}^2(\Omega)$, and we still denote by $\bff$ its restriction to $\Omega^+$ or $\Omega^-$. The homogeneous Dirichlet condition on $\p \Omega$ can be replaced by non-homogeneous boundary conditions mixing Neumann conditions and Dirichlet conditions. Moreover, the pressure condition on $\Gamma_T$ can be replaced by general conditions, as $\left( \sigma_{L}(\bu)\bn \right)^{\pm} = \pm\boldsymbol{\mathrm{g}}$.

\subsection{Continuous formulation}
Consider the following functional spaces:
\begin{eqnarray*}
\mathbf{V}^+ & = & \left\{ \bv\in \mathbf{H}^1(\Omega^+) \mid \bv = 0 \text{ on } \p \Omega \cap \p \Omega^+ \right\}, \\
\mathbf{V}^- & = & \left\{ \bv\in \mathbf{H}^1(\Omega^-) \mid \bv = 0 \text{ on } \p \Omega \cap \p \Omega^- \right\}, \\
\mathbf{W} & = & \left(\mathbf{H}^{1/2}(\Gamma_0) \right)'.
\end{eqnarray*}
We choose to impose the jump condition on $\Gamma_0$ by a Lagrange multiplier $\blambda$. A weak solution of system~\eqref{pbdec} can be seen as the stationary point in $\mathbf{V}^+\times \mathbf{V}^- \times \mathbf{W}$ of the following Lagrangian:
\begin{eqnarray}
\mathcal{L}_0(\bu^+, \bu^-, \blambda) & = & \frac{1}{2}\int_{\Omega^+} \sigma_{L}(\bu^+) : \varepsilon(\bu^+)\d \Omega^+ + \frac{1}{2}\int_{\Omega^-} \sigma_{L}(\bu^-) : \varepsilon(\bu^-)\d \Omega^- \nonumber \\
& & - \int_{\Omega^+} \bff \cdot \bu^+  \d \Omega^+ 
- \int_{\Omega^-} \bff \cdot \bu^-  \d \Omega^- 
- \int_{\Gamma_T}\bu^+ \cdot p\bn^+ \d \Gamma_T - \int_{\Gamma_T}\bu^- \cdot p\bn^- \d \Gamma_T \nonumber \\
& & + \left\langle \blambda , (\bu^+ - \bu^-) \right\rangle_{\mathbf{W}, \mathbf{W}'}. \label{L0}
\end{eqnarray}
In this expression $\displaystyle \langle \ \cdot \ , \ \cdot \ \rangle_{\mathbf{W},\mathbf{W}'}$ denotes the duality pairing between $\mathbf{W}$ and $\mathbf{W}'$, and $\sigma_{L}(\bu) : \varepsilon(\bu) = \mathrm{trace}\left(\sigma_L(\bu)\varepsilon(\bu)^T\right)$ denotes the classical inner product for matrices. Let us recall that the bilinear form $(\bu, \bv) \mapsto \sigma_{L}(\bu) : \varepsilon(\bv)$ is symmetric. 

\begin{remark} \label{remarkcont}
Note that in the writing of this Lagrangian the jump condition $\displaystyle \left[ \sigma_L(\bu) \right]\bn^+ = 0$ on $\Gamma_0$ (fifth equation of~\eqref{pbdec}) is no longer taken into account. 
Indeed, the first-order optimality conditions for $\mathcal{L}_0$ give
\begin{eqnarray*}
\int_{\Omega^{\pm}} \sigma_{L}(\bu^{\pm}) : \varepsilon(\bv)\d \Omega^{\pm}
\pm \left\langle \blambda , \bv \right\rangle_{\mathbf{W},\mathbf{W}'} & = & 
\int_{\Omega^{\pm}} \bff \cdot \bv \d \Omega^{\pm}+ \int_{\Gamma_T} \bv \cdot p\bn^{\pm} \d \Gamma_T, 
\end{eqnarray*}
for all test function $\bv \in \mathbf{V}^{\pm}$. On the other hand, taking the inner product by $\bv$ of the first equation of~\eqref{pbdec} yields, after integration by parts
\begin{eqnarray*}
\int_{\Omega^{\pm}} \sigma_{L}(\bu^{\pm}) : \varepsilon(\bv)\d \Omega^{\pm}
- \left\langle \sigma_L(\bu^{\pm})\bn^{\pm} , \bv \right\rangle_{\mathbf{W},\mathbf{W}'} & = & 
\int_{\Omega^{\pm}} \bff \cdot \bv \d \Omega^{\pm} + \int_{\Gamma_T} \bv \cdot p\bn^{\pm} \d \Gamma_T.
\end{eqnarray*}
Thus we identify $\blambda = -\sigma_L(\bu^+)\bn^+$ and $\blambda = \sigma_L(\bu^-)\bn^-$.
\end{remark}

The variational problem derived from the functional $\mathcal{L}_0$ as first-order optimality conditions is then
\begin{eqnarray}
& & \text{Find $(\bu^+, \bu^-, \blambda)$ in $\mathbf{V}^+\times \mathbf{V}^- \times \mathbf{W}$ such that} \nonumber \\
& & \left\{ \begin{array} {lcl}
\mathcal{A}_0^{\pm}((\bu^+, \bu^-, \blambda);\bv) = l^{\pm}(\bv) & \quad & \forall \bv \in \mathbf{V}^{\pm}, \\
\mathcal{B}_0((\bu^+, \bu^-, \blambda);\bmu) = 0, & \quad & \forall \bmu \in \mathbf{W},
\end{array} \right. 
\label{pbfv0}
\end{eqnarray}
where we set
\begin{eqnarray}
\mathcal{A}_0^{\pm}((\bu^+, \bu^-, \blambda);\bv) & = & 
\int_{\Omega^{\pm}} \sigma_{L}(\bu^{\pm}) : \varepsilon(\bv)\d \Omega^{\pm}
\pm \left\langle \blambda , \bv \right\rangle_{\mathbf{W},\mathbf{W}'}, \nonumber\\
l^{\pm}(\bv) & = &  \int_{\Omega^{\pm}} \bff \cdot \bv  \d \Omega^{\pm}
+ \int_{\Gamma_T} \bv \cdot p\bn^{\pm} \d \Gamma_T, \label{Lrhs}\\
\mathcal{B}_0((\bu^+, \bu^-, \blambda);\bmu) & = & 
 \left\langle \bmu , \bu^+ \right\rangle_{\mathbf{W},\mathbf{W}'} - \left\langle \bmu , \bu^- \right\rangle_{\mathbf{W}, \mathbf{W}'}. \nonumber
\end{eqnarray}

\section{Discrete formulation} \label{sec3}
The discrete formulation is adapted from \cite{HaslR} and \cite{Court}: It is a {\it fictitious domain} method, in which the choice of the degrees of freedom for the multiplier on the boundary $\Gamma_0$ is made independently of the mesh. Let us first explain how we proceed to take into account degrees of freedom which do not originally lie on the edges of the mesh.

\subsection{The fictitious domain approach} \label{secfict}
Unknowns for the fictitious domains are first considered on the whole domain  $\Omega$. Let us consider some discrete finite element spaces, {$\tilde{\mathbf{V}}_h \subset \mathbf{H}^1(\Omega)$ and $\tilde{\mathbf{W}}_h \subset \mathbf{L}^2(\Omega)$}. These spaces can be defined on the same structured mesh of $\Omega$, which can be chosen Cartesian. We set
\begin{eqnarray}
\tilde{\mathbf{V}}_h & = & \left\{\bv_h \in C(\overline{\Omega})\mid {\bv_h}_{\left| \p \Omega\right.} = 0, \  {\bv_h}_{\left| T\right.} \in P(T), \ \forall T \in \mathcal{T}_h \right\}, \label{defvtilde}
\end{eqnarray}
where $P(T)$ is a finite dimensional space of regular functions, including polynomial functions of order $k \geq 1$ on a triangle $T$ in the set $\mathcal{T}_h$ of the triangles of the mesh. See \cite{Ern} for more details. The mesh parameter stands for $\displaystyle h = \max_{T\in \mathcal{T}_h} h_T$, where $h_T$ is the diameter of the triangle $T$. We define
\begin{eqnarray*}
\mathbf{V}^{+}_h := {\tilde{\mathbf{V}}{}_h}_{\left| \Omega^{+} \right.}, \quad 
\mathbf{V}^{-}_h := {\tilde{\mathbf{V}}{}_h}_{\left| \Omega^{-} \right.}, \quad  
\mathbf{W}_h := {\tilde{\mathbf{W}}{}_h}_{\left| \Gamma_0 \right.},
\end{eqnarray*}
which are natural respectively discretizations of $\mathbf{V}^+$, $\mathbf{V}^-$ and $\left(\mathbf{H}^{1/2}(\Gamma_0) \right)'$. The selection of degrees of freedom for these spaces is illustrated in Figure~\ref{figdof}. This approach is similar to the eXtended Finite Element Method \cite{MoesD}, except that the standard basis functions near the boundary $\Gamma$ are not enriched by singular functions but only multiplied by Heaviside functions ($H({\bf x}) = 1$ for ${\bf x} \in \Omega^{\pm}$ and $H({\bf x})=0$ for ${\bf x}\in \Omega \setminus \Omega^{\pm}$), and the corresponding products come up in the integrals of the variational formulation of the problem. This kind of strategy is also adopted in \cite{Gerstenberger2008} and \cite{Choi2010} for instance.

\begin{figure}[!h]
\begin{center}
\begin{tabular} {ccc}
\includegraphics[trim = 0cm 0cm 0cm 0cm, clip, scale=0.37]{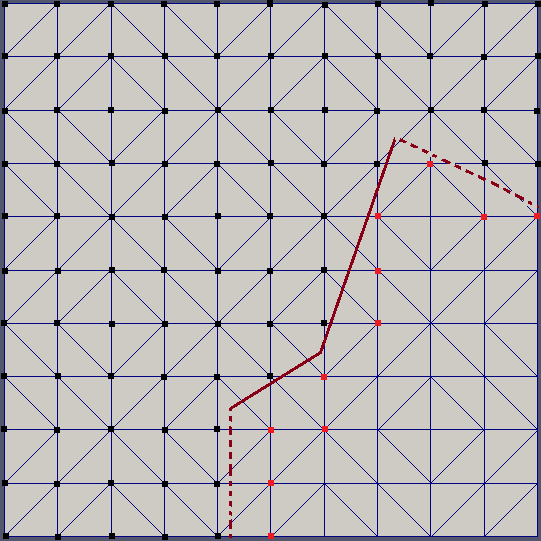}
&
\includegraphics[trim = 0cm 0cm 0cm 0cm, clip, scale=0.37]{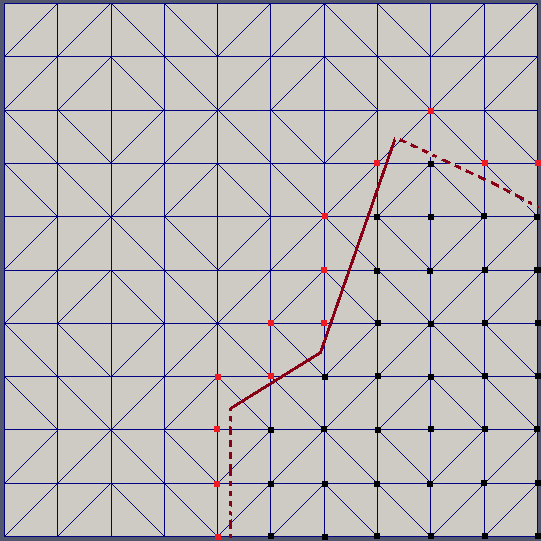}
&
\includegraphics[trim = 0cm 0cm 0cm 0cm, clip, scale=0.37]{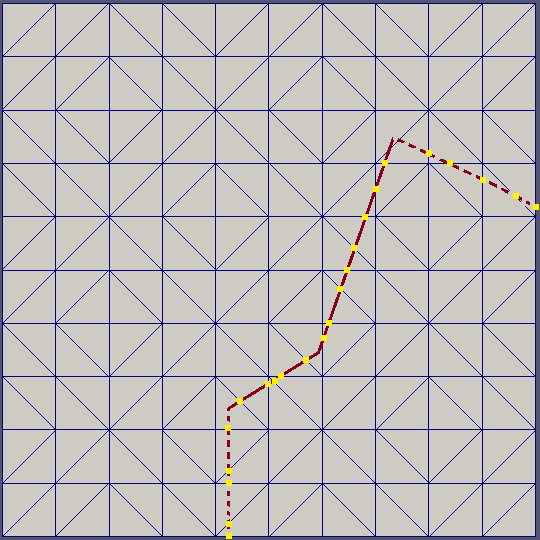}
\end{tabular}
\end{center}
\centering\caption{Selection of degrees of freedom: The black base nodes are kept for displacements $\bu^{\pm}$, the red ones are used for cutting the standard basis functions, and the yellow ones are those of the multiplier $\blambda$.
\label{figdof}}
\end{figure}
\FloatBarrier

An approximation of problem~\eqref{pbfv0} is given as follows
\begin{eqnarray}
& & \text{Find $(\bu^+_h, \bu^-_h, \blambda_h)$ in $\mathbf{V}^+_h\times \mathbf{V}^-_h \times \mathbf{W}_h$ such that} \nonumber \\
& & \left\{ \begin{array} {lcl}
a_0^{+}(\bu^{+}_h,\bv^+_h) + b_0(\blambda_h, \bv^+_h) = l^{+}(\bv^+_h) & \quad & \forall \bv^+_h \in \mathbf{V}^{+}_h, \\
a_0^{-}(\bu^{-}_h,\bv^-_h) - b_0(\blambda_h, \bv^-_h) = l^{-}(\bv^-_h) & \quad & \forall \bv^-_h \in \mathbf{V}^{-}_h, \\
b_0(\bu^+_h,\bmu_h) - b_0(\bu^-_h,\bmu_h) = 0, & \quad & \forall \bmu_h \in \mathbf{W}_h,
\end{array} \right. \label{discfv}
\end{eqnarray}
where
\begin{eqnarray}
a_0^{+}(\bu^{+},\bv^+) = \int_{\Omega^+} \sigma_{L}(\bu^+) : \varepsilon(\bv^+)\d \Omega^+, & \quad &
a_0^{-}(\bu^{-},\bv^-) = \int_{\Omega^-} \sigma_{L}(\bu^-) : \varepsilon(\bv^-)\d \Omega^-, \label{fbil1} \\
b_0(\bu^{\pm},\bmu_h) = \int_{\Gamma_0} \bmu \cdot \bu^{\pm} \d \Gamma_0.
\label{fbil2}
\end{eqnarray}
Note that the duality pairing between $\left(\mathbf{H}^{1/2}(\Gamma_0) \right)'$ and $\mathbf{H}^{1/2}(\Gamma_0)$ has been turned in the expression of~$b_0$ into the inner product of $\mathbf{L}^2(\Gamma_0)$. We could avoid this by using a Laplace-Beltrami operator, but under strong regularity assumptions, it is simpler to proceed like this. Thus we now consider that
\begin{eqnarray*}
\mathbf{W}' \equiv \mathbf{W}  =  \mathbf{L}^2(\Gamma_0),  &\quad &
\mathbf{W}_h  \subset \mathbf{L}^2(\Gamma_0),
\end{eqnarray*}
but we can keep the abstract formalism between $\mathbf{W}$ and $\mathbf{W}'$ in the mathematical analysis. 

\subsection{Theoretical convergences - limitations on the orders} \label{sectha0}
We make the following hypothesis:
\begin{eqnarray*}
(\mathbf{H}):  \qquad \qquad  \overline{\bmu}_h \in \mathbf{W}_h: 
\ b_0(\bv^{\pm},\overline{\bmu}_h) = 0 \ \forall \bv^{\pm} \in \mathbf{V}^{\pm}_h
& \Rightarrow & \overline{\bmu}_h = 0.
\end{eqnarray*}
This hypothesis is not as strong as an {\it Inf-sup condition} which would lead automatically to the optimal order of convergence for all the variables - primal and dual. See \cite{Fortin} for more details. It only requires that the discretization space for the displacement is at least as rich as the one chosen for the multiplier on $\Gamma_0$.

\begin{lemma} \label{lemmaco}
The bilinear forms $a_0^{+}$ and $a_0^-$ defined in~\eqref{fbil1} as
\begin{eqnarray*}
a_0^{\pm} : (\bu, \bv) \mapsto \int_{\Omega^{\pm}} \sigma_L(\bu):\varepsilon(\bv) \d \Omega^{\pm}
\end{eqnarray*}
are respectively uniformly $\mathbf{V}^+_h$-elliptic and $\mathbf{V}^-_h$-elliptic. Namely there exists $\alpha >0$ independent of $h$ such that for all $\bv^{\pm}_h \in \mathbf{V}_h^{\pm}$ we have
\begin{eqnarray*}
a_0^{\pm}(\bv^{\pm},\bv^{\pm}) & \geq & \alpha \left\| \bv^{\pm}_h\right\|^2_{\mathbf{V}^{\pm}}.
\end{eqnarray*}
\end{lemma}

\begin{proof}
Since $\mathbf{V}_h^{\pm} \subset \mathbf{V}^{\pm}$, let us show the coerciveness of $a^{\pm}$ in $\mathbf{V}^{\pm}$. Arguing by contradiction, assume that there exists $\displaystyle \left(\bv^{\pm}_n\right)_n$ such that for all $n\in \N$ we have
\begin{eqnarray}
n\int_{\Omega^{\pm}} \sigma_L(\bv_n^{\pm}) : \varepsilon(\bv_n^{\pm}) \d \Omega^{\pm} 
& < & \left\| \bv_n^{\pm} \right\|^2_{\mathbf{V}^{\pm}}. \label{ineqn}
\end{eqnarray}
After a quick calculation we have
\begin{eqnarray*}
 \sigma_L(\bv_n^{\pm}) : \varepsilon(\bv_n^{\pm}) & = & 
 2\mu_L | \varepsilon(\bv_n^{\pm}) |^2_{\R^{d\times d}} + \lambda_L (\div \ \bv_n^{\pm})^2.
\end{eqnarray*}
On the other hand, because of the homogeneity of the inequality \eqref{ineqn}, we can assume that $\displaystyle \left\| \bv_n^{\pm} \right\|_{\mathbf{V}^{\pm}} = 1$, without loss of generality. Then, inequality~\eqref{ineqn} implies in particular that $\varepsilon(\bv_n^{\pm})$ tends to $0$ in $\displaystyle [ \L^2(\Omega^{\pm}) ]^{d\times d}$. The Rellich's theorem gives us a subsequence $\displaystyle \left(\bv_m^{\pm}\right)_m$ which has a limit in $\mathbf{L}^2(\Omega^{\pm})$. From the Korn's second inequality (see for instance \cite{Ern}, Theorem 3.78, or \cite[page 10]{Ciarlet97} for the proof) there exists $C > 0$ such that\footnote{In the following, $C$ denotes a generic positive constant independent of the mesh size~$h$.}
\begin{eqnarray*}
\left\| \bv^{\pm}_m - \bv^{\pm}_p \right\|_{\mathbf{H}^1(\Omega^{\pm})} & \leq & C
\left( \left\| \bv^{\pm}_m - \bv^{\pm}_p \right\|_{\mathbf{L}^2(\Omega^{\pm})} + 
\left\| \varepsilon(\bv^{\pm}_m) - \varepsilon(\bv^{\pm}_p) \right\|_{[ \L^2(\Omega^{\pm}) ]^{d\times d}}  \right).
\end{eqnarray*}
Thus $\displaystyle \left(\bv^{\pm}_m\right)_m$ is a Cauchy sequence in $\mathbf{H}^1(\Omega^{\pm})$, and so it tends to some $\bv^{\pm}_{\infty}$ which satisfies in particular $\varepsilon(\bv^{\pm}_{\infty}) = 0$. It is then well-known that $\bv^{\pm}_{\infty}$ reduces to a rigid displacement, that is to say $\bv^{\pm}_{\infty}$ is an affine function of type $\bv^{\pm}_{\infty}(\mathbf{x}) = \mathbf{l}^{\pm} + \mathbf{r}^{\pm}\wedge \mathbf{x}$ (see \cite[page 18]{Temam}). Besides, the trace theorem implies that we have also $\bv^{\pm}_{\infty} = 0$ on $\p \Omega^{\pm}$. It enables us to conclude that $\mathbf{l}^{\pm} = \mathbf{r}^{\pm} = 0$ and thus $\bv^{\pm}_{\infty} = 0$ in $\Omega^{\pm}$, which belies the assumption $\displaystyle \left\| \bv_n^{\pm} \right\|_{\mathbf{V}^{\pm}} = 1$.
\end{proof}


\begin{proposition}
Assume that the hypothesis $(\mathbf{H})$ holds. Then Problem~\eqref{discfv} admits a unique solution $(\bu_h^+,\bu_h^-,\blambda_h)$.
\end{proposition}

\begin{proof}
Problem~\eqref{discfv} is of finite dimension, so the existence of the solution is a consequence of its uniqueness. For proving uniqueness, it is sufficient to show that $(\bu_h^+,\bu_h^-,\blambda_h) = 0$ when $l^{\pm} \equiv 0$. If we add the first two equations of~\eqref{discfv} while choosing $\bv^+_h = \bu^-_h$ and $\bv^+_h = \bu^-_h$, the third equation of~\eqref{discfv} enables us to write
\begin{eqnarray*}
a_0^{+}(\bu^{+}_h,\bu^+_h) + a_0^{-}(\bu^{-}_h,\bu^-_h) & = & 0,
\end{eqnarray*}
and thus by Lemma~\ref{lemmaco} we get $\bu^+_h = \bu^-_h = 0$. Then, the first two equations of~\eqref{discfv} reduce to the conditions of Hypothesis $(\mathbf{H})$, and thus we can conclude that $\blambda_h = 0$.
\end{proof}

We now define the space
\begin{eqnarray*}
\mathbf{\mathbb{V}}_h^0 & = & \left\{
(\bv_h^+, \bv_h^-) \in \mathbf{V}^+_h \times \mathbf{V}^-_h \mid 
b_0(\bv_h^+ - \bv_h^-,\bmu_h ) = 0 \ \forall \bmu_h \in \mathbf{W}_h 
\right\}.
\end{eqnarray*}
Let us give the abstract error estimate for the displacement.

\begin{proposition}
Assume that the hypothesis $(\mathbf{H})$ is satisfied. Let $(\bu^+,\bu^-,\blambda)$ and $(\bu_h^+,\bu_h^-,\blambda_h)$ be the respective solutions of Problems~\eqref{pbfv0} and~\eqref{discfv}. There exists a constant $C>0$ - independent of $h$ - such that
\begin{eqnarray}
\left\| \bu^+ - \bu^+_h \right\|_{\mathbf{V}^+} + \left\| \bu^- - \bu^-_h \right\|_{\mathbf{V}^-} & \leq & 
C\left(
\inf_{(\bv_h^+, \bv_h^-) \in \mathbf{\mathbb{V}}^0_h}
\left( \left\| \bu^+ - \bv^+_h \right\|_{\mathbf{V}^+} + \left\| \bu^- - \bv^-_h \right\|_{\mathbf{V}^-} \right) \right. 
 \left. +
\inf_{\bmu_h\in \mathbf{W}_h} \left\| \blambda - \bmu_h \right\|_{\mathbf{W}}
\right). \nonumber \\ \label{estapriori}
\end{eqnarray}
\end{proposition}

\begin{proof}
From Lemma~\ref{lemmaco}, for all test functions $(\bv^+_h,\bv^-_h) \in \mathbf{\mathbb{V}}_h^0$ we have
\begin{eqnarray*}
\alpha \left\| \bu_h^{\pm} - \bv^{\pm}_h \right\|^2_{\mathbf{V}^{\pm}} & \leq & 
a_0^{\pm}(\bu^{\pm}_h - \bv^{\pm}_h, \bu^{\pm}_h - \bv^{\pm}_h) \\
& = & a_0^{\pm}(\bu^{\pm} - \bv^{\pm}_h, \bu^{\pm}_h - \bv^{\pm}_h) +
l(\bu^{\pm}_h - \bv^{\pm}_h) - a_0^{\pm}(\bu^{\pm}, \bu^{\pm}_h - \bv^{\pm}_h)\\
& = & a_0^{\pm}(\bu^{\pm} - \bv^{\pm}_h, \bu^{\pm}_h - \bv^{\pm}_h) \pm
\left\langle \blambda, \bu^{\pm}_h - \bv^{\pm}_h \right\rangle_{\mathbf{W}, \mathbf{W}'}.
\end{eqnarray*}
Summing the two inequalities above which correspond to the symbols $+$ and $-$, and denoting $\bw_h^{\pm} = \bu^{\pm}_h - \bv^{\pm}_h$, we get
\begin{eqnarray}
\alpha \left( \left\| \bw^{+}_h \right\|^2_{\mathbf{V}^{+}} + 
 \left\| \bw^{-}_h \right\|^2_{\mathbf{V}^{-}} \right)  & \leq &
   a_0^{+}(\bu^{+} - \bv^{+}_h, \bw^{+}_h ) +
  a_0^{-}(\bu^{-} - \bv^{-}_h, \bw^{-}_h ) \nonumber \\
& &  +
\left\langle \blambda, (\bu_h^+ - \bu^-_h) - (\bv_h^+ - \bv^-_h) \right\rangle_{\mathbf{W}, \mathbf{W}'}. \label{ineqcea}
\end{eqnarray}
From the definition of $\mathbf{\mathbb{V}}_h^0$, and noticing that $(\bu_h^+,\bu_h^-) \in \mathbf{\mathbb{V}}_h^0$, we can write
\begin{eqnarray*}
\left\langle \blambda, (\bu_h^+ - \bu^-_h) - (\bv_h^+ - \bv^-_h) \right\rangle_{\mathbf{W}, \mathbf{W}'} & = & 
\left\langle \blambda - \bmu_h, (\bu_h^+ - \bu^-_h) - (\bv_h^+ - \bv^-_h) \right\rangle_{\mathbf{W}, \mathbf{W}'} \quad \forall \bmu_h \in \mathbf{W}_h.
\end{eqnarray*}
Thus, for $(\bv^+_h,\bv^-_h) \in \mathbf{\mathbb{V}}_h^0$ and $\bmu_h \in \mathbf{W}_h$, the inequality~\eqref{ineqcea} becomes
\begin{eqnarray*}
\alpha \left( \left\| \bw^{+}_h \right\|^2_{\mathbf{V}^{+}} + 
 \left\| \bw^{-}_h \right\|^2_{\mathbf{V}^{-}} \right)  & \leq &
   a_0^{+}(\bu^{+} - \bv^{+}_h, \bw^{+}_h ) +
  a_0^{-}(\bu^{-} - \bv^{-}_h, \bw^{-}_h ) +
\left\langle \blambda - \bmu_h, \bw_h^+  - \bw^-_h \right\rangle_{\mathbf{W}, \mathbf{W}'},
\\
\left\| \bw^{+}_h \right\|^2_{\mathbf{V}^{+}} + 
 \left\| \bw^{-}_h \right\|^2_{\mathbf{V}^{-}} & \leq & C\left(
\left\|\bu^{+} - \bv^{+}_h \right\|_{\mathbf{V}^+} + \left\|\bu^{-} - \bv^{-}_h \right\|_{\mathbf{V}^-}
+ \left\| \blambda - \bmu_h \right\|_{\mathbf{W}}
 \right) 
 \left( \left\| \bw^+_h \right\|_{\mathbf{V}^+} + \left\| \bw^-_h \right\|_{\mathbf{V}^-}
 \right),
\end{eqnarray*}
and, recalling that $\bw^{\pm}_h = \bu^{\pm}_h - \bv_h^{\pm}$, we obtain \eqref{estapriori}.
\end{proof}

These results show us that, under Hypothesis $(\mathbf{H})$, Problem~\eqref{discfv} admits a unique solution, which satisfies the {\it a priori} estimate~\eqref{estapriori}. We have no such estimate for the approximated multiplier $\blambda_h$.\\
Besides, the estimation of the convergence rate for the primal unknown given in \cite{HaslR} for the Poisson problem can be transposed to our elasticity problem. Proposition 3 of \cite{HaslR} - page 1480 - gives an order of convergence at least equal to $1/2$. It can be adapted to our case as follows.

\begin{proposition}
Assume Hypothesis $(\mathbf{H})$. Let $(\bu^+, \bu^-, \blambda)$ be the solution of Problem~\eqref{pbfv0} such that $\bu^{\pm} \in \mathbf{H}^{d/2+1+\varepsilon}(\Omega^{\pm}) \cap \mathbf{V}^{\pm}$ for some $\varepsilon > 0$. Assume that
\begin{eqnarray*}
\inf_{\bmu_h \in \mathbf{W}_h} \left\| \blambda - \bmu_h \right\|_{\mathbf{W}} & \leq & Ch^{\delta}
\end{eqnarray*}
for some $\delta \geq 1/2$. Then we have
\begin{eqnarray*}
\left\| \bu^{\pm} - \bu_h^{\pm} \right\|_{\mathbf{V}^{\pm}} & \leq & C\sqrt{h}.
\end{eqnarray*}
\end{proposition}

\begin{proof}
As it is shown in the proof of Proposition 3 of \cite{HaslR}, page 1481, for any $\bu^{\pm} \in  \mathbf{H}^{d/2+1+\varepsilon}(\Omega^{\pm}) \cap \mathbf{V}^{\pm}$, we can show that there exists two finite element interpolating functions $(\bv_h^{+},\bv_h^-) \in \mathbf{\mathbb{V}}_h^{0}$ such that
\begin{eqnarray}
\left\| \bu^{\pm} - \bv_h^{\pm}\right\|_{\mathbf{V}^{\pm}} & \leq & C\sqrt{h}. \label{estsqrt}
\end{eqnarray}
Such functions $\bv_h^{\pm}$ can be constructed as standard interpolating vectors of $(1-\eta_h)\bu^{\pm}$, where $\eta_h$ is a cut-off function equal to $1$ in a vicinity of the boundary $\Gamma_0$. More precisely, in a band of width $3h/2$, $\bv_h^{+}-\bv_h^-$ vanishes on all the convexes intersected by $\Gamma_0$. So we can be sure that $\bv^{+}_h - \bv^-_h$ vanishes on $\Gamma_0$, and that $(\bv_h^{+},\bv_h^-) \in \mathbf{\mathbb{V}}_h^{0}$. The announced estimate can then be deduced from~\eqref{estapriori} combined with~\eqref{estsqrt} and the hypothesis on the interpolating functions $\bmu_h$ in the statement of this proposition.
\end{proof}
This theoretical limitation of the order of convergence is very common for fictitious domain methods (without special treatments as those provided by stabilization techniques). However, the numerical tests provided in Section~\ref{secexp0} indicate that this rate of convergence is not sharp enough; Indeed, a test which would show the optimality of this estimate is hard to provide: It has been provided in \cite{HaslR} in a very specific configuration, and merely not found in \cite{Court}. In our case, we are not able to observe numerically this limitation when the outer boundary $\p \Omega$ fits the mesh, but we observe it easily when this boundary and the Dirichlet condition we impose on it are taken into account with a fictitious domain method in addition to the conditions considered on $\Gamma$. We do not comment furthermore this technical point, since it is beyond the scope of this paper.

\subsection{Stabilization technique} \label{secsuperstab} \label{secaug} \label{secinfsup}

\paragraph{The augmented Lagrangian technique} \hfill \\
In order to enforce convergence on the multiplier $\blambda$ for all fracture configurations, we develop a stabilization method \`a la {\it Barbosa-Hughes} \cite{Barbosa1, Barbosa2}; The idea is to use the Lagrangian functional $\mathcal{L}_0$ given in~\eqref{L0} with an additional term which takes into account the required value for $\blambda$. Thus we now consider the Lagrangian functional below
\begin{eqnarray}
\mathcal{L}(\bu^+, \bu^-, \blambda) & = & \mathcal{L}_0(\bu^+, \bu^-, \blambda)
- \frac{\gamma}{2} \int_{\Gamma_0} \left| \blambda + \sigma_L(\bu^+)\bn^+\right|^2 \d\Gamma_0
- \frac{\gamma}{2} \int_{\Gamma_0} \left| \blambda - \sigma_L(\bu^-)\bn^-\right|^2 \d\Gamma_0. \label{expLL}
\end{eqnarray}
Observe that this extended Lagrangian is equal to the previous one for an exact solution. The additional quadratic term tends to force $\blambda$ to reach the wanted value of the surface force $\sigma_L(\bu^-)\bn^- = -\sigma_L(\bu^+)\bn^+$ (see Remark~\ref{remarkcont}). The parameter $\gamma > 0$ indicates the importance given to this requirement. This additional term deteriorates the positivity of the Lagrangian $\mathcal{L}$, and thus the coerciveness of the formulation which stems from it. As we cannot choose a too large $\gamma$, this technique is not a penalization method. This choice is discussed in Section~\ref{secgamma}.\\

A formal calculation of the first variations gives us
\begin{eqnarray*}
\frac{\delta \mathcal{L}}{\delta \bu^+}(\bv) & = & \frac{\delta \mathcal{L}_0}{\delta \bu^+}(\bv)
-\gamma \int_{\Gamma_0}\blambda \cdot \sigma_L(\bv)\bn^+ \d \Gamma_0
-\gamma \int_{\Gamma_0} \sigma_L(\bu^+)\bn^+  \cdot \sigma_L(\bv)\bn^+ \d \Gamma_0
,\\
\frac{\delta \mathcal{L}}{\delta \bu^-}(\bv) & = & \frac{\delta \mathcal{L}_0}{\delta \bu^-}(\bv)
+\gamma \int_{\Gamma_0}\blambda \cdot \sigma_L(\bv)\bn^- \d \Gamma_0
-\gamma \int_{\Gamma_0} \sigma_L(\bu^-)\bn^-  \cdot \sigma_L(\bv)\bn^- \d \Gamma_0
,\\
\frac{\delta \mathcal{L}}{\delta \blambda}(\bmu) & = & \frac{\delta \mathcal{L}_0}{\delta \blambda}(\bv)
-2\gamma \int_{\Gamma_0} \blambda \cdot \bmu \d \Gamma_0.
\end{eqnarray*}

The stabilized discrete formulation is then:
\begin{eqnarray}
& & \text{Find $(\bu_h^+, \bu_h^-, \blambda_h)$ in $\mathbf{V}_h^+\times \mathbf{V}_h^- \times \mathbf{W}_h$ such that} \nonumber \\
& & \left\{ \begin{array} {lcl}
\mathcal{A}^{\pm}((\bu_h^+, \bu_h^-, \blambda_h);\bv_h) = l^{\pm}(\bv_h) & \quad & \forall \bv_h \in \mathbf{V}_h^{\pm}, \\
\mathcal{B}((\bu_h^+, \bu_h^-, \blambda_h);\bmu_h) = 0, & \quad & \forall \bmu_h \in \mathbf{W}_h,
\end{array} \right. 
\label{pbfv}
\end{eqnarray}
where
\begin{eqnarray*}
\mathcal{A}^{\pm}((\bu^+, \bu^-, \blambda);\bv) & = & 
\int_{\Omega^{\pm}} \sigma_{L}(\bu^{\pm}) : \varepsilon(\bv)\d \Omega^{\pm}
\pm \int_{\Gamma_0} \blambda \cdot \bv \d\Gamma_0 \nonumber\\
& & \mp \gamma \int_{\Gamma_0}\blambda \cdot \sigma_L(\bv)\bn^{\pm} \d \Gamma_0
-\gamma \int_{\Gamma_0} \sigma_L(\bu^{\pm})\bn^{\pm}  \cdot \sigma_L(\bv)\bn^{\pm} \d \Gamma_0, \nonumber \\
\mathcal{B}((\bu^+, \bu^-, \blambda);\bmu) & = & 
 \int_{\Gamma_0} \bmu \cdot \bu^+ \d \Gamma_0 -\int_{\Gamma_0} \bmu \cdot \bu^- \d \Gamma_0 \nonumber \\
& &  - \gamma \int_{\Gamma_0}\bmu \cdot \sigma_L(\bu^+)\bn^{+} \d \Gamma_0
 + \gamma \int_{\Gamma_0}\bmu \cdot \sigma_L(\bu^-)\bn^{-} \d \Gamma_0
 -2\gamma \int_{\Gamma_0} \blambda \cdot \bmu \d \Gamma_0. \nonumber
\end{eqnarray*}

\paragraph{The discrete Inf-Sup condition} \hfill \\
Let us choose $\gamma = \gamma_0 h$, where $\gamma_0 >0$ is independent of $h$. Note that Problem~\eqref{pbfv} can be rewritten as follows:
\begin{eqnarray*}
& & \text{Find } (\bu_h^+,\bu_h^-, \blambda_h) \in \mathbf{V}^+_h \times \mathbf{V}_h^- \times \mathbf{W}_h \text{ such that} \\
& & \mathcal{M}((\bu_h^+,\bu_h^-, \blambda_h) ; (\bv_h^+,\bv_h^-, \bmu_h)) =
\mathcal{F}(\bv_h^+,\bv_h^-, \bmu_h), \quad \forall 
(\bv_h^+,\bv_h^-, \bmu_h) \in \mathbf{V}^+_h \times \mathbf{V}_h^- \times \mathbf{W}_h,
\end{eqnarray*}
where
\begin{eqnarray*}
& & \mathcal{M}((\bu^+,\bu^-, \blambda) ; (\bv^+,\bv^-, \bmu)) = 
 \int_{\Omega^+} \sigma_L(\bu^+):\varepsilon(\bv^+)\d\Omega^+ +
 \int_{\Omega^-} \sigma_L(\bu^-):\varepsilon(\bv^-)\d\Omega^- \\
& & + \int_{\Gamma_0} \blambda \cdot ( \bv^+-\bv^-) \d\Gamma_0 +
\int_{\Gamma_0} \bmu \cdot ( \bu^+ -\bu^-) \d\Gamma_0 \\
& & -\gamma_0 h \int_{\Gamma_0}(\blambda + \sigma_L(\bu^+)\bn^+ ) \cdot 
(\bmu + \sigma_L(\bv^+)\bn^+ )\d \Gamma_0 
 -\gamma_0 h \int_{\Gamma_0}(\blambda - \sigma_L(\bu^-)\bn^- ) \cdot 
(\bmu - \sigma_L(\bv^-)\bn^-)\d \Gamma_0, \\
& & \mathcal{F}(\bv^+,\bv^-, \bmu)  =  \int_{\Omega^+} \bff \cdot \bv^+ \d \Omega^+
+ \int_{\Omega^-} \bff \cdot \bv^-\d \Omega^- 
+ \int_{\Gamma_T} \bv^+ \cdot p\bn^+ \d \Gamma_T
+ \int_{\Gamma_T} \bv^- \cdot p\bn^- \d \Gamma_T.
\end{eqnarray*}

In the following, we will need an assumption for the theoretical result stated in Lemma~\ref{lemmainfsup}:
\begin{eqnarray*}
(\mathbf{A}):  \qquad \qquad  h\left\| \sigma_L(\bv_h^{\pm})\bn^{\pm} \right\|^2_{\mathbf{L}^2(\Gamma_0)} & \leq & C\left\|\bv_h^{\pm} \right\|^2_{\mathbf{V}^{\pm}} \quad 
\forall \bv^{\pm}_h \in \mathbf{V}^{\pm}_h.
\end{eqnarray*}
This assumption is of same type as the ones considered in \cite{HaslR} (page 1482, the equation (27)) for the Poisson problem, or in \cite{Court} (page 84, Assumption $\mathbf{A1}$ stated in Section 4.2) for the Stokes problem, both for the theoretical study of the fictitious domain approach stabilized \`a la Barbosa-Hughes.\\

We will also use an estimate for the $\L^2$-orthogonal projection from $\mathbf{H}^{1/2}(\Gamma_0)$ on $\mathbf{W}_h$. It has been established and proved in \cite{Court} (page 85, Lemma 3):

\begin{lemma} \label{lemmaproj}
For all $\bv \in \mathbf{H}^{1/2}(\Gamma_0)$ we have
\begin{eqnarray*}
\left\| \mathcal{P}_h \bv - \bv \right\|_{\mathbf{L}^2(\Gamma_0)} & \leq & 
C \sqrt{h}\left\| \bv \right\|_{\mathbf{H}^{1/2}(\Gamma_0)},
\end{eqnarray*}
where $\mathcal{P}_h$ denotes the $L^2$-orthogonal projection from $\mathbf{H}^{1/2}(\Gamma_0)$ on $\mathbf{W}_h$.
\end{lemma}

The main purpose of this subsection is to prove the following inf-sup condition.

\begin{lemma} \label{lemmainfsup}
Assume that $(\mathbf{A})$ holds. For $\gamma_0$ small enough, there exists $C>0$ independent of $h$ such that:
\begin{eqnarray*}
\inf_{\left(\bu_h^+,\bu_h^-, \blambda_h\right) \in \mathbf{V}^+_h \times \mathbf{V}_h^- \times \mathbf{W}_h} \sup_{\left(\bv_h^+,\bv_h^-, \bmu_h\right)\in \mathbf{V}^+_h \times \mathbf{V}_h^- \times \mathbf{W}_h} \frac{\mathcal{M}((\bu_h^+,\bu_h^-, \blambda_h) ; (\bv_h^+,\bv_h^-, \bmu_h))}
{\left|\left|\left|\bu_h^+,\bu_h^-, \blambda_h\right|\right|\right| \ 
\left|\left|\left|\bv_h^+,\bv_h^-, \bmu_h\right|\right|\right|}
& \geq & C.
\end{eqnarray*}
where the norm $||| \cdot |||$ is defined as follows:
\begin{eqnarray*}
\left|\left|\left|\bu^+,\bu^-, \blambda\right|\right|\right|^2 & = & 
\left\| \bu^+\right\|^2_{\mathbf{V}^+} + \left\| \bu^-\right\|^2_{\mathbf{V}^-}
 + \frac{1}{h} \left\|\bu^+ - \bu^- \right\|^2_{\mathbf{L}^2(\Gamma_0)}
 + h\left\|\blambda \right\|_{\mathbf{L}^2(\Gamma_0)}^2 .
\end{eqnarray*}
\end{lemma}

\begin{proof}
Ideas for this proof are similar to the ones used in \cite{HaslR} (page 1482, Lemma 3) or \cite{Court} (page 85, Lemma 4), and inspired by \cite{Stenberg} (page 144, Lemma 6). First, we estimate
\begin{eqnarray*}
\mathcal{M}((\bu_h^+,\bu_h^-, \blambda_h);(\bu_h^+,\bu_h^-, -\blambda_h)) & = &
a^+_0(\bu_h^+,\bu_h^+) + a_0^-(\bu_h^-,\bu_h^-)  +2\gamma_0h \int_{\Gamma_0} \left|\blambda_h \right|^2\d \Gamma_0\\ & &
 - \gamma_0h\int_{\Gamma_0} \left|\sigma_L(\bu^+_h)\bn^+\right|^2\d\Gamma_0 
 - \gamma_0h\int_{\Gamma_0} \left|\sigma_L(\bu^-_h)\bn^-\right|^2\d\Gamma_0 \\ 
 & \geq & \alpha \left( \left\| \bu^+_h \right\|^2_{\mathbf{V}^+} + \left\| \bu^-_h \right\|^2_{\mathbf{V}^-} \right)
 +2\gamma_0h \int_{\Gamma_0} \left|\blambda_h \right|^2\d \Gamma_0
 \\
& &   -C\gamma_0 \left( \left\| \bu^+_h \right\|^2_{\mathbf{V}^+} + \left\| \bu^-_h \right\|^2_{\mathbf{V}^-} \right) \\
 & \geq & \frac{\alpha}{2} \left( \left\| \bu^+_h \right\|^2_{\mathbf{V}^+} + 
 \left\| \bu^-_h \right\|^2_{\mathbf{V}^-} \right)  
 +2\gamma_0h  \left\|\blambda_h \right\|_{\mathbf{L}^2(\Gamma_0)}^2,
\end{eqnarray*}
by using Lemma~\ref{lemmaco} and Assumption $(\mathbf{A})$, and then by choosing $\gamma_0$ small enough. 
Next, let us consider $\overline{\bmu}_h = \frac{1}{h}\mathcal{P}_h(\bu_h^+ - \bu_h^-)$. In view of Assumption $(\mathbf{A})$, the Cauchy-Schwarz and Young inequalities, we have
\begin{eqnarray*}
\mathcal{M}((\bu_h^+,\bu_h^-, \blambda_h);(0,0,\overline{\bmu}_h)) & = &
\frac{1}{h}\left\| \mathcal{P}_h(\bu_h^+ - \bu_h^-) \right\|^2_{\mathbf{L}^2(\Gamma_0)} 
 - \gamma_0 h \int_{\Gamma_0} \sigma_L(\bu^+)\bn^+ \cdot \overline{\bmu}_h \d \Gamma_0 \\
& &  + \gamma_0 h \int_{\Gamma_0} \sigma_L(\bu^-)\bn^- \cdot \overline{\bmu}_h \d \Gamma_0 -2\gamma_0 h \int_{\Gamma_0} \blambda_h \cdot \overline{\bmu}_h \d \Gamma_0 \\
& \geq & \frac{1}{h}\left\| \mathcal{P}_h(\bu_h^+ - \bu_h^-) \right\|^2_{\mathbf{L}^2(\Gamma_0)}
-2\gamma_0 \sqrt{h} \left\| \blambda_h \right\|_{\mathbf{L}^2(\Gamma_0)} \sqrt{h}
\left\| \overline{\bmu}_h \right\|_{\mathbf{L}^2(\Gamma_0)} \\
& & - \gamma_0 \left( \sqrt{h} \left\| \sigma_L(\bu^+)\bn^+ \right\|_{\mathbf{L}^2(\Gamma_0)} + 
\sqrt{h} \left\| \sigma_L(\bu^-)\bn^- \right\|_{\mathbf{L}^2(\Gamma_0)}
 \right)\sqrt{h}
 \left\| \overline{\bmu}_h \right\|_{\mathbf{L}^2(\Gamma_0)} \\
 & \geq & \frac{1}{h}\left\| \mathcal{P}_h(\bu_h^+ - \bu_h^-) \right\|^2_{\mathbf{L}^2(\Gamma_0)}
 - \gamma_0 h \left\| \blambda_h \right\|^2_{\mathbf{L}^2(\Gamma_0)}
 \\
 & &  -\frac{3\gamma_0}{2h}\left\| \mathcal{P}_h(\bu_h^+ - \bu_h^-) \right\|^2_{\mathbf{L}^2(\Gamma_0)} -C\gamma_0 \left( \left\| \bu^+_h \right\|^2_{\mathbf{V}^+} + \left\| \bu^-_h \right\|^2_{\mathbf{V}^-} \right)\\
  & \geq & \frac{1}{2h}\left\| \mathcal{P}_h(\bu_h^+ - \bu_h^-) \right\|^2_{\mathbf{L}^2(\Gamma_0)}
  - \gamma_0 h \left\| \blambda_h \right\|^2_{\mathbf{L}^2(\Gamma_0)} 
  -C\gamma_0 \left( \left\| \bu^+_h \right\|^2_{\mathbf{V}^+} + \left\| \bu^-_h \right\|^2_{\mathbf{V}^-} \right),
\end{eqnarray*}
for $\gamma_0$ small enough. Now, choosing $\left(\bv_h^+,\bv_h^-, \bmu_h\right) = \left(\bu_h^+,\bu_h^-, -\blambda_h + \eta \overline{\bmu}_h\right)$ - where $\eta > 0$ is supposed to be small enough - the previous inequalities above yield
\begin{eqnarray}
\mathcal{M}((\bu_h^+,\bu_h^-, \blambda_h);(\bv_h^+,\bv_h^-, \bmu_h)) & \geq &
 \frac{\eta}{2h}\left\| \mathcal{P}_h(\bu_h^+ - \bu_h^-) \right\|^2_{\mathbf{L}^2(\Gamma_0)}
 +(2\gamma_0h - \eta \gamma_0 h)\left\| \blambda_h \right\|^2_{\mathbf{L}^2(\Gamma_0)} \nonumber \\
 & & +\left( \frac{\alpha}{2} - C\eta\gamma_0 \right)\left( \left\| \bu^+_h \right\|^2_{\mathbf{V}^+} + \left\| \bu^-_h \right\|^2_{\mathbf{V}^-} \right). \label{ineqproj0}
\end{eqnarray}
The projection term can be handled by using Lemma~\eqref{lemmaproj} as follows:
\begin{eqnarray}
\left\| \mathcal{P}_h(\bu_h^+ - \bu_h^-) \right\|^2_{\mathbf{L}^2(\Gamma_0)} & \geq &  
\left\| \mathcal{P}_h(\bu_h^+ - \bu_h^-) \right\|^2_{\mathbf{L}^2(\Gamma_0)} \nonumber \\
& = & \left\| \bu_h^+ - \bu_h^- \right\|^2_{\mathbf{L}^2(\Gamma_0)} - 
\left\| \mathcal{P}_h(\bu_h^+ - \bu_h^-) - (\bu_h^+ - \bu_h^-) \right\|^2_{\mathbf{L}^2(\Gamma_0)} \nonumber \\
& \geq &  \left\| \bu_h^+ - \bu_h^- \right\|^2_{\mathbf{L}^2(\Gamma_0)} - Ch
 \left\| \bu_h^+ - \bu_h^- \right\|^2_{\mathbf{H}^{1/2}(\Gamma_0)} \nonumber \\
 & \geq & \left\| \bu_h^+ - \bu_h^- \right\|^2_{\mathbf{L}^2(\Gamma_0)} - 2 Ch
 \left( \left\| \bu_h^+ \right\|^2_{\mathbf{V}^{+}} + \left\| \bu_h^- \right\|^2_{\mathbf{V}^{-}} \right). \label{ineqproj1}
\end{eqnarray}
Then, in view of~\eqref{ineqproj0} and~\eqref{ineqproj1}, we have
\begin{eqnarray*}
\mathcal{M}((\bu_h^+,\bu_h^-, \blambda_h);(\bv_h^+,\bv_h^-, \bmu_h)) & \geq &
 \left( \frac{\alpha}{2} - C\eta\gamma_0 - 2\eta C \right)\left( \left\| \bu^+_h \right\|^2_{\mathbf{V}^+} + \left\| \bu^-_h \right\|^2_{\mathbf{V}^-} \right)
 \nonumber \\
 & & +\frac{\eta}{h} \left\| \bu_h^+ - \bu_h^- \right\|^2_{\mathbf{L}^2(\Gamma_0)}
 + \gamma_0(2 - \eta )h\left\| \blambda_h \right\|^2_{\mathbf{L}^2(\Gamma_0)},
\end{eqnarray*}
and thus, recalling that $\left(\bv_h^+,\bv_h^-, \bmu_h\right) = \left(\bu_h^+,\bu_h^-, -\blambda_h + \eta \overline{\bmu}_h\right)$, and choosing $\gamma_0 >0$ and $\eta>0$ small enough we obtain
\begin{eqnarray}
\mathcal{M}((\bu_h^+,\bu_h^-, \blambda_h);(\bv_h^+,\bv_h^-, \bmu_h)) & \geq &
C\left|\left|\left| \bu_h^+,\bu_h^-, \blambda_h \right|\right|\right|^2, \label{ineqM1}
\end{eqnarray}
where $C >0$ is independent of $h$. On the other hand, we have
\begin{eqnarray}
\left|\left|\left| \bv_h^+,\bv_h^-, \bmu_h \right|\right|\right| & \leq & 
M\left|\left|\left| \bu_h^+,\bu_h^-, \blambda_h \right|\right|\right|
\label{ineqM2}
\end{eqnarray}
where $M>0$ is independent of $h$ too. Indeed,
\begin{eqnarray*}
\left|\left|\left| \bv_h^+,\bv_h^-, \bmu_h \right|\right|\right| & \leq & \left|\left|\left| \bu_h^+,\bu_h^-, \blambda_h \right|\right|\right| +
\eta \left|\left|\left| 0,0, \bmu_h \right|\right|\right| \\
& = & \left|\left|\left| \bu_h^+,\bu_h^-, \blambda_h \right|\right|\right| + \frac{\eta}{\sqrt{h}}
\left\| \mathcal{P}_h(\bu_h^+ - \bu_h^-) \right\|_{\mathbf{L}^2(\Gamma_0)} \\
& \leq & \left|\left|\left| \bu_h^+,\bu_h^-, \blambda_h \right|\right|\right| + \frac{\eta}{\sqrt{h}}
\left\| \bu_h^+ - \bu_h^- \right\|_{\mathbf{L}^2(\Gamma_0)} \\
& \leq & (1+\eta)\left|\left|\left| \bu_h^+,\bu_h^-, \blambda_h \right|\right|\right|.
\end{eqnarray*}
Thus, combining~\eqref{ineqM1} and~\eqref{ineqM2} gives the inequality
\begin{eqnarray*}
\frac{\mathcal{M}((\bu_h^+,\bu_h^-, \blambda_h);(\bv_h^+,\bv_h^-, \bmu_h))}
{\left|\left|\left| \bv_h^+,\bv_h^-, \bmu_h \right|\right|\right|} & \geq & 
\frac{c}{M}\left|\left|\left| \bu_h^+,\bu_h^-, \blambda_h \right|\right|\right|
\end{eqnarray*}
which leads to the desired result.
\end{proof}

Note that the bilinear form $\mathcal{M}$ is bounded for the triple norm on $\mathbf{V}^+ \times \mathbf{V}^- \times \mathbf{W}$ uniformly with respect to the mesh size~$h$. Then Lemma~\ref{lemmainfsup} leads to the following error estimate with a C\'ea type lemma (see \cite{Ern} for instance):
\begin{eqnarray*}
\left|\left|\left| \bu^{+} - \bu^{+}_h, \bu^{-} - \bu^{-}_h , \blambda - \blambda_h \right|\right|\right| & \leq &
C \inf_{(\bv_h^+,\bv_h^-, \bmu_h)\in \mathbf{V}^+_h \times \mathbf{V}_h^- \times \mathbf{W}_h}
\left|\left|\left| \bu^{+} - \bv^{+}_h, \bu^{-} - \bv^{-}_h , \blambda - \bmu_h \right|\right|\right|. 
\end{eqnarray*}
From that we can invoke the extension theorem for the Sobolev spaces, the standard estimates for the nodal finite element interpolation operators and the trace inequality
\begin{eqnarray*}
\| \bv \|_{\mathbf{L}^2(\Gamma_0)} & \leq & C\left(
h\| \bv \|_{\mathbf{L}^2(T)}
+ \frac{1}{h} \| \bv \|_{\mathbf{L}^2(T)}
\right)
\end{eqnarray*}
on any convex $T \in \mathcal{T}_h$ (see Appendix A of \cite{HaslR}, page 1496, for more details). As a consequence, it yields the following abstract error estimate
\begin{eqnarray*}
&  & \max(\left\|\bu - \bu_h^+\right\|_{\mathbf{V}^+}, \left\|\bu - \bu_h^+\right\|_{\mathbf{V}^+}, \left\|\blambda - \blambda_h^+\right\|_{\mathbf{L}^2(\Gamma_0)} ) \leq 
\left|\left|\left| \bu^{+} - \bu^{+}_h, \bu^{-} - \bu^{-}_h , \blambda - \blambda_h \right|\right|\right| \\
& & \leq C\left( h^{k_u}\left\|\bu^+\right\|_{\mathbf{H}^{k_u+1}(\Omega^+)} +  
h^{k_u}\left\|\bu^-\right\|_{\mathbf{H}^{k_{\bu}+1}(\Omega^-)} 
+ h^{k_{\blambda}+1} \left\|\blambda\right\|_{\mathbf{H}^{k_{\blambda}+1}(\Gamma_0)}
\right),
\end{eqnarray*}
where $k_{\bu}$ and $k_{\blambda}$ are the respective degrees of standard finite elements used for the displacements $\bu_h^+$, $\bu_h^-$ and the multiplier $\blambda$.

\section{Practical remarks on the implementation} \label{secimplem}

\subsection{Matrix formulations} \label{secmatrixform}
In matrix notation, the formulation~\eqref{discfv} gives
\begin{eqnarray*}
\left( \begin{array} {cc|c}
A_0^+ & 0 & {B_0^+}^T \\
0 & A_0^- & -{B_0^-}^T \\
\hline 
B_0^+ & -B_0^+ & 0
\end{array}
\right)
\left( \begin{array} {c}
\mathbf{U}^+ \\
\mathbf{U}^- \\
\hline 
\bLambda
\end{array}
\right)
& = &
\left( \begin{array} {c}
\mathbf{F}^+ \\
\mathbf{F}^- \\
\hline 
0
\end{array}
\right),
\end{eqnarray*}
where $\mathbf{U}^+$, $\mathbf{U}^-$ and $\bLambda$ are the degrees of freedom of $\bu^+_h$, $\bu^-_h$ and $\blambda_h$ respectively. The matrices $A_0^+$, $A_0^-$, $B_0^+$ and $B_0^-$ are the discretization of~\eqref{fbil1}-\eqref{fbil2}. If we denote the basis functions of the spaces $\mathbf{V}^+_h$, $\mathbf{V}^-_h$ and $\mathbf{W}_h$ by $\left\{ \bvarphi^+_i \right\}$, $\left\{ \bvarphi^-_i \right\}$ and $\left\{ \bpsi_i \right\}$ respectively, we have:
\begin{eqnarray*}
\left( A_0^+ \right)_{ij} = \int_{\Omega^+} \sigma_{L}(\bvarphi^+_i) : \varepsilon(\bvarphi^+_j)\d \Omega^+, & \quad & 
\left( A_0^- \right)_{ij} = \int_{\Omega^-} \sigma_{L}(\bvarphi^-_i) : \varepsilon(\bvarphi^-_j)\d \Omega^-, \\
\left( B_0^+ \right)_{ij} = \int_{\Gamma_0} \bpsi_i \cdot \bvarphi^+_j \d \Gamma_0, & \quad &
\left( B_0^- \right)_{ij} = \int_{\Gamma_0} \bpsi_i \cdot \bvarphi^-_j \d \Gamma_0.
\end{eqnarray*}
Vectors $\mathbf{F}^{\pm}$ are the discretization of $\eqref{Lrhs}$.\\
In matrix writing, the stabilized formulation~\eqref{pbfv} corresponds to
\begin{eqnarray}
\left( \begin{array} {cc|c}
A^+ & 0 & {B^+}^T \\
0 & A^- & -{B^-}^T \\
\hline 
B^+ & -B^+ & -C
\end{array}
\right)
\left( \begin{array} {c}
\mathbf{U}^+ \\
\mathbf{U}^- \\
\hline 
\bLambda
\end{array}
\right)
& = &
\left( \begin{array} {c}
\mathbf{F}^+ \\
\mathbf{F}^- \\
\hline 
0
\end{array}
\right). \label{systdec}
\end{eqnarray}
The matrices $A^{\pm}$, $B^{\pm}$ and $C$ introduced here are the respective discretizations of the following bilinear forms:
\begin{eqnarray}
a^{\pm}(\bu^{\pm},\bv) & = & \int_{\Omega^{\pm}} \sigma_{L}(\bu^{\pm}) : \varepsilon(\bv)\d \Omega^{\pm}
-\gamma \int_{\Gamma_0} \sigma_L(\bu^{\pm})\bn^{\pm}  \cdot \sigma_L(\bv)\bn^{\pm} \d \Gamma_0, \label{fbstabext1}
\\
b(\bv, \blambda) & = &  \int_{\Gamma_0} \blambda \cdot \bv \d\Gamma_0 
\mp \gamma \int_{\Gamma_0}\blambda \cdot \sigma_L(\bv)\bn^{\pm} \d \Gamma_0,
\\
c(\blambda, \bmu) & = & 2\gamma \int_{\Gamma_0} \blambda \cdot \bmu \d \Gamma_0.
\label{fbstabext3}
\end{eqnarray}

\subsection{Solving the whole system} \label{secuzawa}
Let us recall the extended stabilized formulation introduced in Section~\ref{secaug}:
\begin{eqnarray}
& & \text{Find $(\bu^+, \bu^-, \blambda)$ in $\mathbf{V}^+\times \mathbf{V}^- \times \mathbf{W}$ such that} \nonumber \\
& & \left\{ \begin{array} {lcl}
a^{+}(\bu^{+},\bv^+) = - b(\blambda, \bv^+) + l^{+}(\bv^+) & \quad & \forall \bv^+ \in \mathbf{V}^{+}, \\
a^{-}(\bu^{-},\bv^-) = b(\blambda, \bv^-) + l^{-}(\bv^-) & \quad & \forall \bv^- \in \mathbf{V}^{-}, \\
b(\bu^+,\bmu) - b(\bu^-,\bmu) - c(\blambda, \bmu) = 0, & \quad & \forall \bmu \in \mathbf{W},
\end{array} \right. \label{fvextx}
\end{eqnarray}
where the linear forms $l^{\pm}$ are given by~\eqref{Lrhs} and the bilinear forms $a^{\pm}$, $b$ and $c$ are given by~\eqref{fbstabext1}--\eqref{fbstabext3}. Recall also that in the discrete formulation~\eqref{discfv} we considered for the bilinear form~\eqref{fbil2} the inner product in $\L^2(\Gamma_0)$ instead of the duality pairing $\displaystyle \langle \ \cdot \ , \ \cdot \ \rangle_{\mathbf{W},\mathbf{W}'}$, leading to consider that $\mathbf{W} \equiv \mathbf{W}' = \mathbf{L}^2(\Gamma_0)$.\\
In view of the formulation~\eqref{fvextx}, we deduce that the mapping $\blambda \mapsto \bu^{\pm}(\blambda)$ is affine, and $\bu^{\pm}(\blambda + t\overline{\bmu}) = \bu^{\pm}(\blambda) + t \bomega^{\pm}$, where $\bomega^{\pm} \in \mathbf{V}^{\pm}$ are such that
\begin{eqnarray}
a^{+}(\bomega^{+},\bv^+) = - b(\overline{\bmu}, \bv^+) & \quad & \forall \bv^+ \in \mathbf{V}^{+}, \label{sens1}\\
a^{-}(\bomega^{-},\bv^-) = b(\overline{\bmu}, \bv^-) & \quad & \forall \bv^- \in \mathbf{V}^{-},
\label{sens2}\\
c(\overline{\bmu},\bmu) = 0 & \quad & \forall \bmu \in \mathbf{W}.
\nonumber
\end{eqnarray}
If we set $\bv^{\pm} = \bu^{\pm}$ in the two first equations of~\eqref{fvextx}, we get after summation
\begin{eqnarray*}
a^+(\bu^+, \bu^+) + a^-(\bu^-,\bu^-) & = & -b(\blambda,\bu^+) + b(\blambda,\bu^-) + l^+(\bu^+)  + l^-(\bu^-).
\end{eqnarray*} 
Substituting this equality in the expressions of the Lagrangian~\eqref{expLL} based on the expression~\eqref{L0}, we obtain the following dual functional
\begin{eqnarray*}
J^{\ast}(\blambda) & = & \mathcal{L}(\bu^+(\blambda), \bu^-(\blambda), \blambda) \\
& = &  -\frac{1}{2}a^+(\bu^+,\bu^+) - \frac{1}{2}a^-(\bu^-, \bu^-)
-\frac{1}{2}c(\blambda, \blambda). \label{dualJ}
\end{eqnarray*}
The dual problem is therefore the maximization problem below:
\begin{eqnarray}
 \text{Find $\blambda^{\ast} \in \mathbf{W}$ such that:}
& & J^{\ast}(\blambda^{\ast}) \geq J^{\ast}(\bmu), \quad \forall \bmu \in \mathbf{W}.
\label{pboptJ}
\end{eqnarray}
For a given direction $\overline{\bmu}$, the directional derivative for $J^{\ast}$ is given by
\begin{eqnarray*}
DJ^{\ast}(\blambda).\overline{\bmu} & = & 
-a^+(\bu^+,\bomega^+) - a^-(\bu^-,\bomega^-). 
\end{eqnarray*}
In view of the sensitivity system~\eqref{sens1}-\eqref{sens2} with $\bv^{\pm} = \bu^{\pm}$, and in view of Remark~\ref{remarkcont}, this derivative can be expressed as
\begin{eqnarray*}
DJ^{\ast}(\blambda).\overline{\bmu} & = & b^+(\bu^+,\overline{\bmu}) - b^-(\bu^-,\overline{\bmu}) = \left\langle \overline{\bmu} , (\bu^+ - \bu^-) \right\rangle_{\mathbf{W},\mathbf{W}'}.
\end{eqnarray*}
Thus we deduce that the gradient of $J^{\ast}$ - for the $\mathbf{L}^2(\Gamma_0)$-inner product - is $D J^{\ast}(\blambda) = \left[ \bu \right]$. Moreover, if $\overline{\bmu}$ is a search direction for $J^{\ast}$, we obtain
\begin{eqnarray*}
DJ^{\ast}(\blambda + t \overline{\bmu}).\overline{\bmu} & = & \left\langle \overline{\bmu} , \left[ \bu \right] + t \left[ \bomega \right] \right\rangle_{\mathbf{W},\mathbf{W}'}.
\end{eqnarray*} 
The optimal step-size $t^{\ast}$ in the search direction $\overline{\bmu}$ is therefore
\begin{eqnarray*}
t^{\ast} & = & - \frac{\left\langle \overline{\bmu} , \left[ \bu \right] \right\rangle_{\mathbf{W},\mathbf{W}'}}{\left\langle \overline{\bmu} , \left[ \bomega \right] \right\rangle_{\mathbf{W},\mathbf{W}'}}.
\end{eqnarray*}
Since $J^{\ast}$ is quadratic and (strongly) concave, a good algorithm for solving~\eqref{pboptJ} is a conjugate gradient algorithm. The Fletcher-Reeves conjugate gradient algorithm for the maximization of the dual functional~\eqref{dualJ} is summarized in Algorithm~\ref{UCG}. 

\begin{algorithm}[htpb]
\begin{description}
\item[Initialization: $k=0$.] For $\blambda_0$ given, compute $\bu_0$ such that
\begin{eqnarray*}
a^{+}(\bu^{+}_0,\bv^+) =- b(\bv^+,\blambda_0) + l^{+}(\bv^+) & \quad & \forall \bv^+ \in \mathbf{V}^{+},  \\[.25pc]
a^{-}(\bu^{-}_0,\bv^-) = b(\bv^-,\blambda_0) + l^{-}(\bv^-) & \quad &  \forall \bv^- \in \mathbf{V}^{-}.   
\end{eqnarray*}
Set $\bg_0 = [\bu_0]$ on $\Gamma_0$ (gradient), and $\overline{\bmu}_0 = \bg_0$ (search direction).
\item[Iteration $k\geq 0$.] Assume that $\bu_k$, $\bg_k$, $\overline{\bmu}_k$ are known.
\begin{description}
   \item[Sensitivity:] Compute $\bomega^{\pm}_k$ such that
\begin{eqnarray*}
 a^{+}(\bomega^{+}_k,\bv^+) = - b(\bv^+,\overline{\bmu}_k) & \quad & \forall \bv^+ \in \mathbf{V}^{+},  \\[.25pc]
a^{-}(\bomega^{-}_k,\bv^-) = b(\bv^-,\overline{\bmu}_k) & \quad & \forall \bv^- \in \mathbf{V}^{-} 
\end{eqnarray*}
\item[Step size:] $t_k=-(\overline{\bmu}_k,[\bu_k])_{\mathbf{L}^2(\Gamma_0)}/(\overline{\bmu}_k,[\bomega_k])_{\mathbf{L}^2(\Gamma_0)}$
\item[Update:] $\bu^{\pm}_{k+1}=\bu^{\pm}_{k}+t_k\bomega^{\pm}_k$
\item[Gradient:] $\bg_{k+1}=\bg_k+t_k[\bomega_k]$
\item[Direction:] $\overline{\bmu}_{k+1}=\bg_{k+1}+\beta_k \overline{\bmu}_k$, where 
  $\beta_k=(\bg_{k+1},\bg_{k+1})_{\mathbf{L}^2(\Gamma_0)}/(\bg_{k},\bg_{k})_{\mathbf{L}^2(\Gamma_0)}$
\end{description}
\end{description}
\caption{Uzawa conjugate gradient algorithm}\label{UCG}
\end{algorithm}
\FloatBarrier

We iterate until the norm of the gradient becomes sufficiently small, namely
\begin{eqnarray*}
\left( \bg_k ; \bg_k \right)_{\mathbf{L}^2(\Gamma_0)} & < & \varepsilon
\left( \bg_0 ; \bg_0 \right)_{\mathbf{L}^2(\Gamma_0)},
\end{eqnarray*}
for some tolerance $\varepsilon >0$. The main characteristic of the dual algorithm is that, at each iteration, we solve one uncoupled problem.

\subsection{Libraries used in the implementation}
Numerical implementation is done with the code developed under the {\sc Getfem++} Library (see \cite{Getfem}). As mentioned previously, our program is based on the approach initially introduced for the Poisson problem (see \cite{HaslR}). In dimension 2, solving the global system can be made by using the library SuperLU \cite{SuperLU}; In dimension 3 the Uzawa iterative algorithm given in Section~\ref{secuzawa} and adapted for system of type~\eqref{systdec} is carried out while using the {\sc Gmm++} Library (included into the {\sc Getfem++} library).\\
For this kind of boundary value problems, in {\sc Getfem++} several difficulties have been tackled. These include:
\begin{itemize}
\item[--] Defining basis functions of ${\mathbf{W}}_h$ from traces on $\Gamma_0$ of the basis functions of $\tilde{{\mathbf{W}}}_h$. In particular, the independence of these functions is not automatically ensured, and redundant degrees of freedom have to be eliminated to avoid handling non-invertible systems.
\item[--] Localizing the interfaces corresponding to the crack and its extension. For that purpose, a level-set function method is implemented in the library.

\item[--] Accurately computing the integrals over elements at these interfaces during the assembly procedure, this involves calling the {\sc Qhull} Library (see \cite{qhull}).
\end{itemize}

\subsection{Tools for defining and updating the geometric configuration}
In practice, cracks are numerically described with level-set functions. A first one $ls_1$ describes the whole boundary $\Gamma = \Gamma_T \cup \Gamma_0$ with the equation $ls_1 = 0$. Then, we cut $\Gamma$ into $\Gamma_T$ with two auxiliary functions $ls_2$ and $ls_3$: The degrees of freedom concerned by $\Gamma_T$ will have to obey simultaneously the inequalities $ls_2 < 0$ and $ls_3 < 0$, while those concerned by $\Gamma_0 = \Gamma \setminus \Gamma_T$ will be the one of $\Gamma$ which satisfy $ls_2 \geq 0$ or $ls_3 \geq 0$. Using level-set functions, the position and shape of cracks can be modified when running inversions and cracks can be extended to study their propagation.\\
\hfill \\
To gain computation time when the crack geometry is updated several processes are followed.
Let us denote by $\left\{ \bvarphi_i \right\}$ the standard uncut basis functions of some discrete finite element space $\tilde{\mathbf{V}}_h \subset \mathbf{H}^1(\Omega)$ introduced in Section~\ref{secfict}. This space is made of basis functions lying in the whole domain $\Omega$, which are uncut by the boundary $\Gamma$. We denote by
\begin{eqnarray*}
\left(A_0\right)_{ij} & = & \int_{\Omega} \sigma_L(\bvarphi_i) : \varepsilon(\bvarphi_j) \d \Omega
\end{eqnarray*}
the symmetric stiffness matrix in $\Omega$, which is independent of the crack and its extension. The integration method for computing each of these integrals is standard too. This matrix can be stored as the same time as some of its decompositions which enables us to invert it quickly and efficiently. From that, the basis functions $\left\{ \bvarphi_i^{\pm} \right\}$ of the spaces $\mathbf{V}^{\pm}_h$ (see Section~\ref{secfict} for more details) can be deduced with the help of {\it reduction} and {\it extension} matrices that we denote by $R^{\pm}$ and $E^{\pm}$ respectively: The reduction matrices $R^{\pm}$ enables to select the indexes of family $\left\{ \bvarphi_i \right\}$ used to define family $\left\{ \bvarphi_i^{\pm} \right\}$; The matrices $E^{\pm}$ have the inverse role. Note that these matrices - $R^{\pm}$ and $E^{\pm}$ - are sparse and binary, and thus can be easily manipulated. Moreover, they satisfy the following useful properties:
\begin{eqnarray*}
R^+E^+ = \I, \quad R^+ = {E^+}^T, \quad R^-E^- = \I, \quad R^- = {E^-}^T.
\end{eqnarray*}
At this stage, we do not have to get the functions $\bvarphi_i^{\pm}$ multiplied by Heaviside functions (mentioned in Section~\ref{secfict}) as they are only used in the integrations over $\Gamma$. Now we define
\begin{eqnarray*}
\tilde{A}_0^+ = R^+A_0, & \quad & \tilde{A}_0^- = R^-A_0.
\end{eqnarray*}
Recovery of stiffness matrices $A_0^{\pm}$ in the domains $\Omega^{\pm}$ (see Section~\ref{secmatrixform}) can be done from a local re-assembly of the integration terms for the matrices $\tilde{A}_0^{\pm}$, by including the Heaviside functions for the terms whose basis functions are intersected by the level-set representing $\Gamma$. In practice we have easily access to the indexes which correspond to the local re-assemblies.

\section{Numerical tests of convergence} \label{secnumcurve}

\subsection{Numerical experiments without stabilization in 2D} \label{secexp0} \label{subsubtest0}

\paragraph{Rates of convergence for displacement in 2D tests} \hfill \\
Given a square $\Omega = [0 ; 1] \times [0 ; 1]$, we consider for $\Gamma$ a straight inclined line splitting $\Omega$ into two parts. We choose for $\Gamma$ the set represented by the level-set function
\begin{eqnarray*}
ls_1(x,y) & = & y -2(x-x_0).
\end{eqnarray*}
The crack $\Gamma_T$ is then chosen to be delimited by the secondary level-set functions:
\begin{eqnarray*}
ls_2(x,y) = x_A - x, & \quad & ls_2(x,y) = x - x_B,
\end{eqnarray*}
as the points of $\Gamma_0$ satisfying $ls_2 < 0$ and $ls_3 < 0$. For instance we can take $x_0 = 0.317$, $x_A = 0.47$ and $x_B = 0.52$. Let us recall that we have defined $\Gamma_0 = \Gamma \setminus \Gamma_T$. For the displacement, we consider the following exact solution
\begin{eqnarray*}
\bu_{ex}(x,y) & = & \left( \begin{array} {ll}
(x+y)\cos(x) \\
(x-y)\sin(y)
\end{array}
\right) \text{ if $ls_1(x,y) > 0$},\\
\bu_{ex}(x,y) & = & \left( \begin{array} {ll}
(x+y)\cos(x) - D_1(x,y)\\
(x-y)\sin(y) - D_2(x,y)
\end{array}
\right) \text{ if $ls_1(x,y) \leq 0$}.
\end{eqnarray*}
The jump $D = (D_1,D_2)^T$ across $\Gamma_0$ is no longer equal to zero. However, it is chosen to be constant, in order to avoid introducing jumps in normal derivatives across this boundary (see Remark~\ref{remarkcont}).\\
In the figures below we show the results of computation of the relative errors on the displacement, for different choices of the finite element spaces $\tilde{\mathbf{V}}_h$ and $\mathbf{W}_h$. For instance, the couple P2/P0 indicates that we have considered a standard continuous P2 element for both partial displacements $\bu^+_h$ and $\bu^-_h$, and a discontinuous P0 element for the multiplier $\blambda$. We deduce an approximation of the order of convergence for the global displacement reconstructed afterwards from the partial ones $\bu^+_h$ and $\bu^-_h$.

\begin{figure}[!h]
\begin{tabular} {c|c}
\includegraphics[trim = 1cm 7cm 2cm 7cm, clip, scale=0.45]{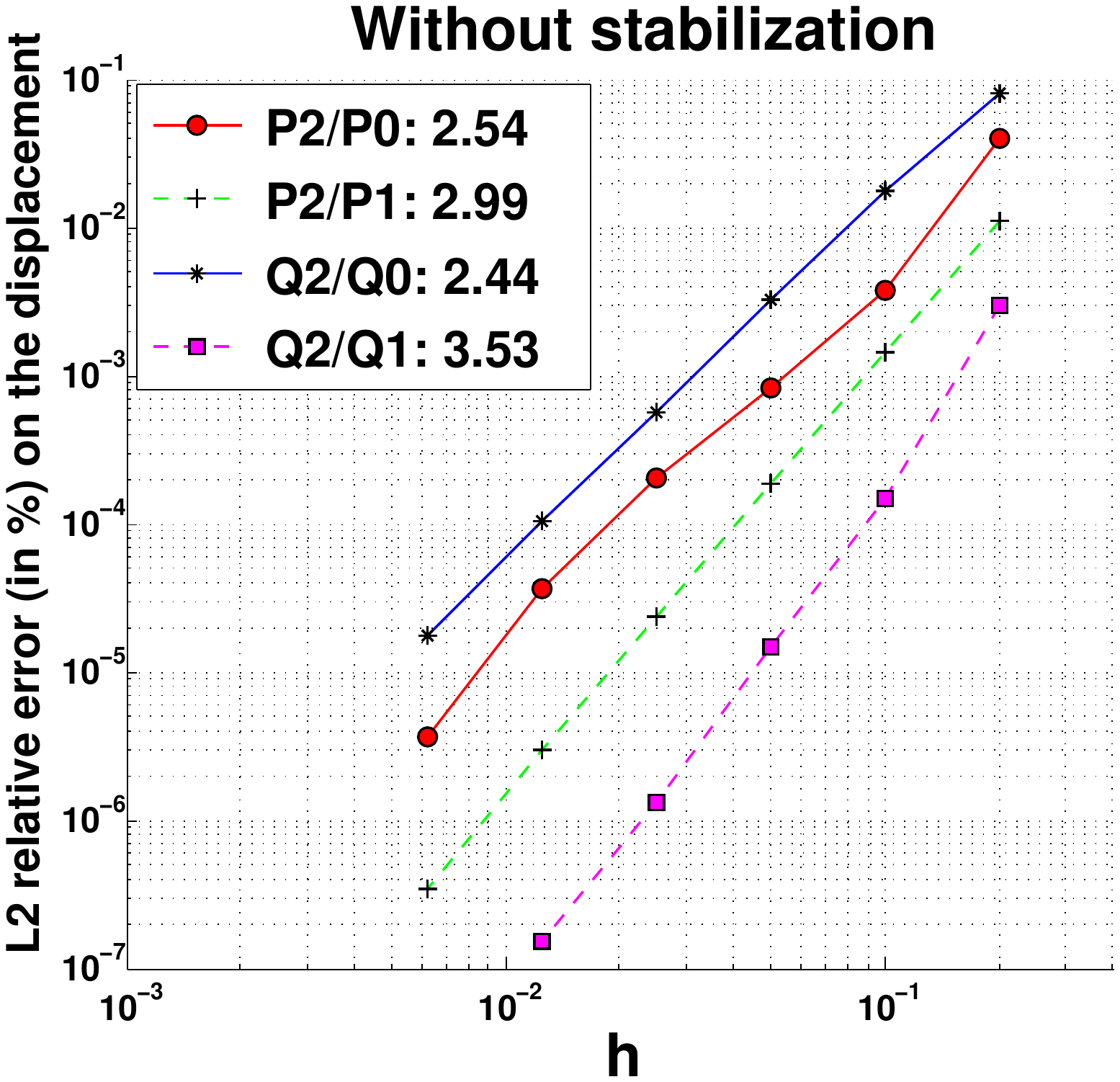}
&
\includegraphics[trim = 1cm 7cm 2cm 7cm, clip, scale=0.45]{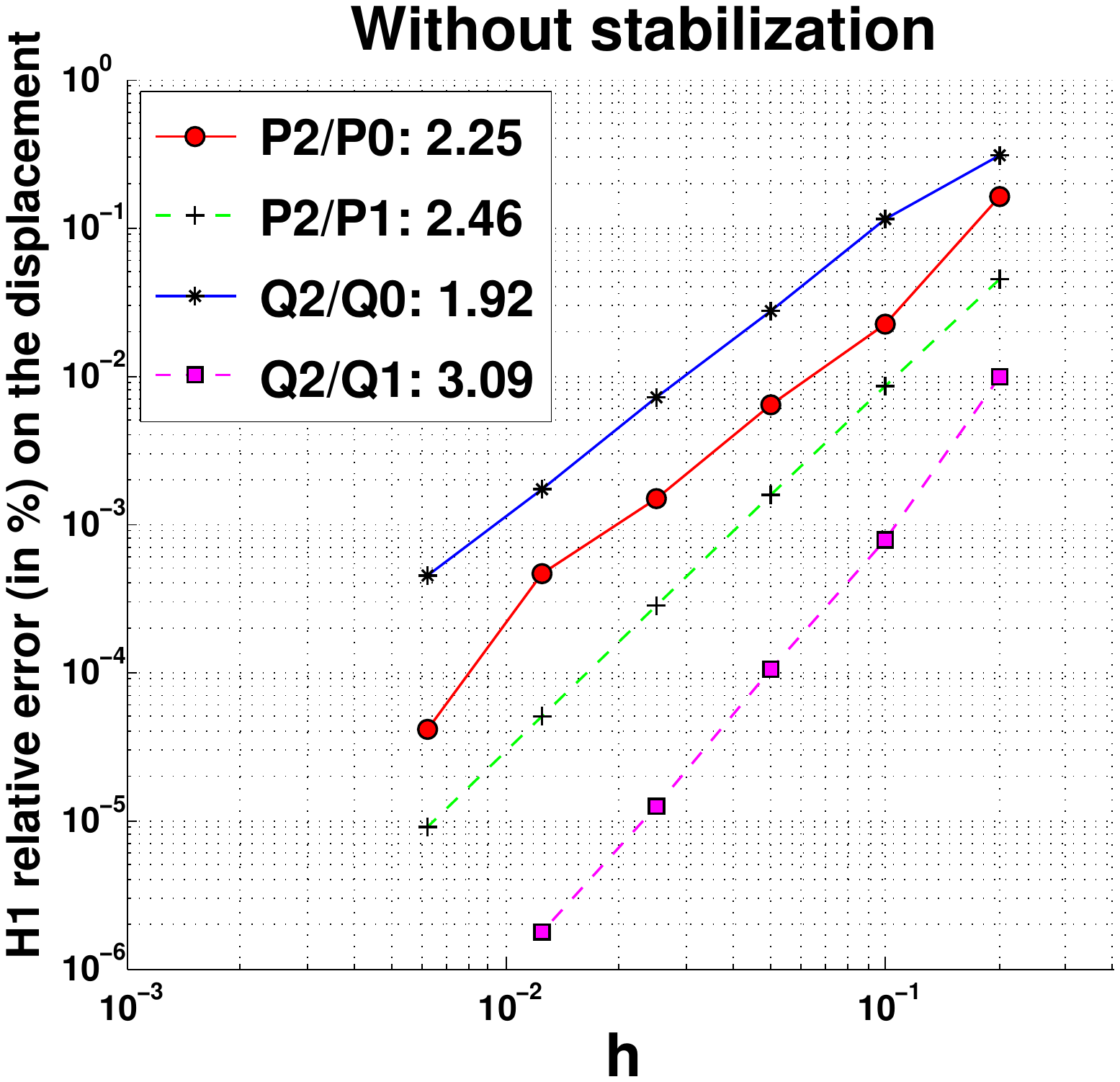}
\\
\includegraphics[trim = 1cm 7cm 2cm 7cm, clip, scale=0.45]{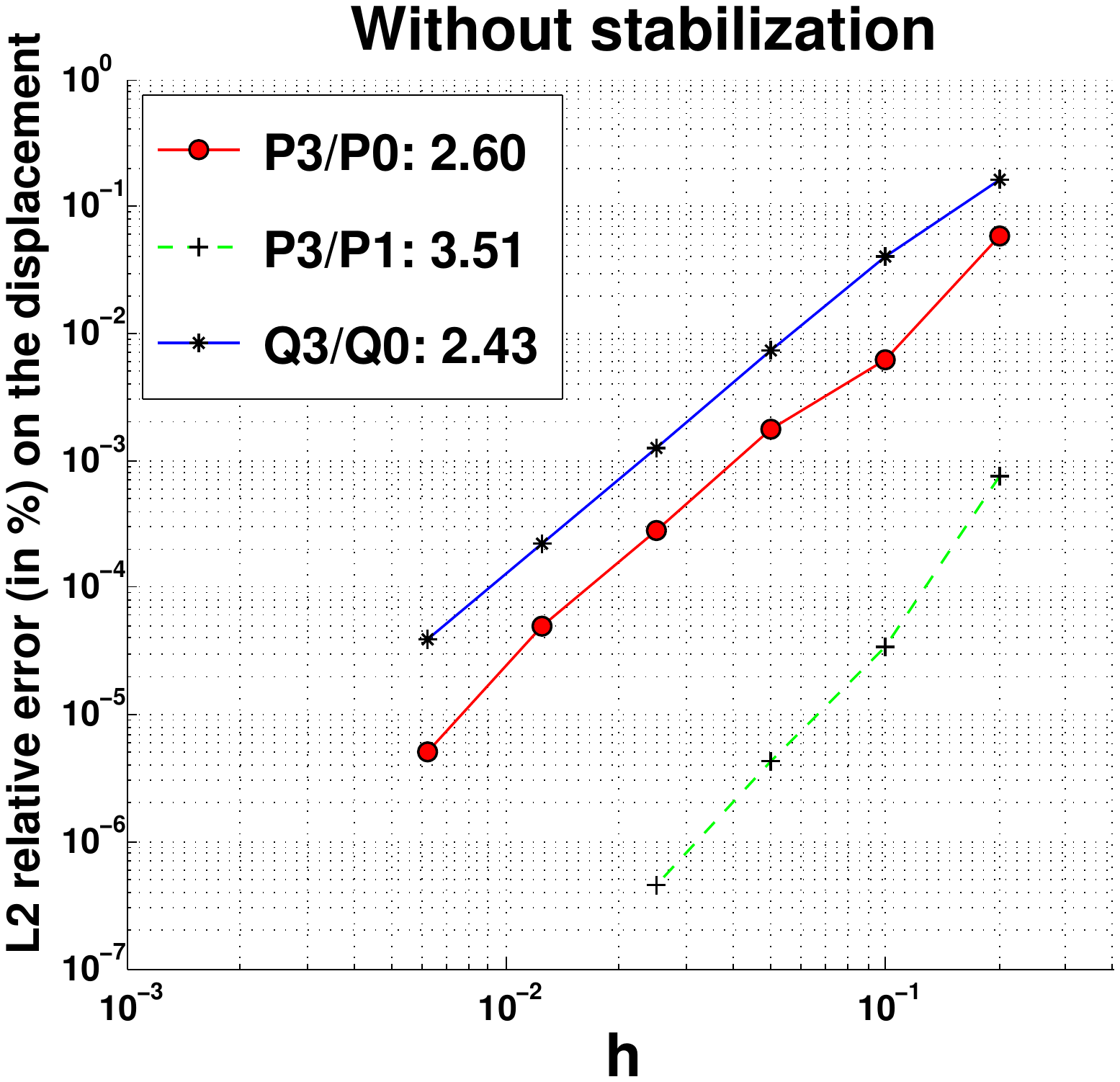}
&
\includegraphics[trim = 1cm 7cm 2cm 7cm, clip, scale=0.45]{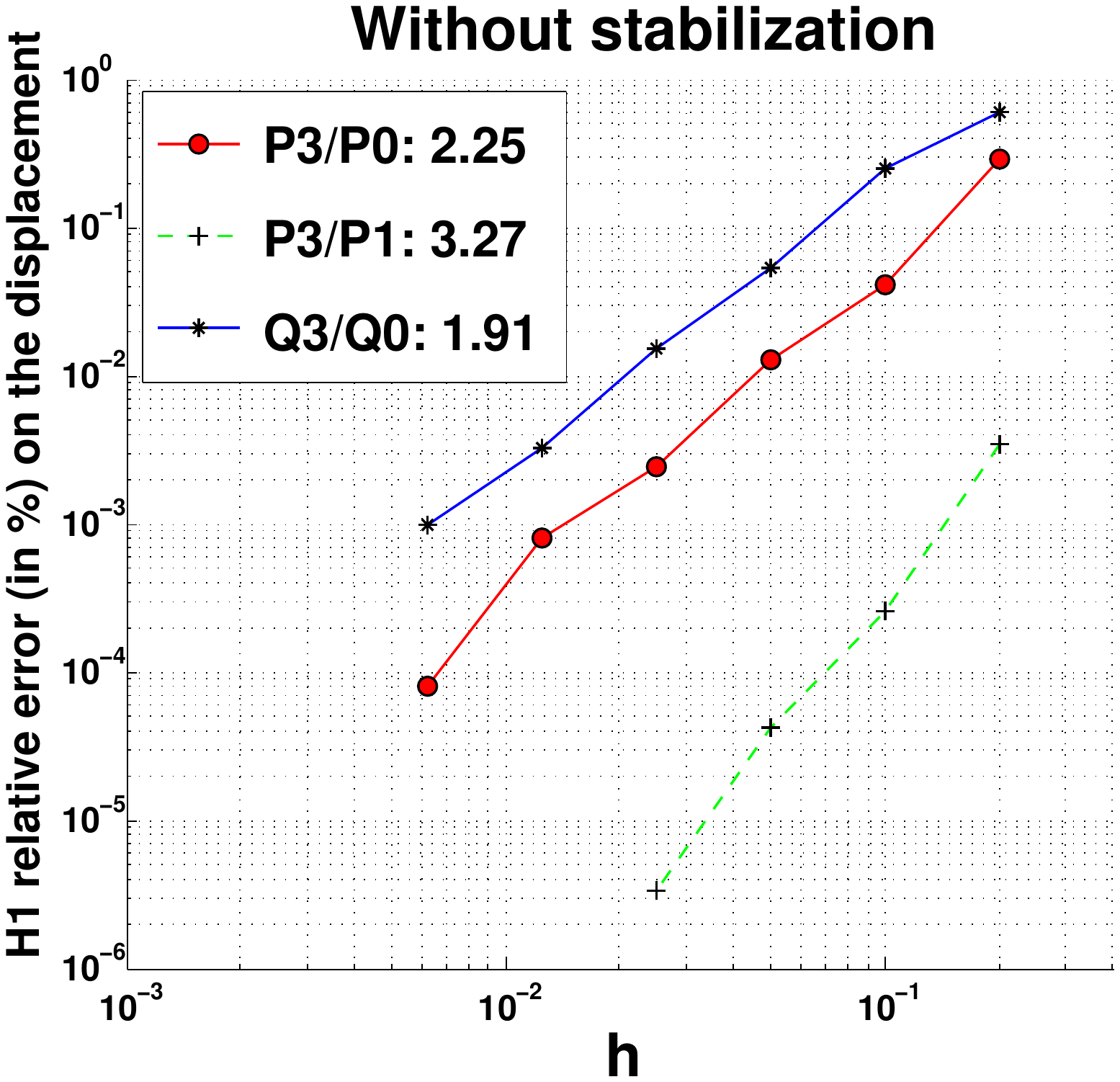}
\end{tabular}
\centering\caption{$\mathbf{L}^2(\Omega)$ and $\mathbf{H}^1(\Omega)$-relative errors (in \%) on the displacement in function of the mesh size~$h$, without stabilization, and estimation of convergence rates from the slope of the curves by linear regression.
\label{figrate0}}
\end{figure}
\FloatBarrier

We observe (in Figure~\ref{figrate0}) that the optimal rate of convergence for the displacement seems to be better reached - up to round numbers errors - when P1 or Q1 elements are chosen for the multiplier instead of discontinuous P0 or Q0 elements, for the $\mathbf{L}^2(\Omega)$-relative errors as well as for the $\mathbf{H}^1(\Omega)$-relative errors. Furthermore, we notice that the rates of convergence of the $\mathbf{H}^1(\Omega)$-norm for the displacement are better than expected when P1 or Q1 elements are chosen for the multiplier. Last, note that some computations are missing - namely tests with Q3/Q1 elements and some tests with P3/P1 elements - because in these cases computing errors implies handling quantities which are close to the engine precision, making these tests irrelevant.

\paragraph{Rates of convergence for the multiplier}\hfill \\
We denote by $\mathbf{U}^+_{ex}$, $\mathbf{U}^-_{ex}$ and $\bLambda_{ex}$ the discrete vectors interpolating the functions $\bu^+_{ex}:= {\bu^+_{ex}}_{| \Omega^+}$, $\bu^-_{ex} :={\bu^-_{ex}}_{| \Omega^-}$ and $\blambda_{ex} = \mp \sigma_{L}(\bu^{\pm}_{ex})\bn^{\pm}$ respectively. The $\mathbf{L}^2$-type error introduced by the approximation $\bLambda$ of the exact vector $\bLambda_{ex}$ is considered as the square root of
\begin{eqnarray*}
\int_{\Gamma_0} \left|\sigma_L(\mathbf{U}^+_{ex})\bn^+ + \bLambda \right|^2 \d\Gamma_0 +
\int_{\Gamma_0} \left|\sigma_L(\mathbf{U}^-_{ex})\bn^- - \bLambda \right|^2 \d\Gamma_0.
\end{eqnarray*}
For practical purposes, we compute this error by developing the underlying inner product as
\begin{eqnarray}
\langle A^+_{\bu \bu} \mathbf{U}^+_{ex}, \mathbf{U}^+_{ex} \rangle 
+ \langle A^-_{\bu \bu} \mathbf{U}^-_{ex}, \mathbf{U}^-_{ex} \rangle 
+ 2 \langle A^+_{\bu \blambda} \mathbf{U}^+_{ex}, \bLambda_{ex} \rangle
- 2 \langle A^-_{\bu \blambda} \mathbf{U}^-_{ex}, \bLambda_{ex} \rangle
+ 2 \langle A_{\blambda \blambda} \bLambda, \bLambda \rangle \label{scmat}
\end{eqnarray}
where, the matrices $A^{\pm}_{\bu \bu}$, $A^{\pm}_{\bu \blambda}$ and $A_{\blambda \blambda}$ are the discretized representations of the respective following bilinear forms
\begin{eqnarray*}
& & \mathcal{A}^{\pm}_{\bu \bu} : (\bu, \bv)  \mapsto 
\int_{\Gamma_0}\sigma_L(\bu)\bn^{\pm} \cdot \sigma_L(\bv)\bn^{\pm} \d \Gamma_0, \quad 
\mathcal{A}^{\pm}_{\bu \blambda} : (\bu, \bv)  \mapsto 
\int_{\Gamma_0}\sigma_L(\bu)\bn^{\pm} \cdot \blambda \d \Gamma_0, \\
& & \mathcal{A}_{\blambda \blambda} : (\blambda, \bmu)  \mapsto 
\int_{\Gamma_0}\blambda \cdot \bmu \d \Gamma_0. 
\end{eqnarray*}
Relative errors provided in Figure~\ref{figmult0} are thus computed as the square root of the inner product~\eqref{scmat} divided by $\displaystyle \left\| \sigma_L(\mathbf{U}^+_{ex})\bn^+ \right\|_{\mathbf{L}^2(\Gamma_0)}^2 + \left\| \sigma_L(\mathbf{U}^-_{ex})\bn^- \right\|_{\mathbf{L}^2(\Gamma_0)}^2 = \langle A^+_{\bu \bu} \mathbf{U}^+_{ex}, \mathbf{U}^+_{ex} \rangle + \langle A^-_{\bu \bu} \mathbf{U}^-_{ex}, \mathbf{U}^-_{ex} \rangle$.

\begin{figure}[!h]
\begin{tabular} {c|c}
\includegraphics[trim = 1cm 7cm 2cm 7cm, clip, scale=0.45]{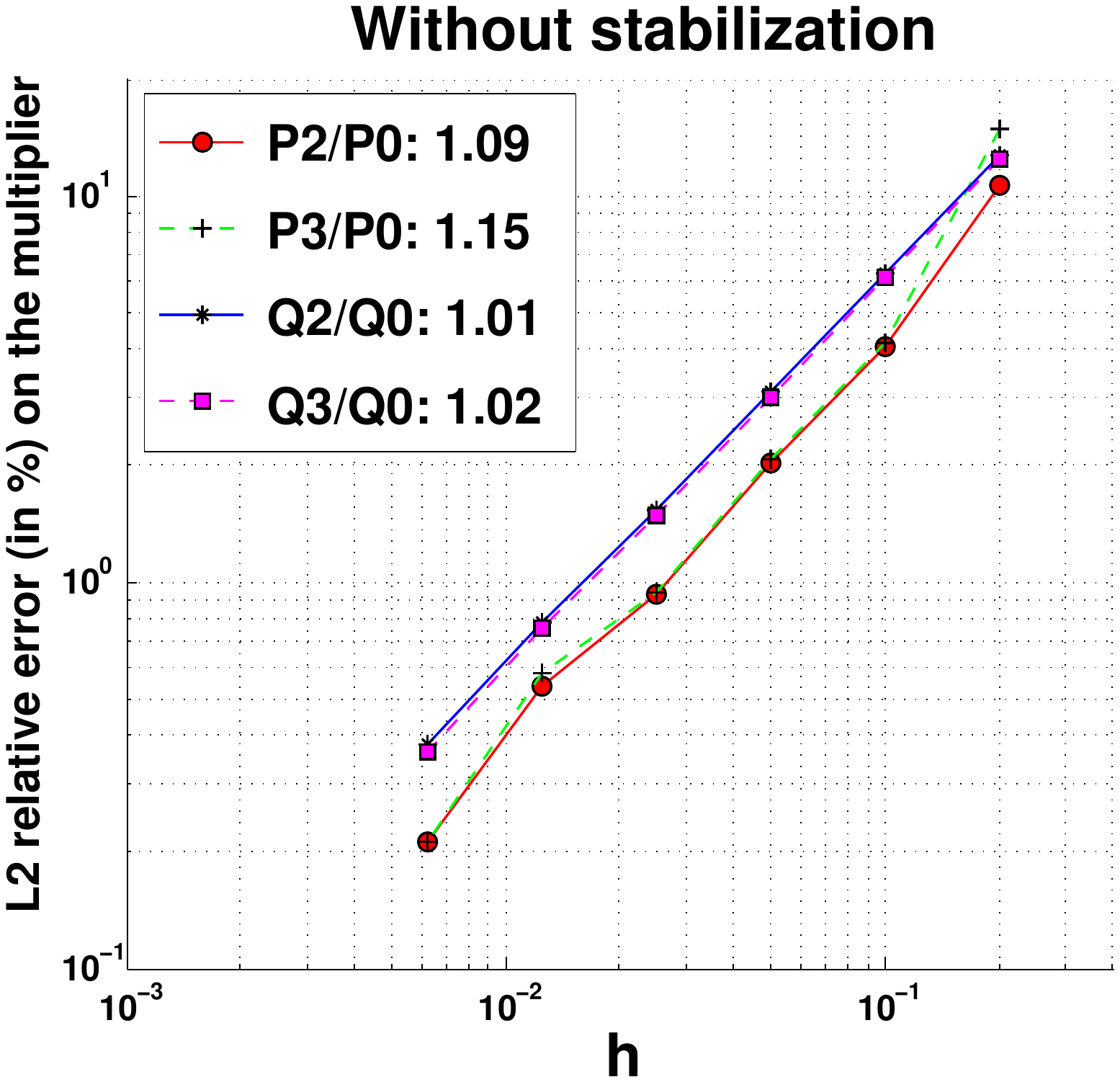}
&
\includegraphics[trim = 1cm 7cm 2cm 7cm, clip, scale=0.45]{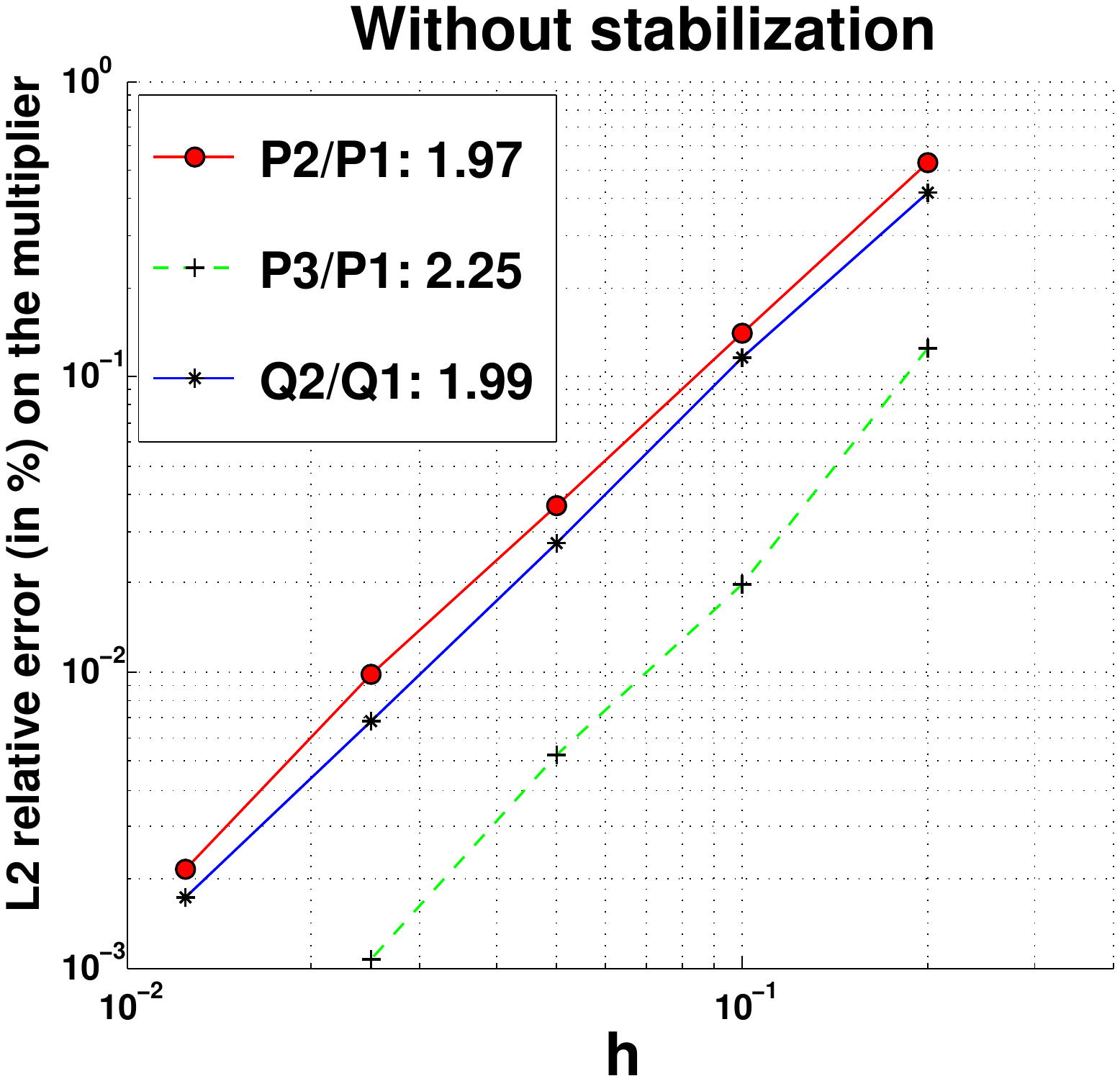}
\end{tabular}
\centering\caption{$\mathbf{L}^2(\Omega)$-relative error (in \%) on the multiplier in function of the mesh size~$h$, without stabilization, and estimation of convergence rates from the slope of the curves by linear regression.\label{figmult0}}
\end{figure}
\FloatBarrier

In Figure~\ref{figmult0} the optimal rate of convergence for the multiplier is obtained for P0 or Q0 elements as well as for P1 or Q1 elements. Note that the theoretical analysis for this fictitious domain approach does not guarantee {\it a priori} such a good convergence for dual variables. Anyway, without performing any stabilization, it seems that, for most geometric configurations the method provides us a good approximation for the values of this multiplier. Pathological configurations which highlight the necessity of a special treatment are in general hard to find. However, we show in Section~\ref{secrobust} that the stabilization technique brings us a gain of robustness with respect to the geometry, as it is shown in Figure~\ref{figrobmL2}.\\
Besides, as expected in view of Remark~\ref{remarkcont}, we observed numerically that the quantity
\begin{eqnarray*}
\displaystyle \left\| \sigma_L(\mathbf{U}^+_{ex})\bn^+  +  \sigma_L(\mathbf{U}^-_{ex})\bn^- \right\|_{\mathbf{L}^2(\Gamma_0)}^2 
\end{eqnarray*}
is negligible compared with $\displaystyle \left\| \sigma_L(\mathbf{U}^+_{ex})\bn^+ \right\|_{\mathbf{L}^2(\Gamma_0)}^2 + \left\| \sigma_L(\mathbf{U}^-_{ex})\bn^- \right\|_{\mathbf{L}^2(\Gamma_0)}^2$, without imposing {\it a priori} the condition $\left[ \sigma_L(\bu) \right]\bn^+$ across $\Gamma_0$.

\paragraph{Illustration in 2D: Deformation in volcanic rift zones} \hfill \\ 
In order to illustrate the method on realistic 2D tests, we consider the models treated in \cite{Pollard}, namely the computation of characteristics of displacements due to inside cracks. We consider the rectangular domain $[0;100]\times [0;50]$ endowed with a Cartesian mesh. The only right-hand-side is a constant pressure applied on both sides of the fracture as a Neumann-condition. The fracture position is determined by the points of coordinates $(x,y)$ satisfying
\begin{eqnarray*}
y = 2(x-48.0) + 35.0, & \quad & 35 \leq y \leq 45. 
\end{eqnarray*}
Standard P2 elements are chosen for the displacement, and P0 elements are chosen for the multiplier. Stabilization is not performed, since here we are not interested in computing the multiplier. Illustrations are presented in Figure~\ref{figpollard} and Figure~\ref{figstream}.

\begin{figure}[!h]
\begin{center}
\begin{tabular} {c|c}
\includegraphics[trim = 1cm 0.5cm 1cm 1cm, clip, scale=0.18]{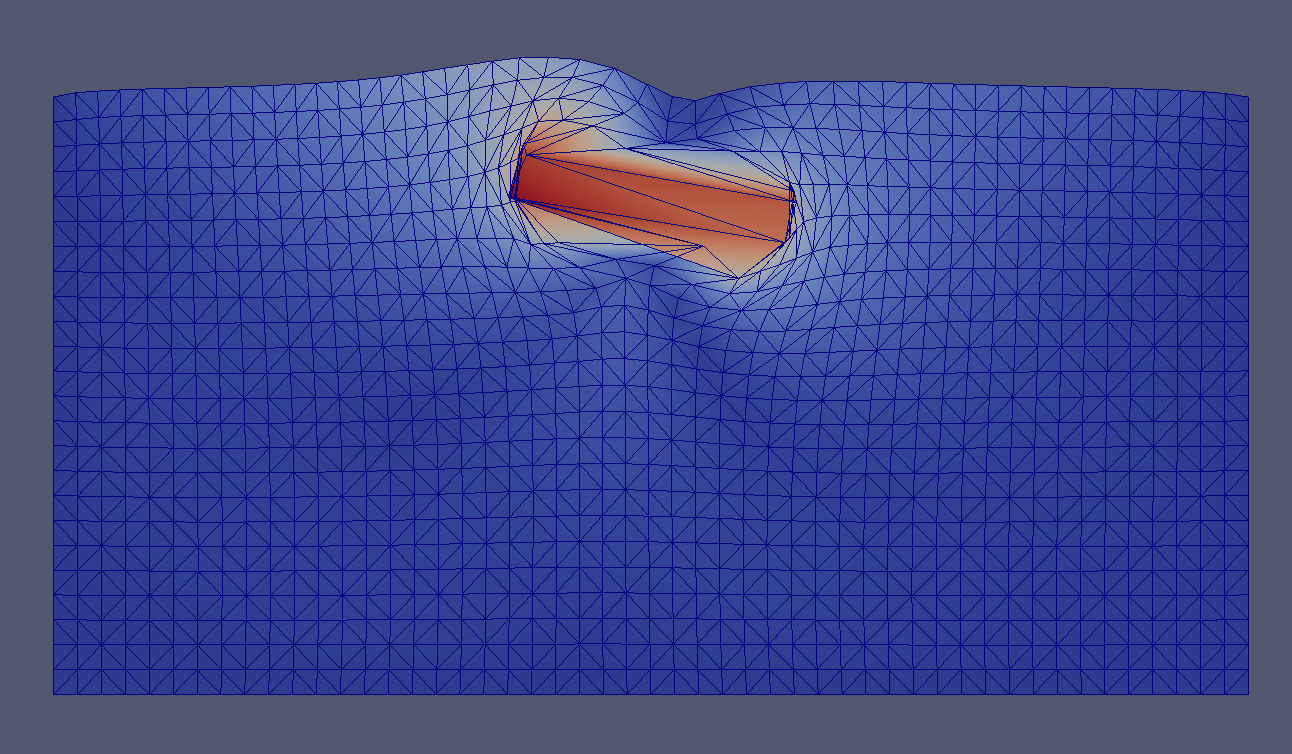}
&
\includegraphics[trim = 1cm 0.5cm 1cm 1cm, clip, scale=0.18]{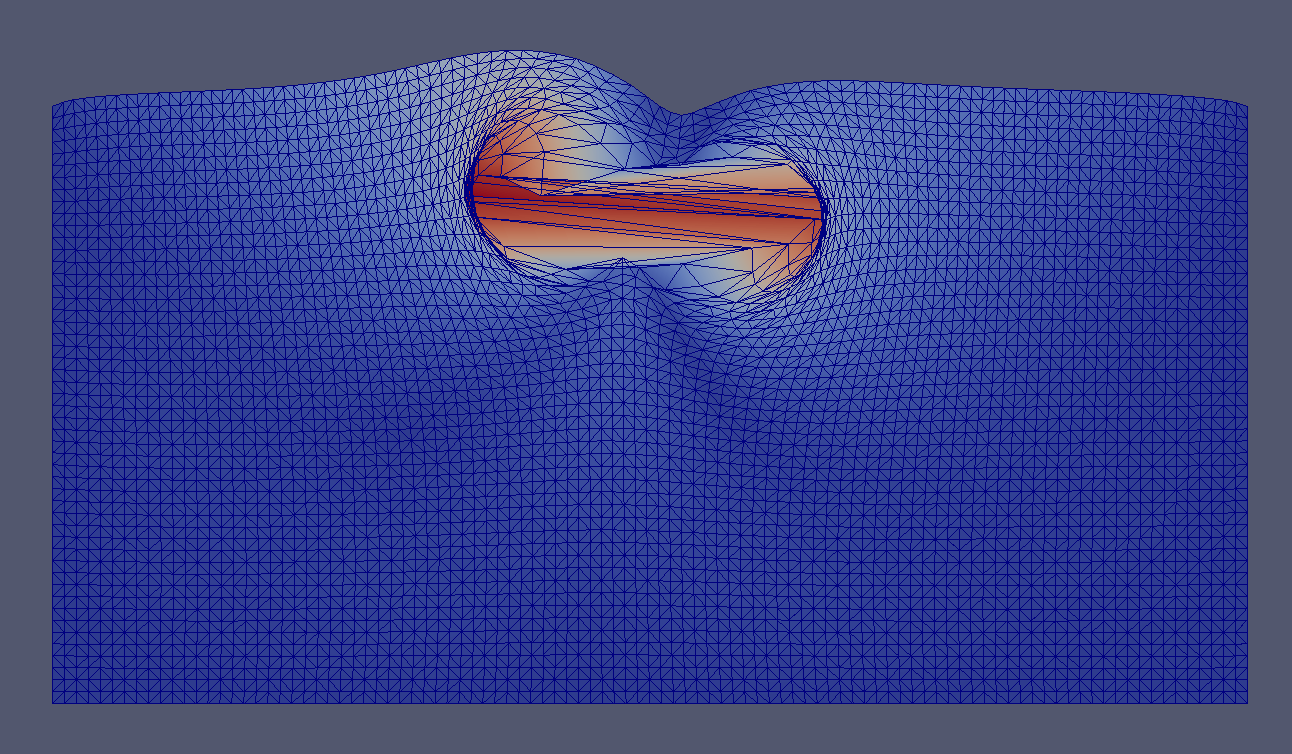}
\\
\hline
\includegraphics[trim = 1cm 0.5cm 1cm 1cm, clip, scale=0.18]{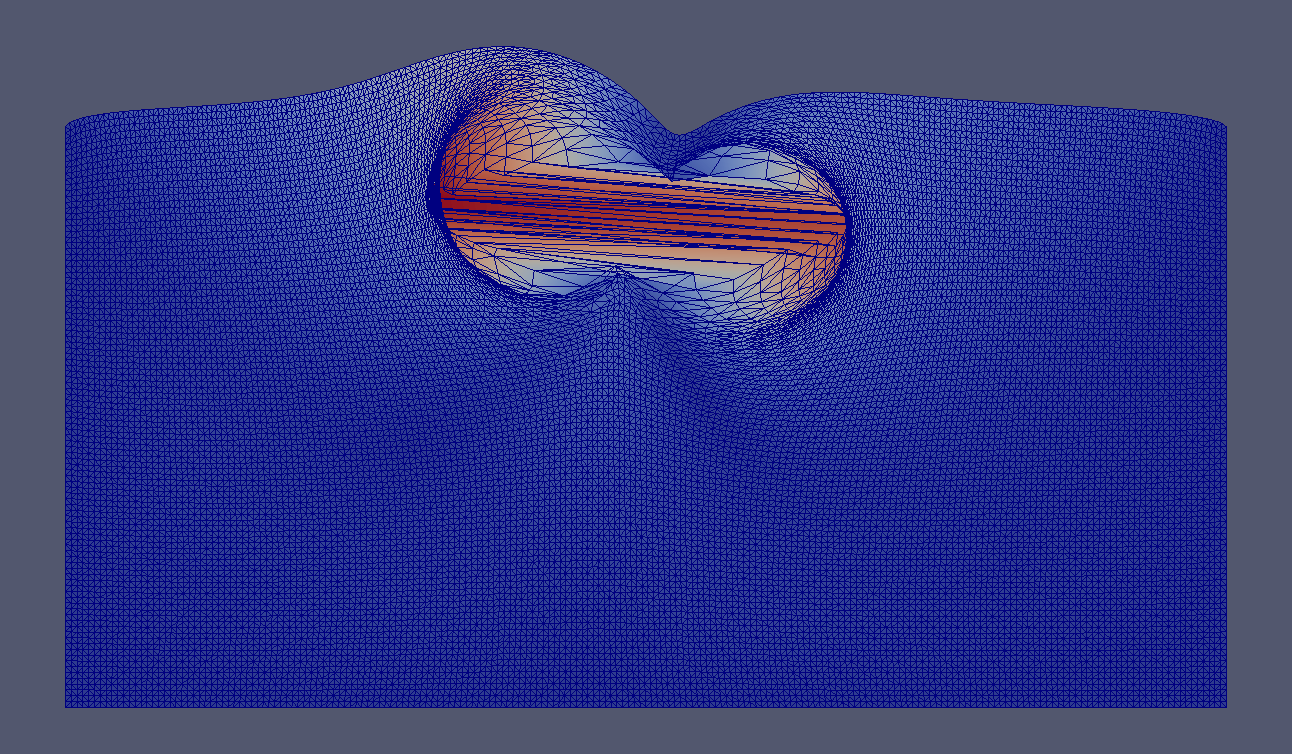}
&
\includegraphics[trim = 1cm 0.5cm 1cm 1cm, clip, scale=0.18]{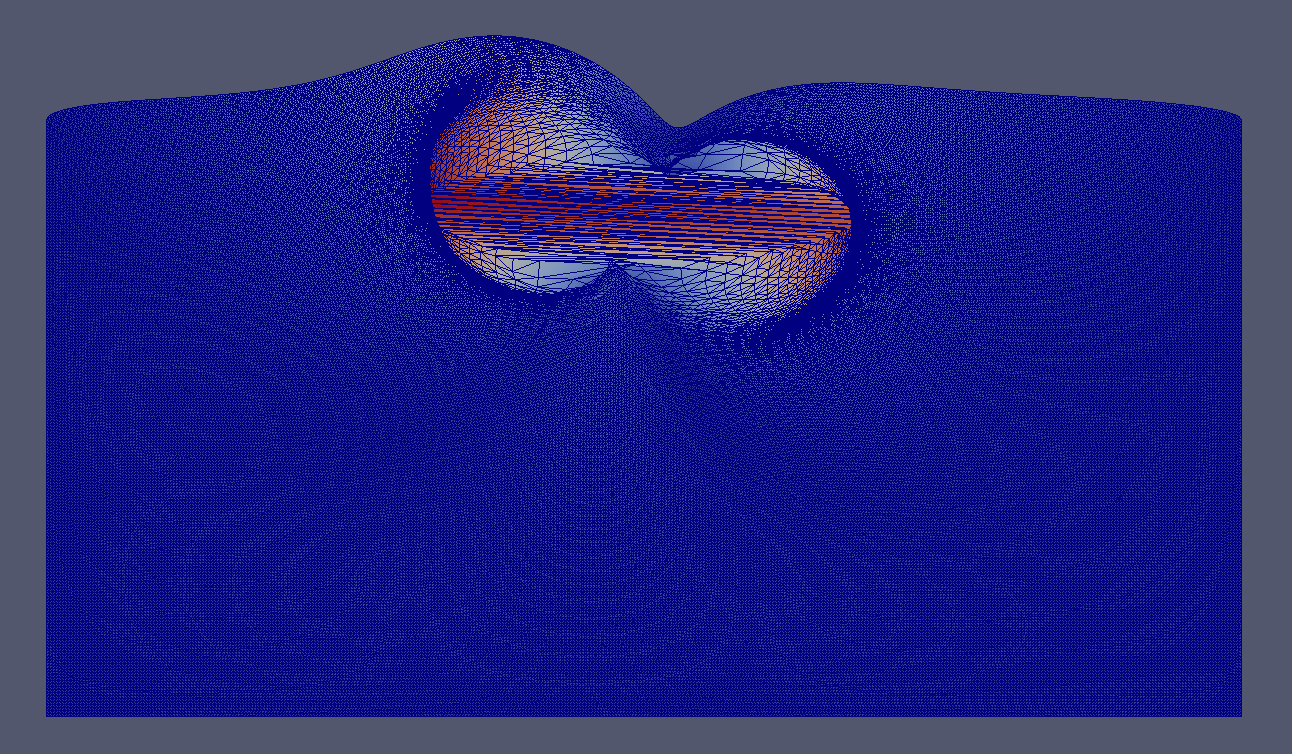}
\end{tabular}
\end{center}
\centering\caption{Representation of the intensity of displacement due to an inclined straight fracture, on a Cartesian mesh warped with respect to the deformation, for different numbers of subdivisions, respectively: $25\times 12$, $50 \times 25$, $100\times 50$ and $200\times 100$.\label{figpollard}}
\end{figure}
\FloatBarrier

\begin{figure}[!h]
\begin{center}
\includegraphics[trim = 1cm 1.6cm 0cm 2.3cm, clip, scale=0.50]{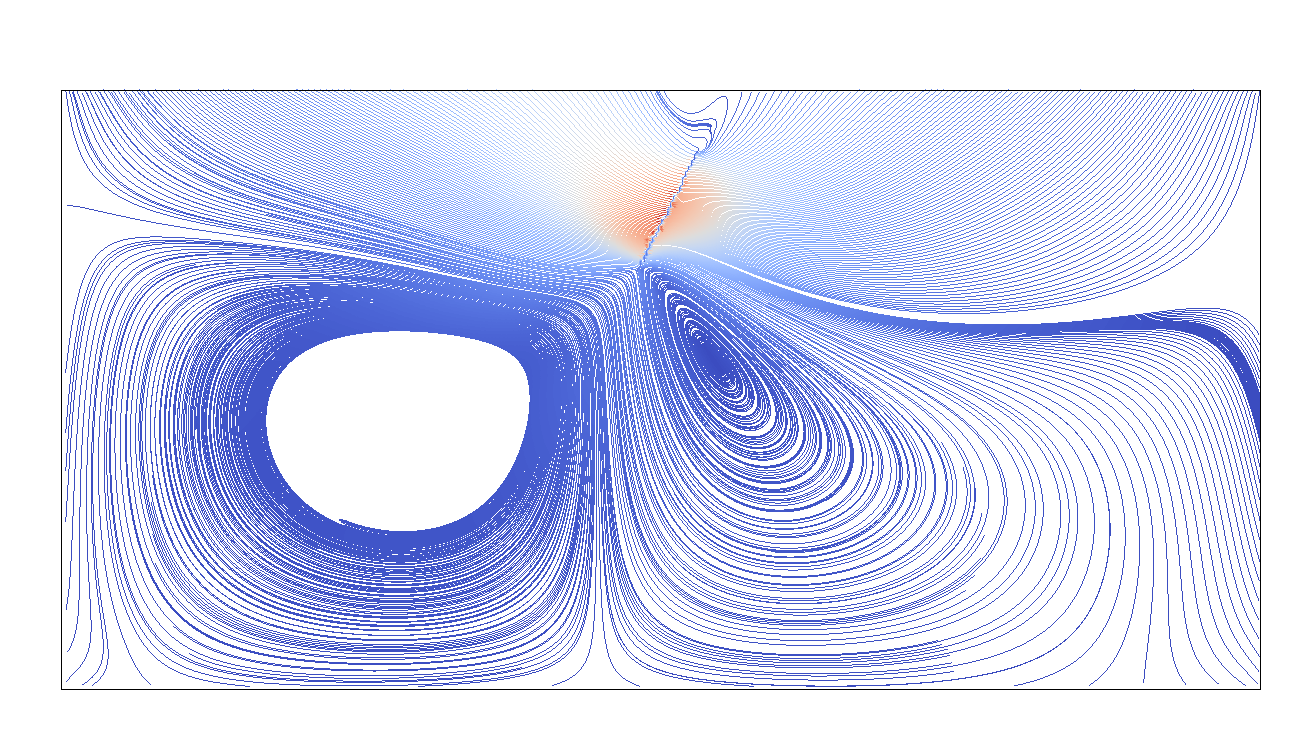}
\end{center}
\centering\caption{Outline of level curves for the displacement due to an inclined straight fracture.\label{figstream}}
\end{figure}
\FloatBarrier

\subsection{Choice of the stabilization parameter $\gamma$} \label{secgamma}
We are now interested in the stabilized problem~\eqref{pbfv} whose the explicit matrix formulation is given by~\eqref{systdec}. Like in \cite{HaslR} and \cite{Court}, recall that we choose $\gamma = \gamma_0*h$, with $\gamma_0 > 0$ constant, independent of the mesh size~$h$. This constant has to be chosen judiciously. On one hand it represents the importance given to the quality of the approximation for the multiplier, but on the other hand, the stabilization term degrades the coerciveness of the whole system, so $\gamma_0$ has to remain moderate. In Figure~\ref{figgamma} we provide a test in which the relative error on the multiplier $\blambda$ is computed for a range of values for $\gamma_0$.

\begin{figure}[!h]
\begin{tabular} {c|c}
\includegraphics[trim = 1cm 7cm 2cm 7cm, clip, scale=0.45]{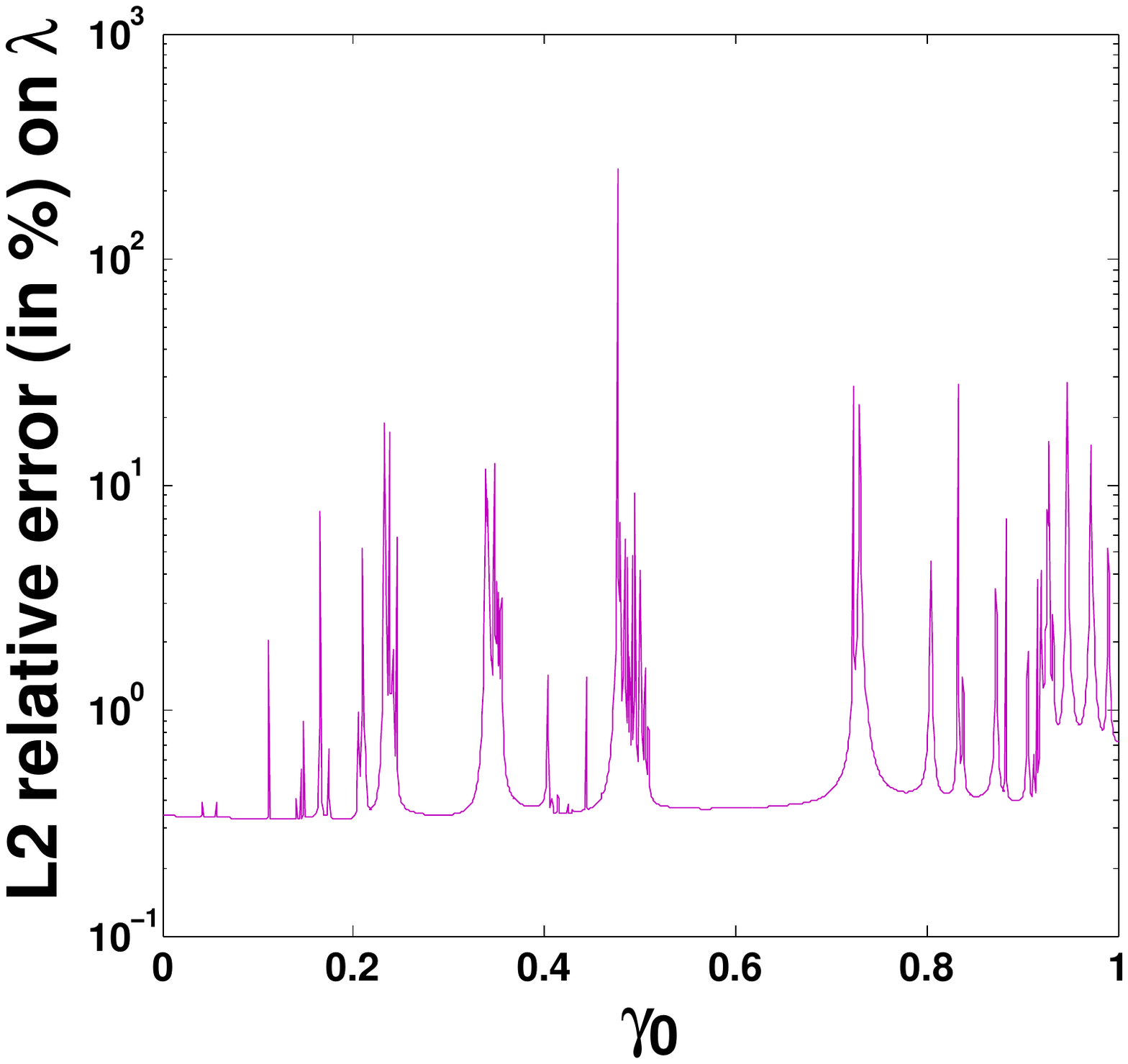} &
\includegraphics[trim = 1cm 7cm 2cm 7cm, clip, scale=0.45]{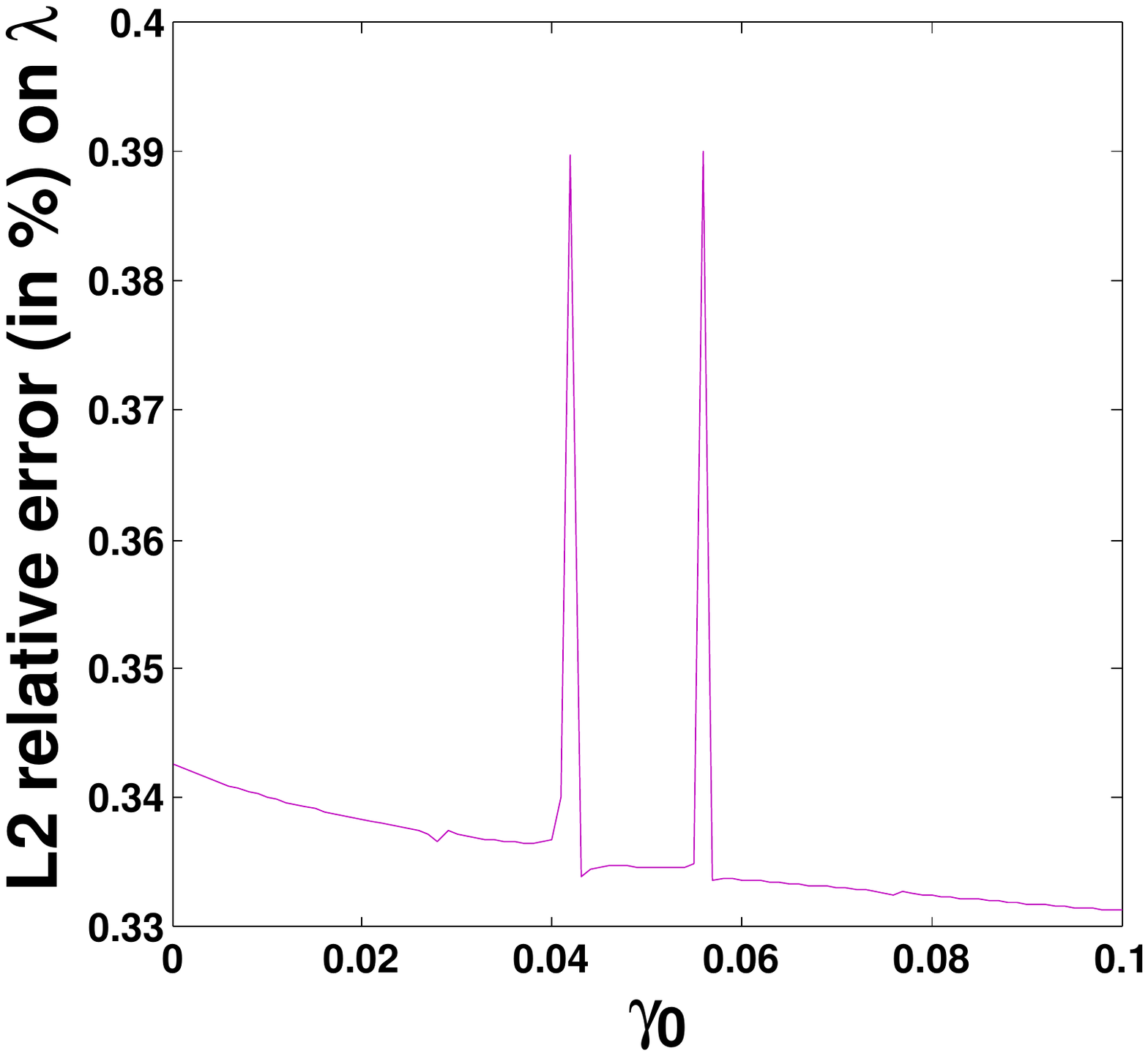}
\end{tabular}
\centering\caption{$\mathbf{L}^2(\Gamma_0)$-relative error (in \%) on $\blambda$ for different position of the crack $\Gamma_T$.\label{figgamma}}
\end{figure}
\FloatBarrier

Note in the graph on the left of Figure~\ref{figgamma} that the first singular approximations on the multiplier occur for $\gamma_0 < 0.1$, and become more frequent and chaotic for larger values of $\gamma_0$. In the graph on the right we examine the error on $\blambda$ for more precise values of $\gamma_0 < 0.1$; We notice that the approximation of $\blambda$ gets better when $\gamma_0$ increases, until some values of $\gamma_0 > 0.04$ which generate the first picks. Thus we choose $\gamma_0 \leq 0.03$ for the rest of the study. 

\subsection{Numerical orders of convergence with stabilization} \label{secrates1}
In this subsection, we provide the same tests as performed without stabilization in Section~\ref{secexp0}. The stabilization technique is performed with the parameter $\gamma_0 = 0.03$. The results are given in Figure~\ref{figcvstab1} and Figure~\ref{figmult1}. As expected for a given geometry, we observe nearly the same rates of convergence as obtained without any stabilization, particularly for the multiplier. 


\begin{figure}[!h]
\begin{tabular} {c|c}
\includegraphics[trim = 1cm 7cm 2cm 7cm, clip, scale=0.45]{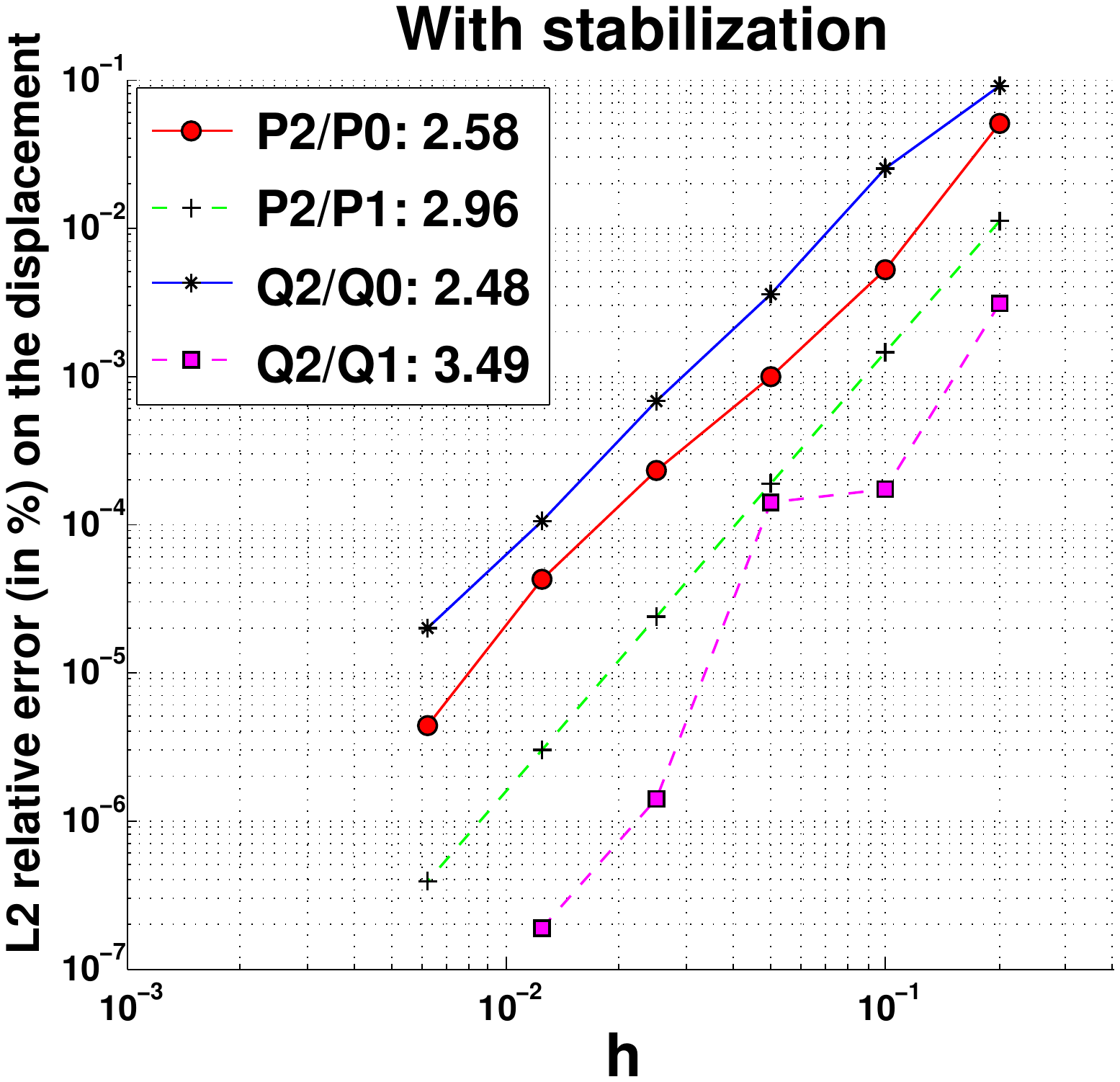}
&
\includegraphics[trim = 1cm 7cm 2cm 7cm, clip, scale=0.45]{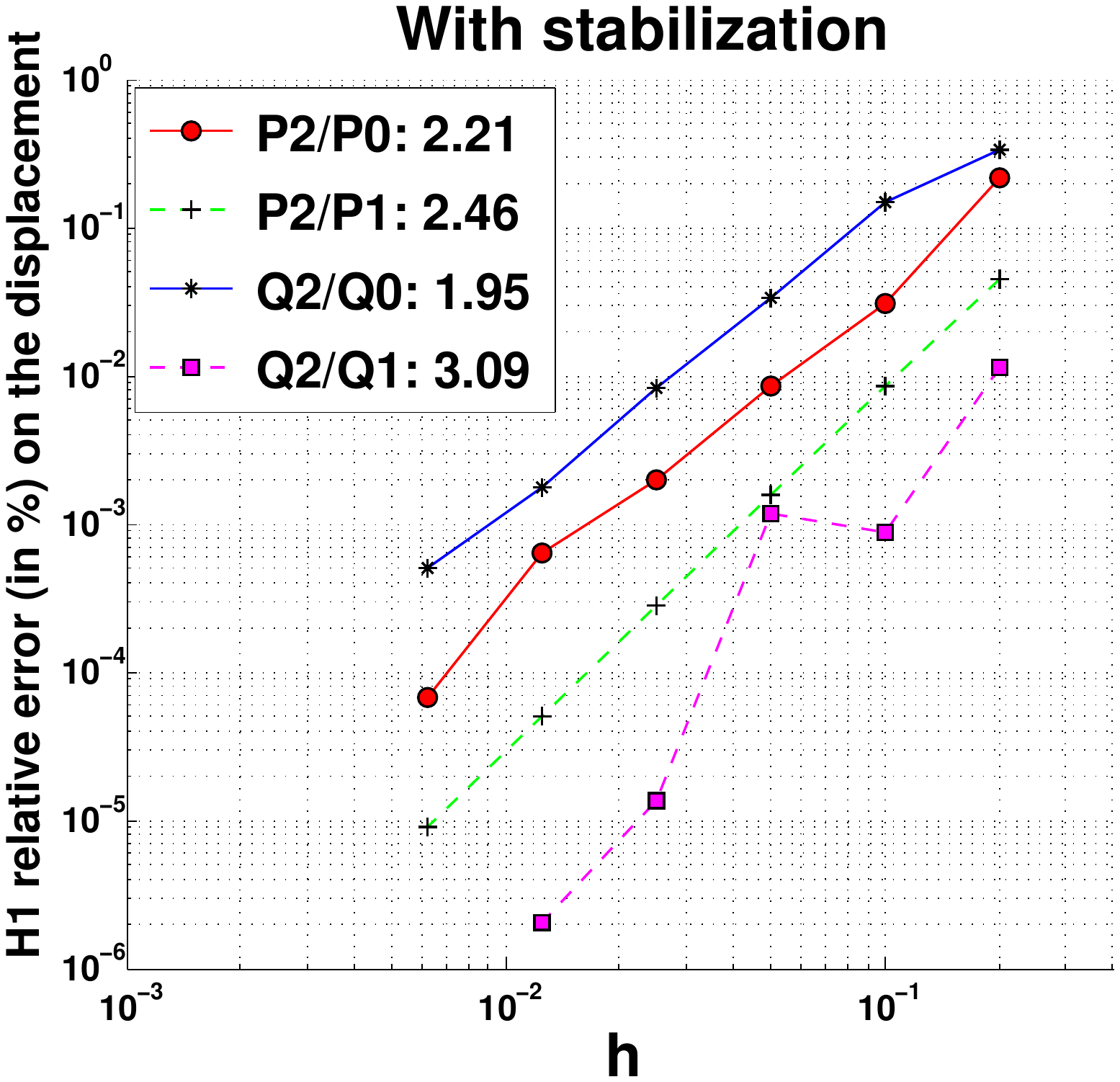}
\\
\includegraphics[trim = 1cm 7cm 2cm 7cm, clip, scale=0.45]{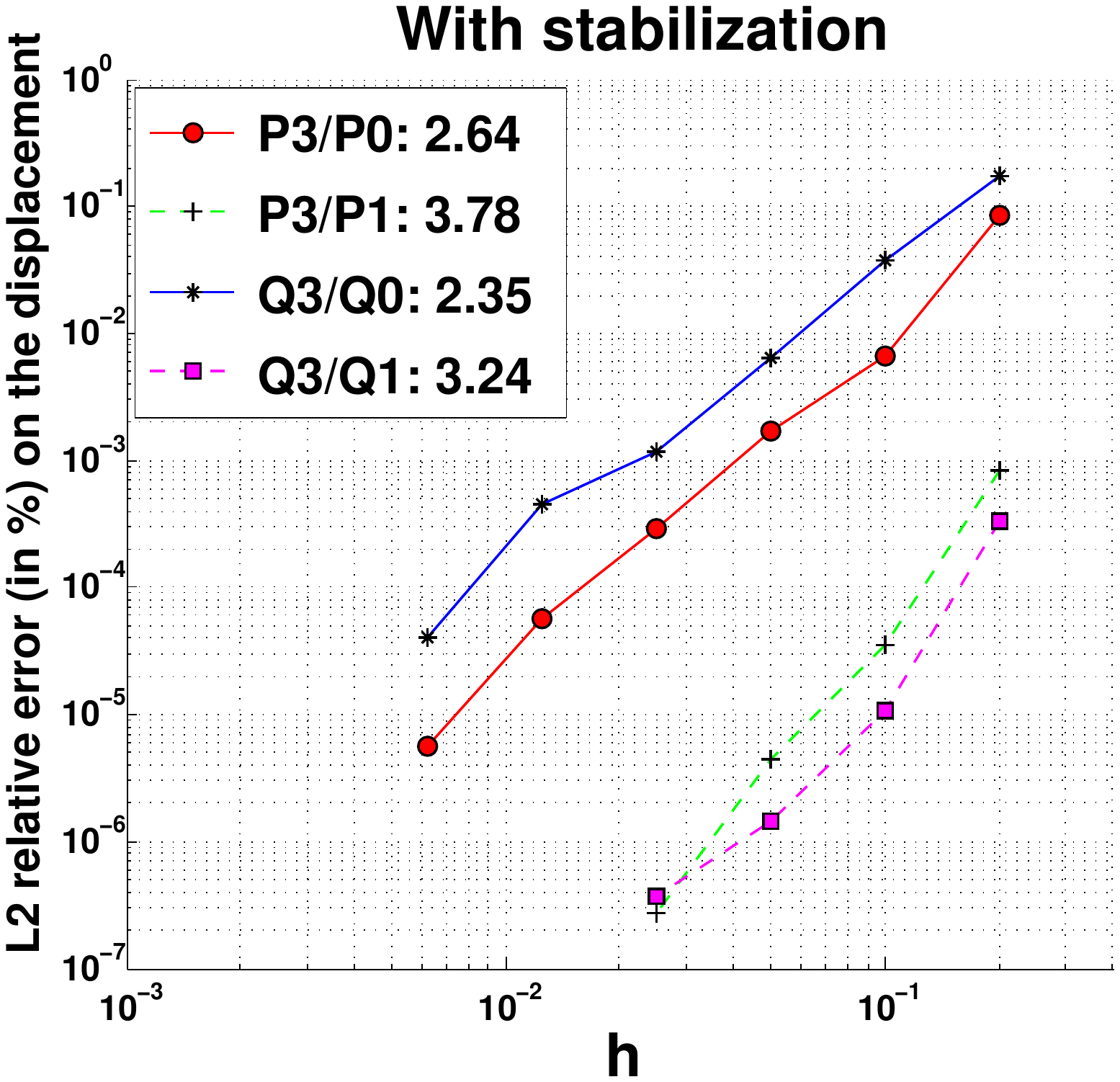}
&
\includegraphics[trim = 1cm 7cm 2cm 7cm, clip, scale=0.45]{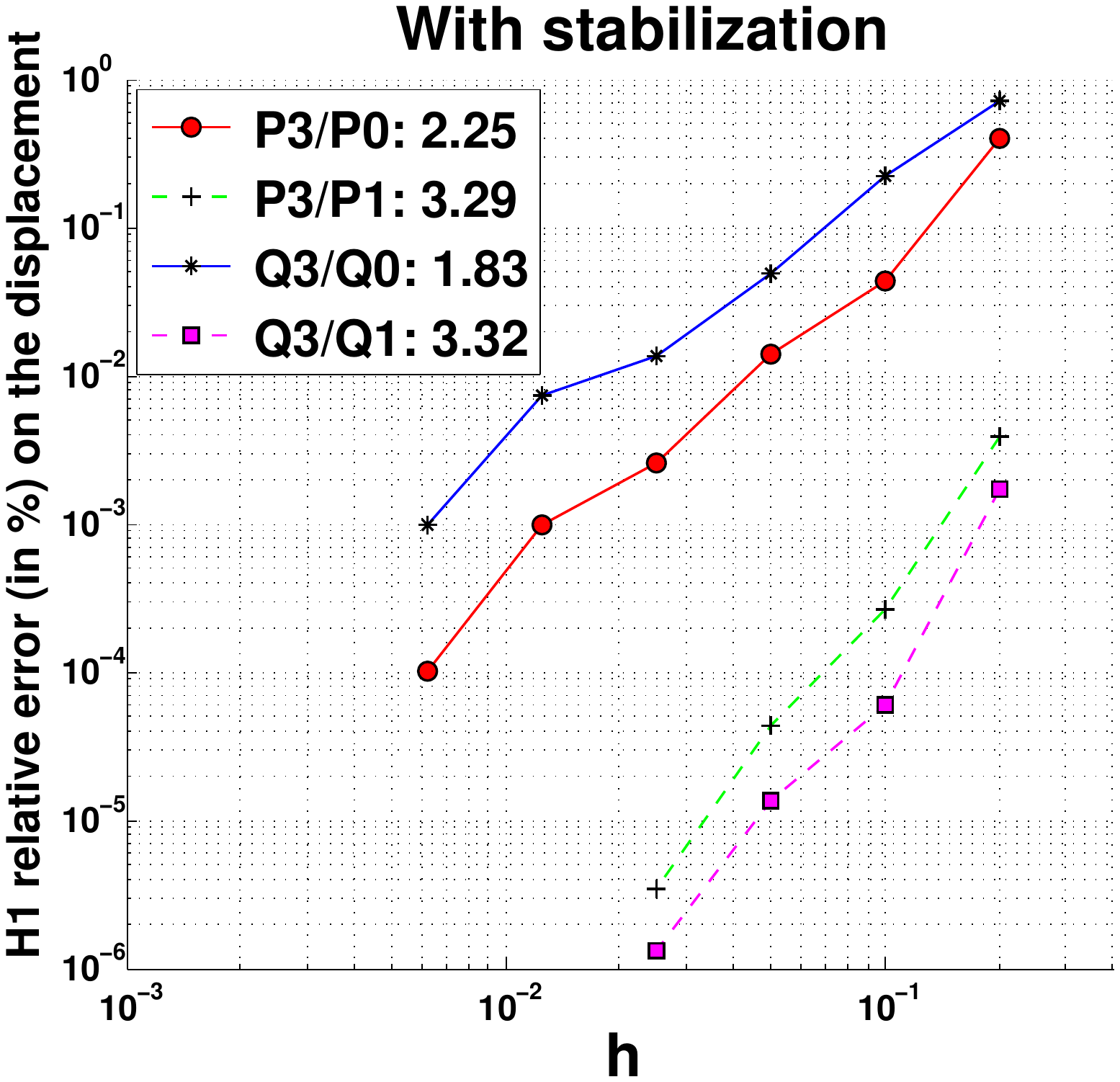}
\end{tabular}
\centering\caption{$\mathbf{L}^2(\Omega)$ and $\mathbf{H}^1(\Omega)$-relative errors (in \%) on the displacement in function of the mesh size~$h$, with stabilization, and estimation of convergence rates from the slope of the curves by linear regression.\label{figcvstab1}}
\end{figure}
\FloatBarrier


\begin{figure}[!h]
\begin{tabular} {c|c}
\includegraphics[trim = 1cm 7cm 2cm 7cm, clip, scale=0.45]{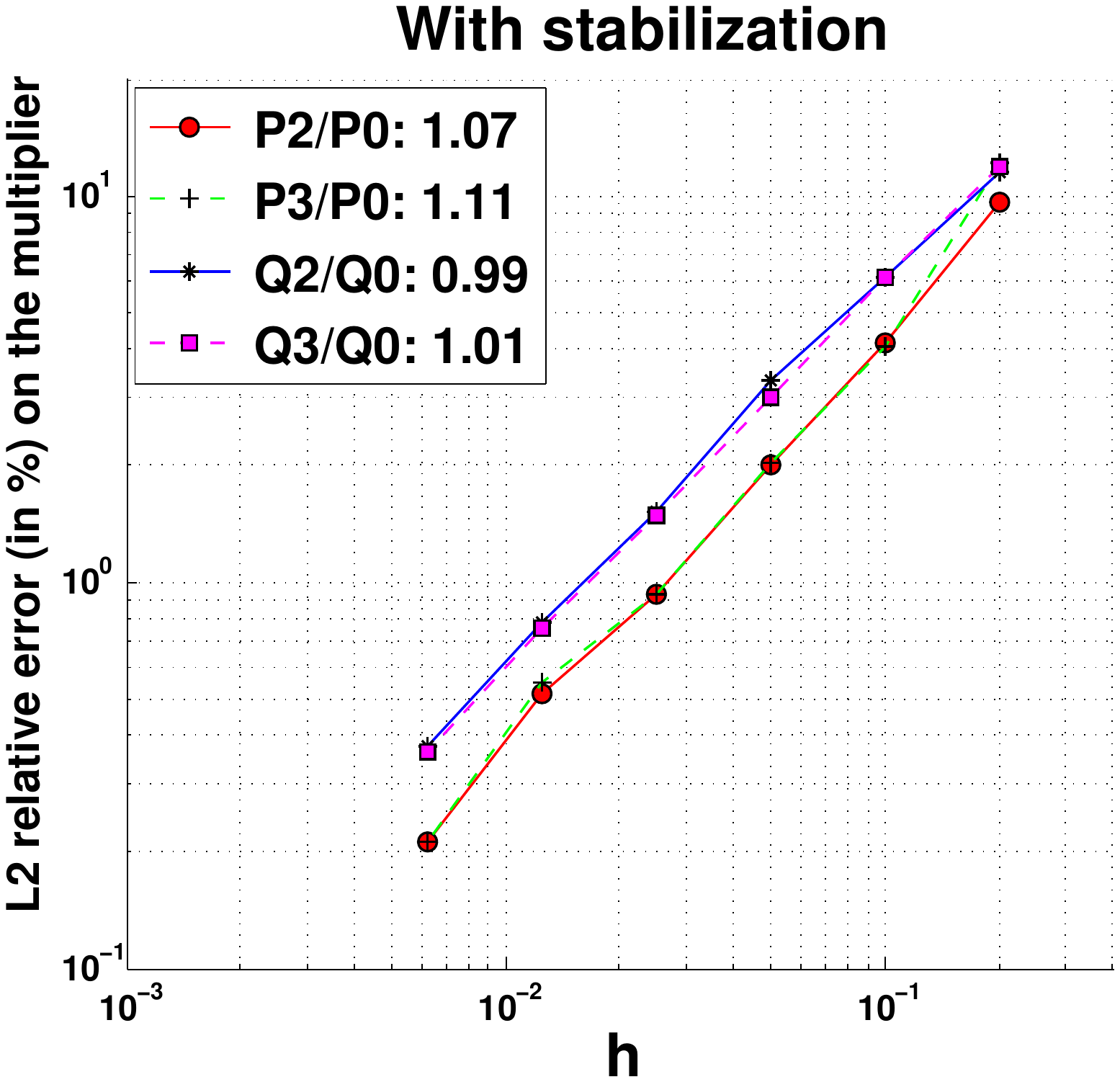}
&
\includegraphics[trim = 1cm 7cm 2cm 7cm, clip, scale=0.45]{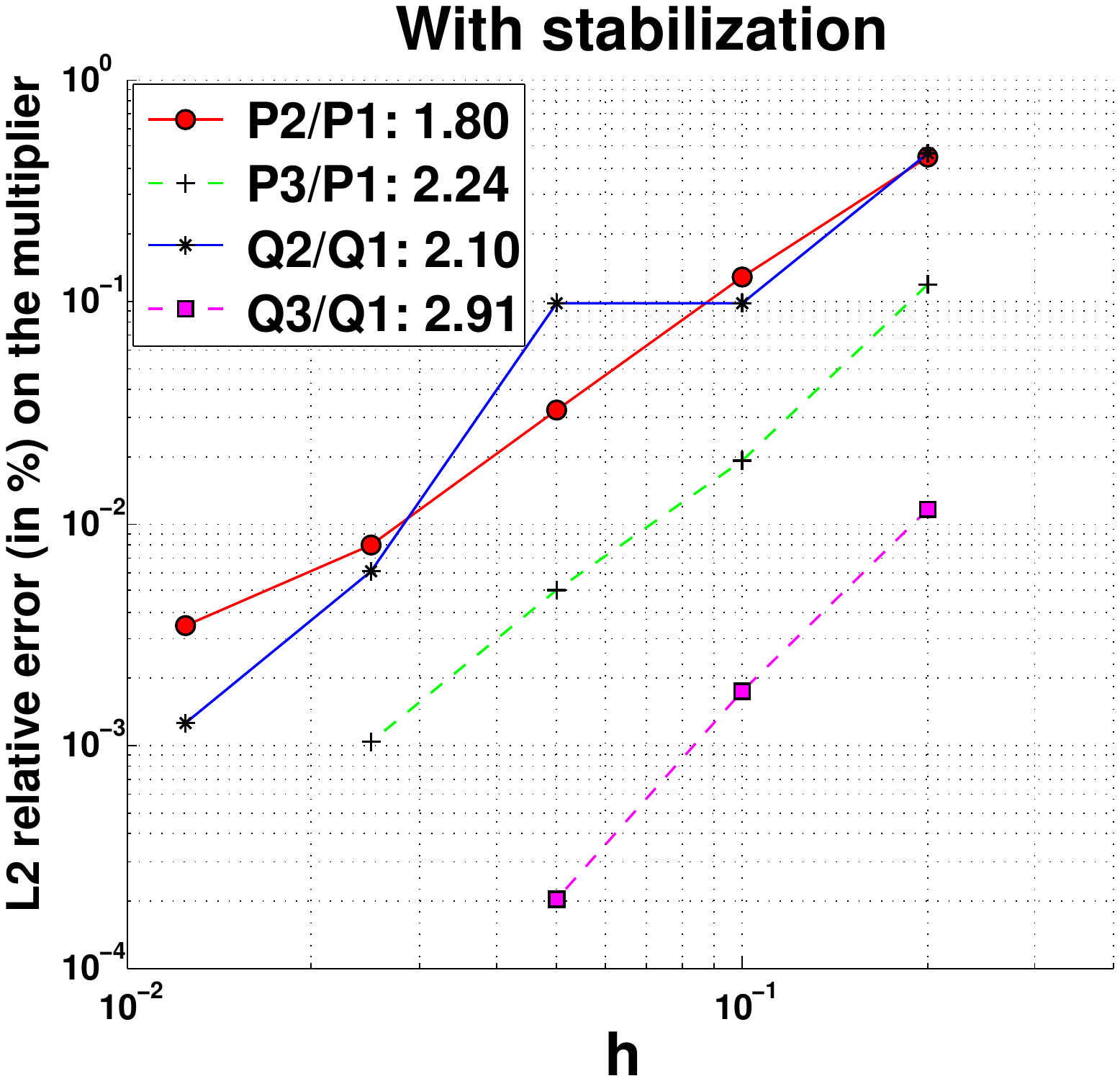}
\end{tabular}
\centering\caption{$\mathbf{L}^2(\Omega)$-relative error (in \%) on the multiplier in function of the mesh size~$h$, with stabilization, and estimation of convergence rates from the slope of the curves by linear regression.\label{figmult1}}
\end{figure}


\section{Robustness with respect to the geometry} \label{secrobust}
In order to highlight the main interest of the stabilization, we test the convergence behavior for different geometries, namely when the crack intersect the global mesh in different manners. For that, we perform two tests for which we analyze the relative error for the multiplier of $\mathbf{L}^2(\Gamma_0)$ as well as relative errors on global displacement, in $\mathbf{L}^2(\Omega)$ and $\mathbf{H}^1(\Omega)$, with and without using the stabilization technique.\\
In the first tests we make the length of the crack vary, keeping its position, and in the second tests the length is kept constant, but the position is varied. In these first tests, we do not observe any significant difference whether we perform the stabilization or not: Results are good and similar, quite better for the multiplier with stabilization. These results are not shown. The second tests show the $\mathbf{L}^2(\Gamma_0)$-relative error for the multiplier with and without stabilization; They are represented in Figure~\ref{figrobmL2}. In these second tests the computation of errors on the displacement do not reveal any significant improvement due to the stabilization technique, and so they are not presented either.

\begin{figure}[!h]
\begin{tabular} {c|c}
\includegraphics[trim = 1cm 7cm 2cm 7cm, clip, scale=0.45]{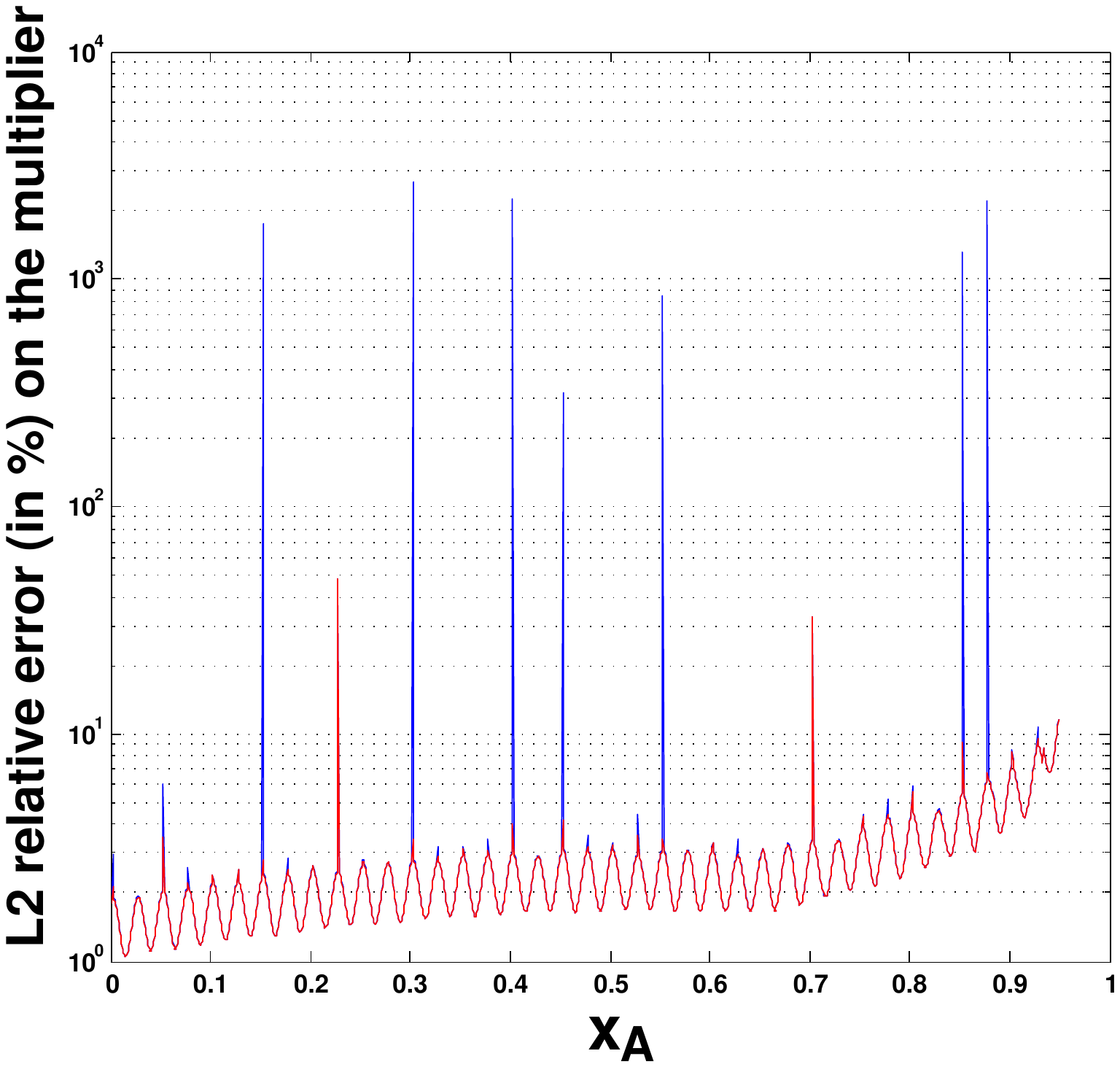}
&
\includegraphics[trim = 1cm 7cm 2cm 7cm, clip, scale=0.45]{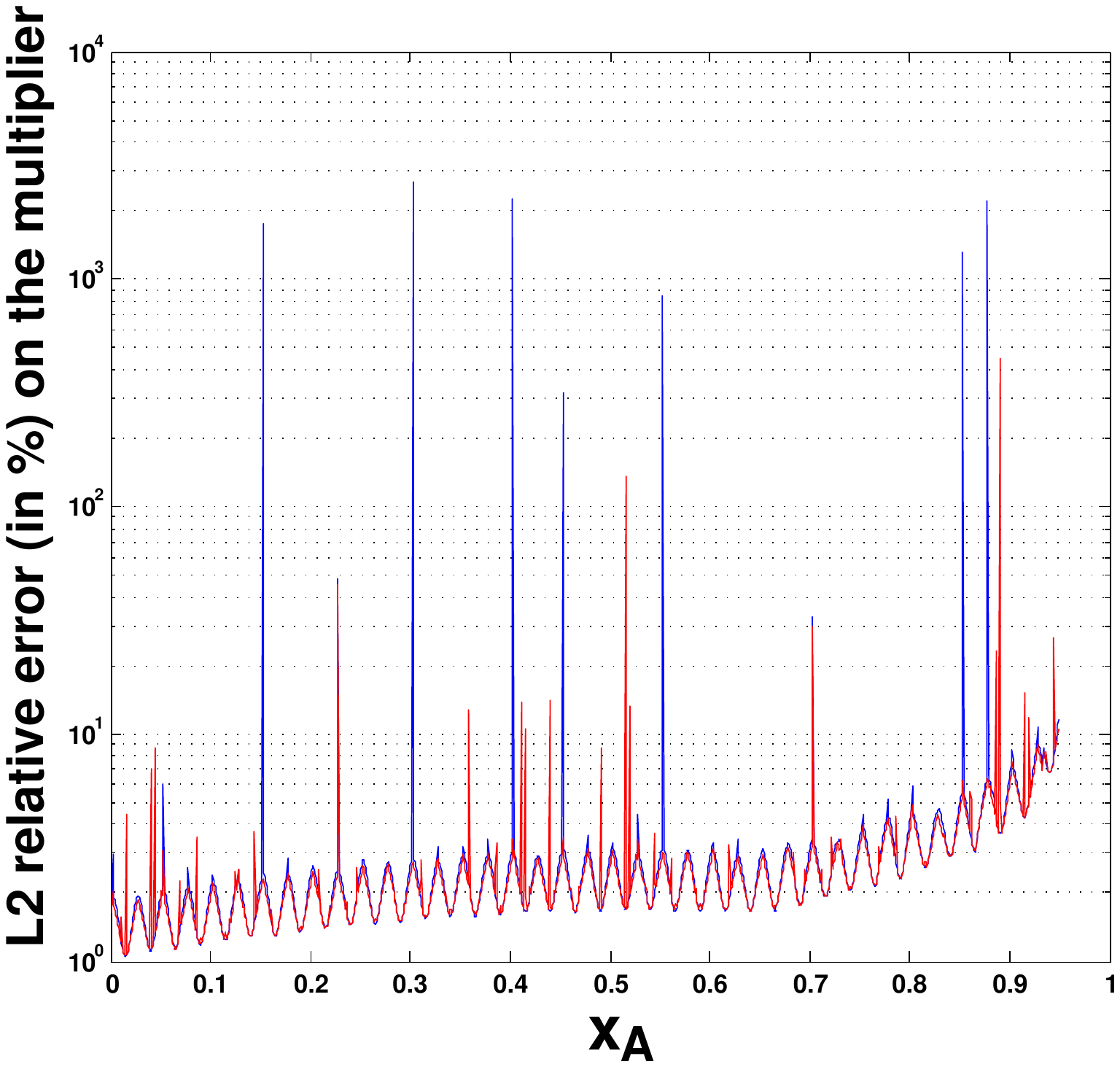}
\end{tabular}
\centering\caption{$\mathbf{L}^2(\Gamma_0)$-relative error (in \%) on the multiplier without stabilization (in blue) and with stabilization (in red), for various positions of the crack, with $\gamma_0 = 0.0005$ (left) and $\gamma_0 = 0.03$ (right).\label{figrobmL2}}
\end{figure}
\FloatBarrier

In Figure~\ref{figrobmL2}, the crack determined by the points of coordinates $(x,y)$ satisfying
\begin{eqnarray*}
y -2(x-x_0) = 0, \quad x_A - x < 0,  \quad  x - x_B < 0
\end{eqnarray*}
has been considered in the square $[0;1] \times [0;1]$, corresponding to the exact solutions given in Section~\ref{subsubtest0}. The abscissas $x_0$, $x_A$ and $x_B$ have been changed simultaneously, with $x_A$ going from $0.000$ to $0.950$, while keeping $x_B = x_A + 0.050$ and $x_0 = x_A - 0.153$. Finite elements P2 and P0 have been chosen for the displacement and the multiplier respectively.\\
We observe that bad computations - corresponding to huge $\mathbf{L}^2(\Gamma_0)$-relative error - occur in fewer cases with the stabilization technique, and have a smaller $\mathbf{L}^2(\Gamma_0)$-relative error. Moreover, the plots show that it is delicate to choose the stabilization parameter $\gamma_0$; If it is chosen too large (like for instance $\gamma_0 = 0.03$), the problematic situations occur more frequently, and thus it is preferable to choose a quite small $\gamma_0$ (like $\gamma_0 = 0.0005$), even if the system is close to a system without any stabilization terms.

\section{Realistic numerical simulations: Deformation of a volcano} \label{secPiton}
The code developed in 2D has been extended to dimension 3, and validated with the same type of convergence curves as the ones given in Section~\ref{secnumcurve} by considering exact solutions. As an illustration, we consider the geophysical problem of a pressurized crack inside a volcano.\\

A 3D magma-filled crack (also called a {\it dike}) is approximated with a set of triangles, as represented in Figure~\ref{figdike}. 

\begin{figure}[!h]
\begin{tabular} {c|c}
\includegraphics[trim = 1cm 0cm 0cm 0cm, clip, scale=0.40]{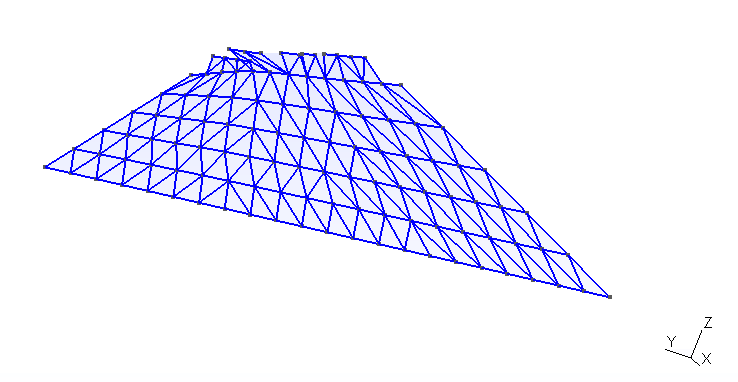} &
\includegraphics[trim = 1cm 0cm 0cm 0cm, clip, scale=0.40]{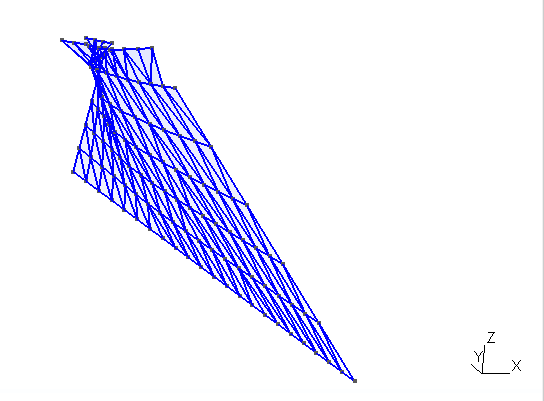}
\end{tabular}
\centering\caption{3D magma-filled crack, called {\it dike}, represented by a set of triangles.\label{figdike}}
\end{figure}
\FloatBarrier
Note that we do not consider this crack as a mesh: This way of representing the crack is just a practical means for describing its geometry from degrees of freedom - the vertices of the triangles - which are independent of the global mesh.

\subsection{Extension of the crack} \label{secextension}
The purpose of this subsection is to explain how to construct an extension of the crack which splits the 3D computational domain, when the crack is represented by a set of triangles, such as the one given in Figure~\ref{figdike}.\\
The process of construction is the following: Given a triangle $ABC$ whose the centroid is denoted by $G$, we define a point $S$ whose orthogonal projection on the triangle is $G$. The distance between $S$ and $G$ is empirically chosen to be of the same order as the size of the triangle (namely the square root of the triangle area). We next define the cone of apex $S$ based on the triangle $ABC$, and we cut this cone with $ABC$. Thus it enables us to define the extension of the triangle as the remaining boundary of the initial cone, and also a region  delimited by the crack and its extension. See Figure~\ref{figextension}.

\begin{figure}[!h]
\begin{tabular} {c|c|c}
\raisebox{33pt}{\includegraphics[trim = 0cm 0cm 0cm 0cm, clip, scale = 0.3]{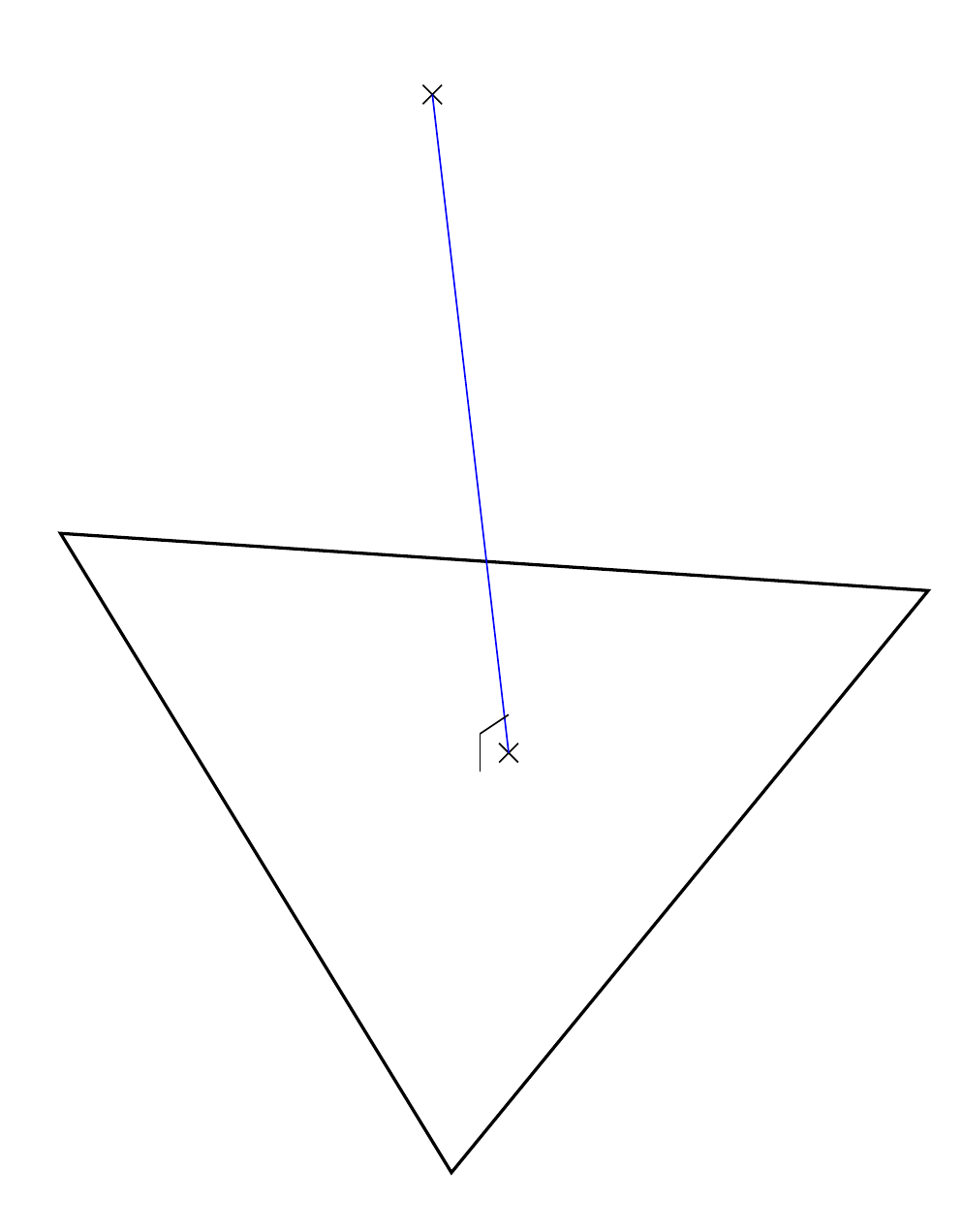}
\setlength{\unitlength}{4144sp}%
\begin{picture}(0,0) (100,100)
\put(-650,1650){\makebox(0,0)[lb]{\smash{{\SetFigFont{12}{14.4}{\rmdefault}{\mddefault}{\updefault}{\color[rgb]{0,0,0}$S$}%
}}}}
\put(-550,700){\makebox(0,0)[lb]{\smash{{\SetFigFont{12}{14.4}{\rmdefault}{\mddefault}{\updefault}{\color[rgb]{0,0,0}$G$}%
}}}}
\put(-1400,1100){\makebox(0,0)[lb]{\smash{{\SetFigFont{12}{14.4}{\rmdefault}{\mddefault}{\updefault}{\color[rgb]{0,0,0}$A$}%
}}}}
\put(-50,1050){\makebox(0,0)[lb]{\smash{{\SetFigFont{12}{14.4}{\rmdefault}{\mddefault}{\updefault}{\color[rgb]{0,0,0}$B$}%
}}}}
\put(-750,10){\makebox(0,0)[lb]{\smash{{\SetFigFont{12}{14.4}{\rmdefault}{\mddefault}{\updefault}{\color[rgb]{0,0,0}$C$}%
}}}}
\end{picture}%
}
&
\includegraphics[trim = 3cm 0cm 3cm 0cm, clip, scale = 0.3]{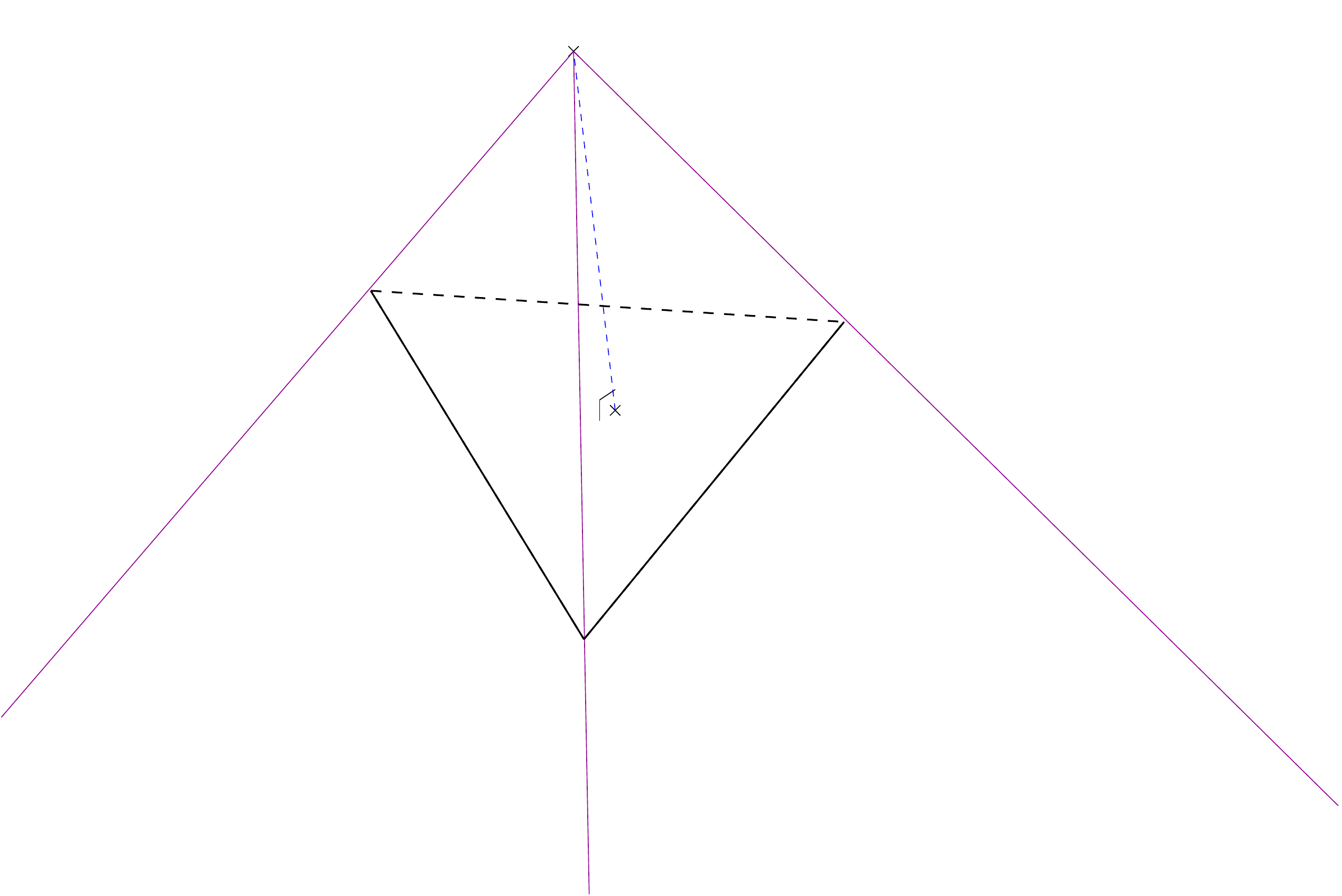}%
\setlength{\unitlength}{4144sp}%
\begin{picture}(0,0) (100,100)
\put(-1400,2350){\makebox(0,0)[lb]{\smash{{\SetFigFont{12}{14.4}{\rmdefault}{\mddefault}{\updefault}{\color[rgb]{0,0,0}$S$}%
}}}}
\put(-1280,1350){\makebox(0,0)[lb]{\smash{{\SetFigFont{12}{14.4}{\rmdefault}{\mddefault}{\updefault}{\color[rgb]{0,0,0}$G$}%
}}}}
\put(-2150,1700){\makebox(0,0)[lb]{\smash{{\SetFigFont{12}{14.4}{\rmdefault}{\mddefault}{\updefault}{\color[rgb]{0,0,0}$A$}%
}}}}
\put(-750,1650){\makebox(0,0)[lb]{\smash{{\SetFigFont{12}{14.4}{\rmdefault}{\mddefault}{\updefault}{\color[rgb]{0,0,0}$B$}%
}}}}
\put(-1450,610){\makebox(0,0)[lb]{\smash{{\SetFigFont{12}{14.4}{\rmdefault}{\mddefault}{\updefault}{\color[rgb]{0,0,0}$C$}%
}}}}
\end{picture}%
&
\includegraphics[trim = 2cm 0cm 5cm 0cm, clip, scale = 0.3]{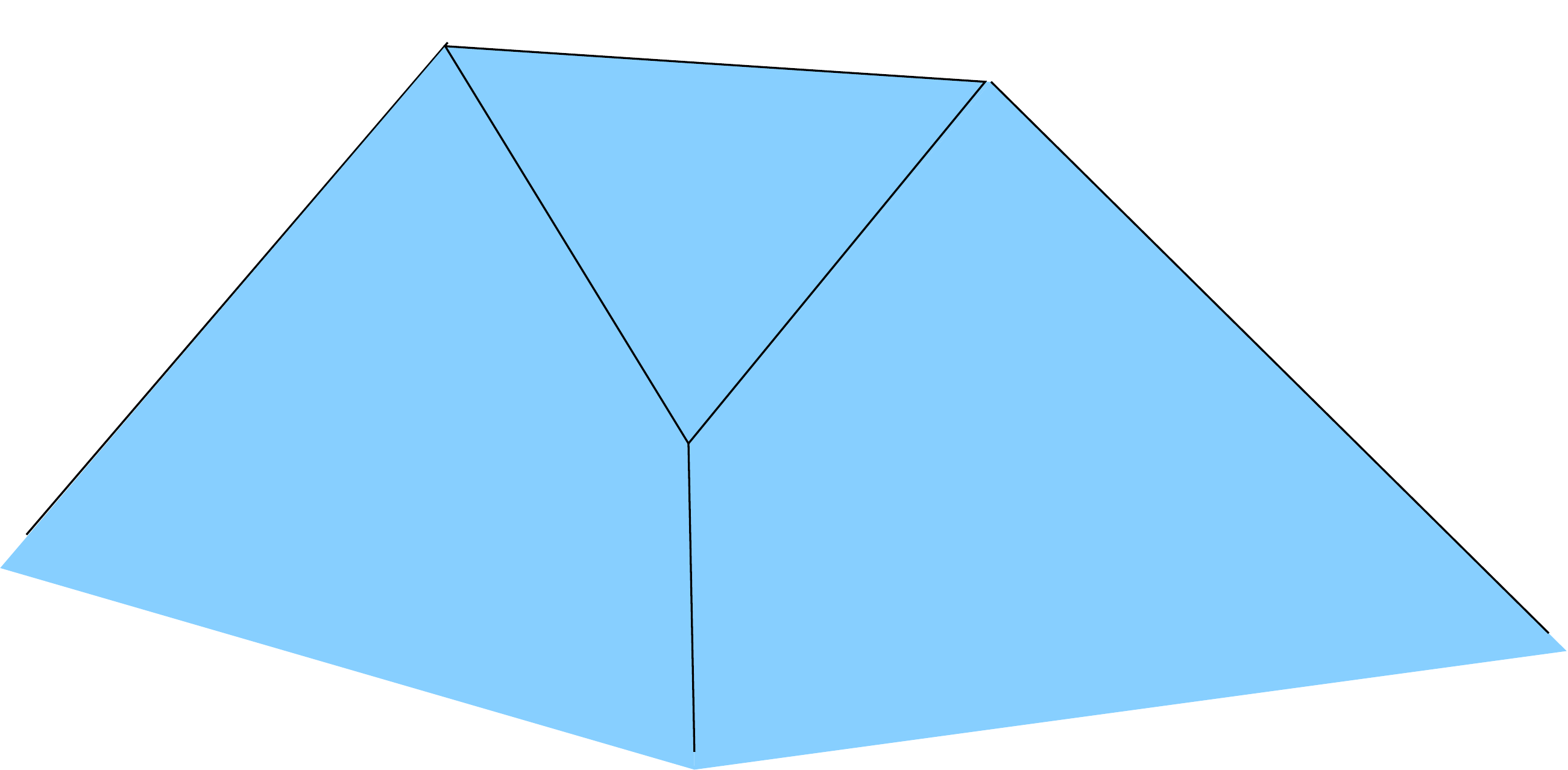}%
\setlength{\unitlength}{4144sp}%
\begin{picture}(0,0) (100,100)
\put(-1950,1750){\makebox(0,0)[lb]{\smash{{\SetFigFont{12}{14.4}{\rmdefault}{\mddefault}{\updefault}{\color[rgb]{0,0,0}$A$}%
}}}}
\put(-1200,2050){\makebox(0,0)[lb]{\smash{{\SetFigFont{22}{24.4}{\rmdefault}{\mddefault}{\updefault}{\color[rgb]{0,0,0}$\Omega^+$}%
}}}}
\put(-1950,650){\makebox(0,0)[lb]{\smash{{\SetFigFont{22}{24.4}{\rmdefault}{\mddefault}{\updefault}{\color[rgb]{0,0,0}$\Omega-$}%
}}}}
\put(-500,1700){\makebox(0,0)[lb]{\smash{{\SetFigFont{12}{14.4}{\rmdefault}{\mddefault}{\updefault}{\color[rgb]{0,0,0}$B$}%
}}}}
\put(-1200,660){\makebox(0,0)[lb]{\smash{{\SetFigFont{12}{14.4}{\rmdefault}{\mddefault}{\updefault}{\color[rgb]{0,0,0}$C$}%
}}}}
\end{picture}%
\end{tabular}
\centering\caption{Construction of the extension of a single triangle modeling a piece of crack.\label{figextension}}
\end{figure}
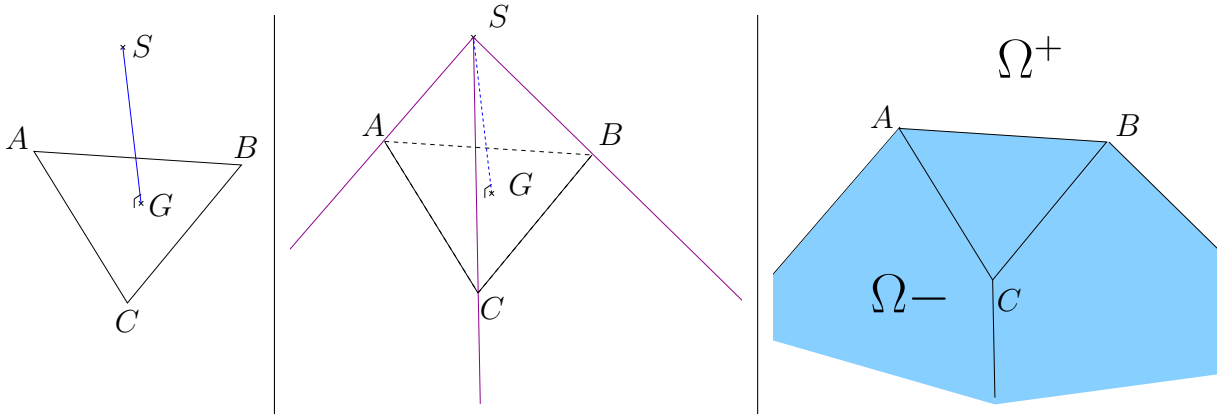
\FloatBarrier

The same process is followed for all the triangles, so that the region we define (blue color in Figure~\ref{figextension}) is obtained as the reunion of all regions corresponding to each triangle. Note that the side of the triangle on which we have set the point $S$ can be chosen arbitrarily for the first triangle, but for practical reason this side - in other words the orientation of the normal vector - has to be kept for the rest of the triangles.

\subsection{Results and comparison with experimental data}
The mesh we use for the whole computational domain is constructed generically from a digital elevation model provided by IGN ({\sc Institut G\'eographique National}, French National Geographic Institute). From these surface data, we first construct a surface mesh using a matlab code (Figure~\ref{figmeshvolc}, left), and then the volume mesh is generated automatically with {\sc Gmsh} software (Figure~\ref{figmeshvolc}, right).

\begin{figure}[!h]
\begin{tabular} {c|c}
\includegraphics[trim = 0cm 0cm 0cm 0cm, clip, scale=0.45]{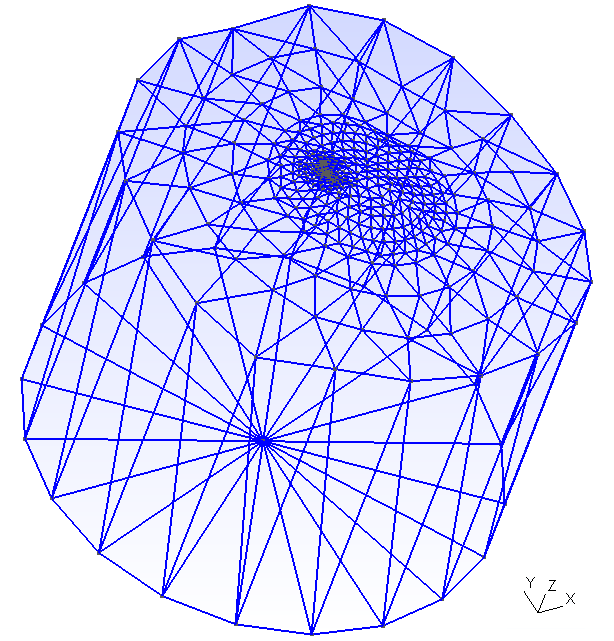} &
\includegraphics[trim = 0cm 0cm 0cm 0cm, clip, scale=0.45]{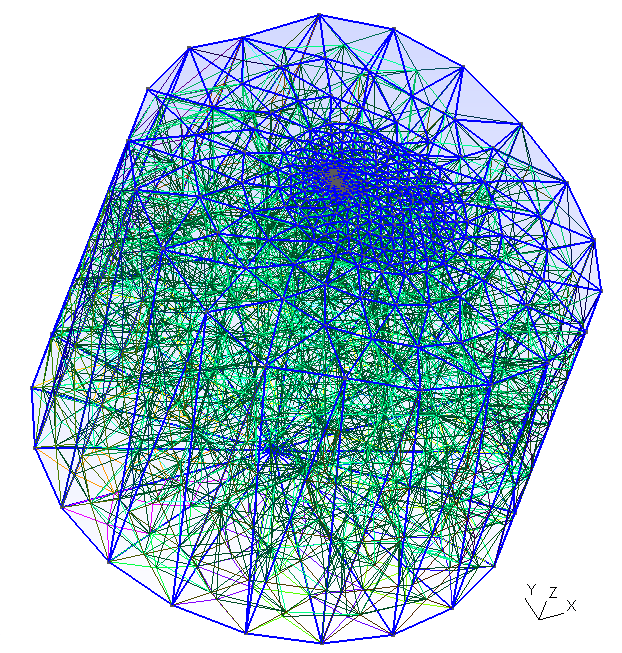}
\end{tabular}
\centering\caption{The 3D mesh of the computational domain is generated from digital elevation model, from which a surface mesh is first generated (left) before creating the volume mesh (right).\label{figmeshvolc}}
\end{figure}
\FloatBarrier
Note that this mesh does not take the crack into account. The local refinement we observe is due to the fact that some {\it eruptive fissures} corresponding to lava emission are observed at the surface (see the upper triangles on Figure~\ref{figdike}), which indicates where the magma-filled crack is intersecting the ground surface. 

Elastic parameters are chosen as $E = 5000$ MPa for the Young modulus and $\nu = 0.25$ for the Poisson ratio following \cite{Fukushima2005}. Recall that the Lam\'e coefficients can be related to the elastic coefficients $\nu$ and $E$ by the formulas below:
\begin{eqnarray*}
\lambda_L = \frac{E\nu}{(1-2\nu)(1+\nu)}, & \quad & \mu_L = \frac{E}{2(1+\nu)}. 
\end{eqnarray*}
Pressure forces are applied on both sides of the crack $\Gamma_T$, with values $p = 5$ MPa. The crater visible on the figures below has a diameter approximately equal to 1 km, and the crack surface is about 2.8 km$^2$. The global mesh consists of 8376 points and 4969 tetrahedral cells. The numerical simulation which solves the displacement corresponding to the crack of Figure~\ref{figdike} in the whole computational domain is presented in Figure~\ref{figvolcan}.

\begin{figure}[!h]
\begin{tabular} {c|c}
\includegraphics[trim = 1cm 0cm 2cm 0cm, clip, scale=0.25]{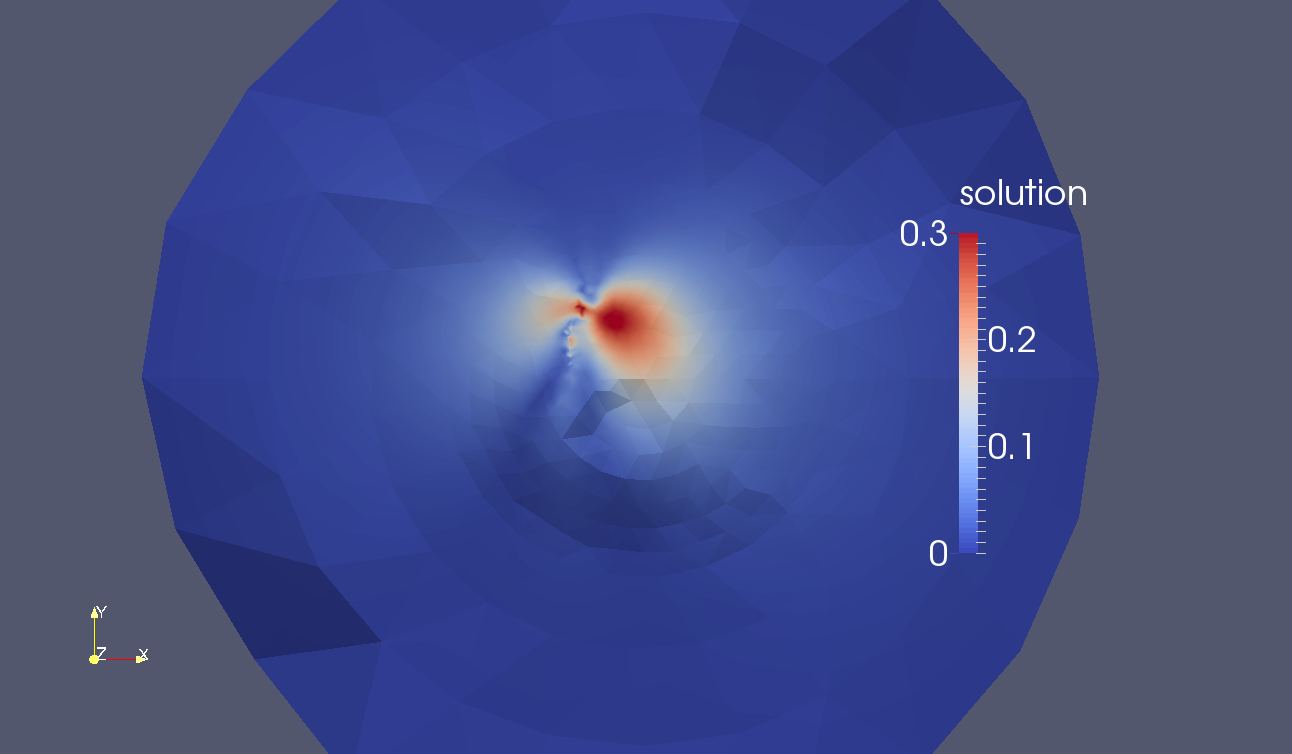} &
\includegraphics[trim = 1cm 0cm 2cm 0cm, clip, scale=0.25]{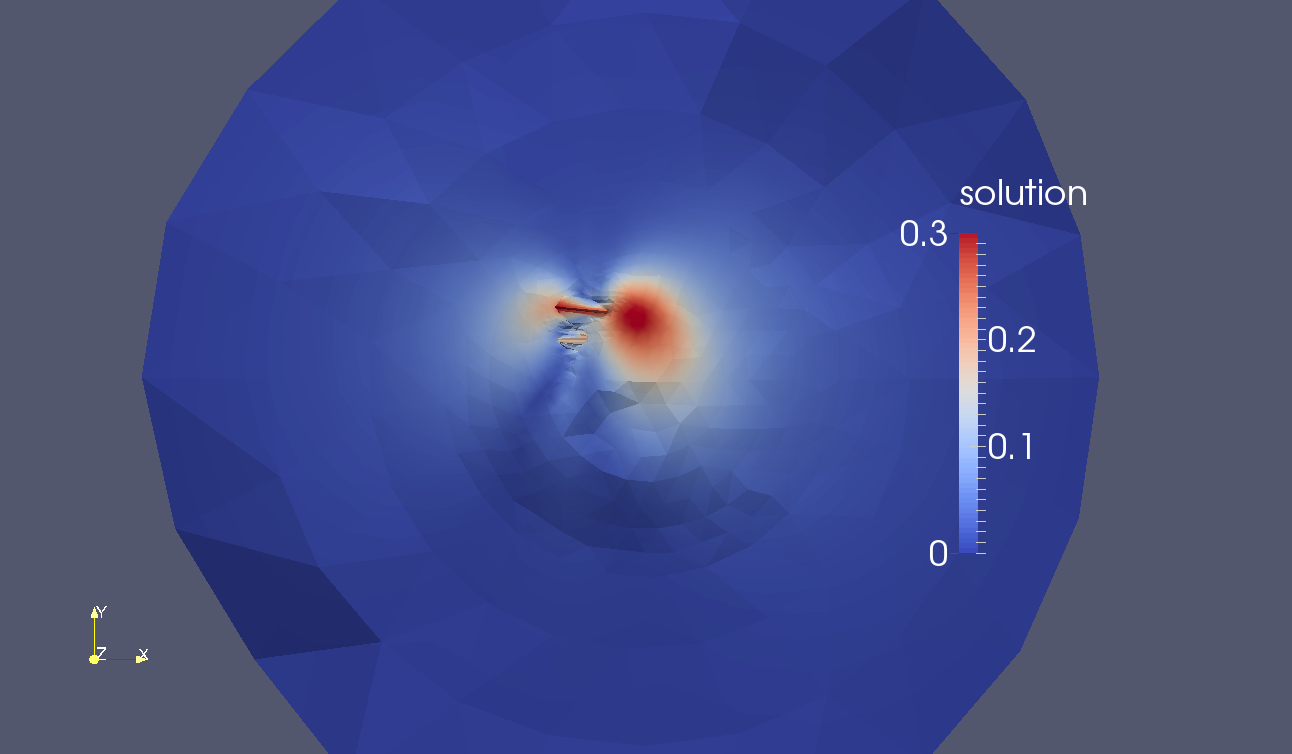}
\\ \hline \\
\includegraphics[trim = 1cm 0cm 2cm 0cm, clip, scale=0.25]{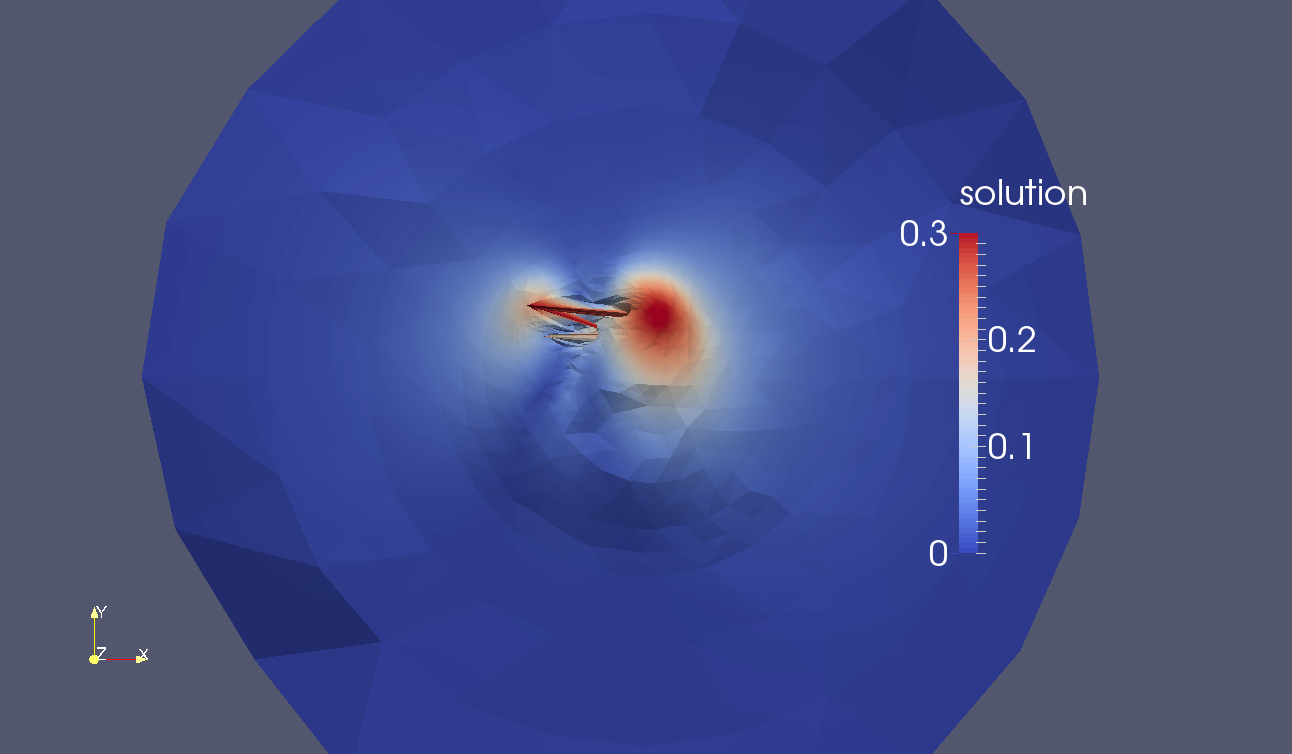}
&
\includegraphics[trim = 1cm 0cm 2cm 0cm, clip, scale=0.25]{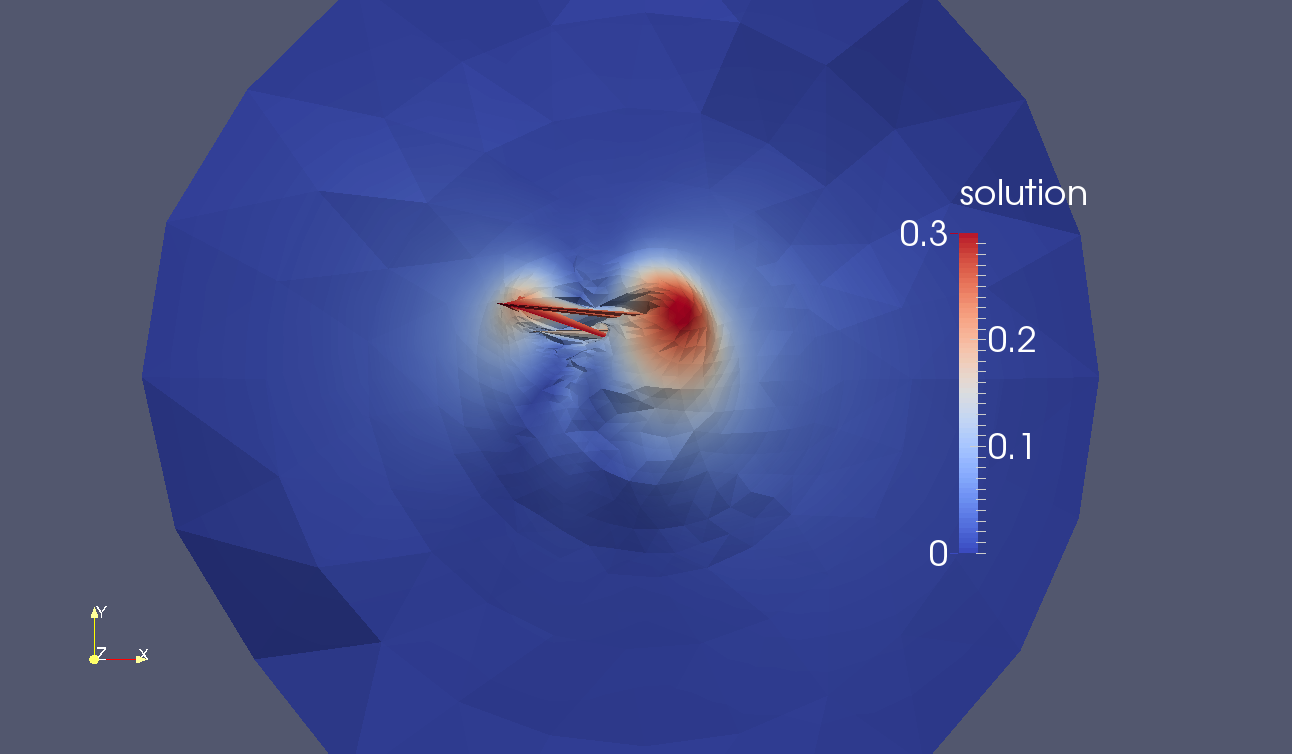}
\end{tabular}
\begin{center}\caption{Amplitude of the displacement due to pressures applied on the crack. The domain is warped by $1$, $1000$, $2000$ and $3000$ respectively, w.r.t. the displacement. The maximum - about 0.65 m - is reached at depth.\label{figvolcan}}
\end{center}
\end{figure}
\FloatBarrier
The large deformations observed at the surface (when the domain is warped w.r.t. the displacement) are due to surface openings at the top of the crack (Figure~\ref{figdike}). We can remove these surface openings, by removing the top triangles of the fracture of Figure~\ref{figdike}; The corresponding displacement is represented in Figure~\ref{figvolcan_ter}.
\begin{figure}[!h]
\begin{tabular} {c|c}
\includegraphics[trim = 0cm 0cm 0cm 0cm, clip, scale=0.30]{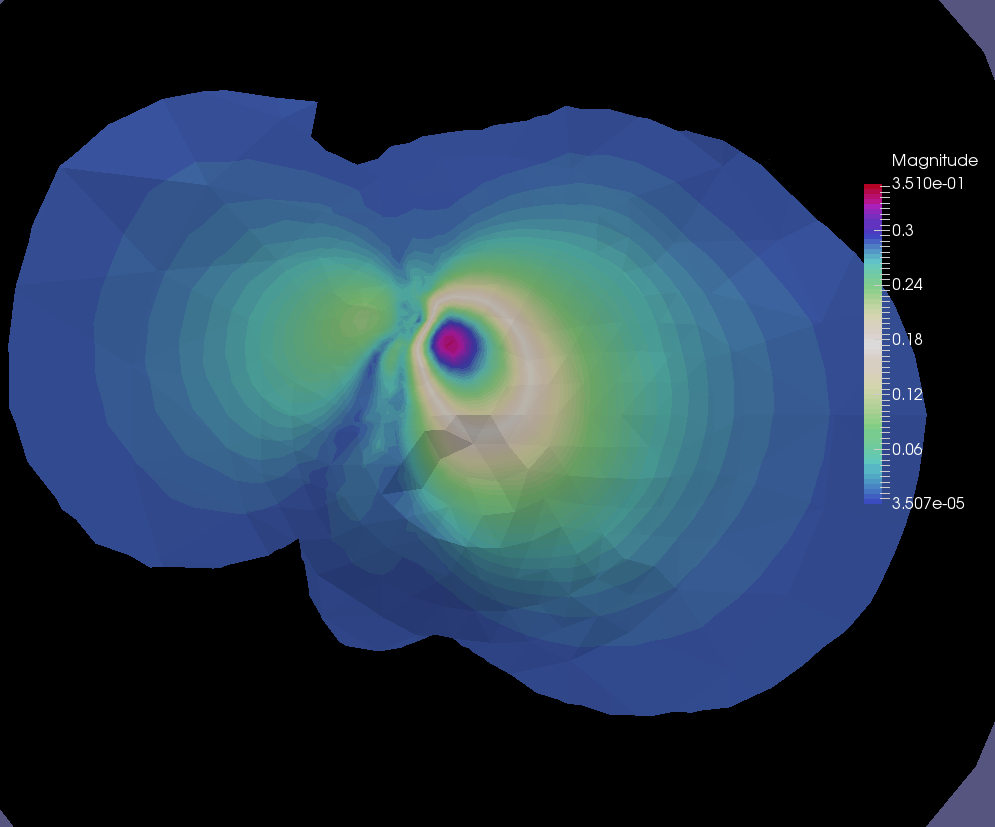} &
\includegraphics[trim = 0cm 0cm 0cm 0cm, clip, scale=0.30]{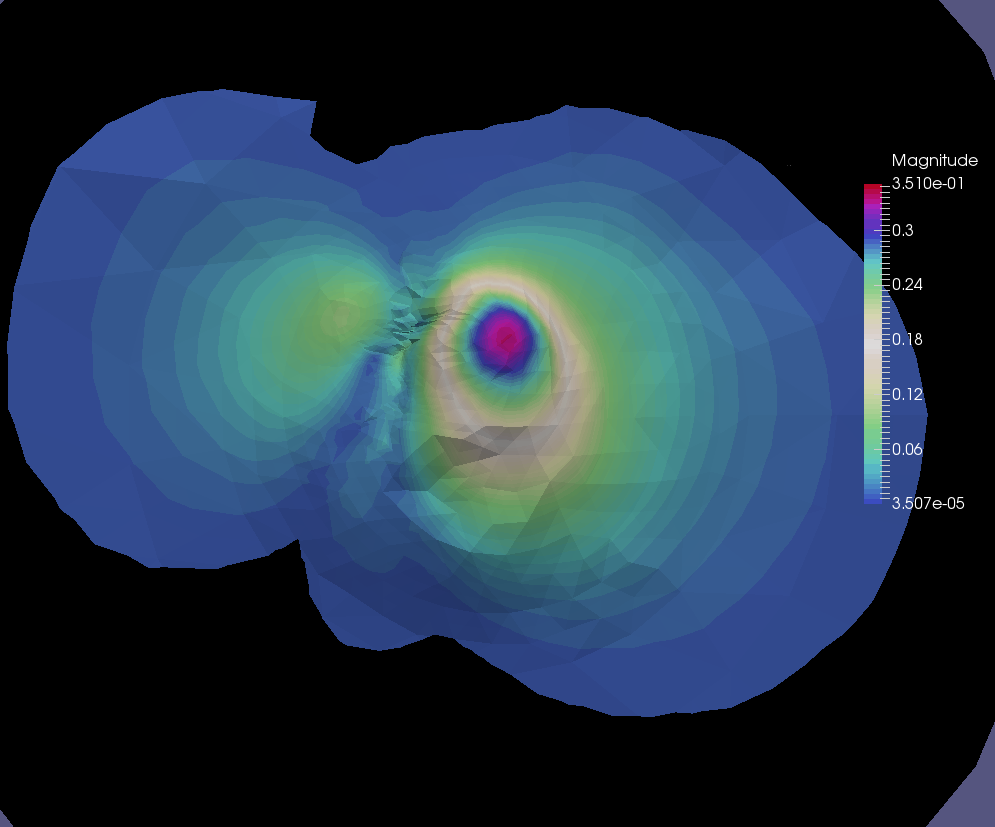}
\end{tabular}
\centering\caption{Amplitude of the displacement due to pressures applied on the fracture without surface openings. The domain is let unwarped (left) and then warped by $2000$ (right), w.r.t. the displacement. The maximum - about 0.62 m - is reached in depth.\label{figvolcan_ter}}
\end{figure}
\FloatBarrier

Our simulation can be compared qualitatively with measurements obtained by synthetic aperture radar \cite{Fukushima2005} (see Figure 16 or Figure 12 for instance), which validates qualitatively our simulation, and thus our method. Moreover, the order of magnitude of the displacement computed (about 0.50 m) seems also to correspond to the experimental observations.

\section{Conclusion} \label{seconc}
In this work, we have proposed a new finite element method for considering a crack and simulating the displacement and stress induced by a crack located inside a heterogeneous and anisotropic elastic medium. In order to avoid re-meshing and save computation time when inverting from surface deformation data or when simulating a propagating fracture, we consider a single global mesh and we develop a fictitious domain approach. Numerical tests were performed in order to highlight the optimal accuracy of the method. Robustness with respect to the geometry is ensured by the implementation of a stabilization technique. A theoretical analysis underlines the relevance of the latter for computing a good approximation of a dual variable. Numerical tests enable us to think that our method is practical for an algorithmic framework in which multiple ways of taking into account the crack will be considered. Indeed, the inversion of the position/shape of the crack - inside a volcano, from surface measurements for instance - is an algorithmic problem that can involve this kind of dual variables (shape derivatives techniques involving adjoint direct systems for instance). The illustration of this method on an inverse problem will be presented in a forthcoming work. The fictitious domain approach we have developed enables us to consider a single global mesh for all kind of geometries associated with the crack. In the context of inverse problems, we expect its simplicity to be quite effective in terms of computational time.

\section*{Acknowledgments}
Our softwares are written in C++ language, using the free finite element library {\sc GETFEM++} \cite{Getfem} developed by Yves Renard and Julien Pommier. The authors thank the {\sc GETFEM++} users community for collaborative efforts.

\nocite*
\bibliographystyle{alpha}
\bibliography{FD_Crack_Ref}

\newcommand{\etalchar}[1]{$^{#1}$}
\begin{thebibliography}{SMMB00}

\bibitem[BDH96]{qhull}
C.~B. Barber, D.~P. Dobkin, and H.~Huhdanpaa.
\newblock The quickhull algorithm for convex hulls.
\newblock {\em ACM Trans. Math. Software}, 22(4):469--483, 1996.

\bibitem[BF91]{Fortin}
F.~Brezzi and M.~Fortin.
\newblock {\em Mixed and hybrid finite element methods}, volume~15 of {\em
  Springer Series in Computational Mathematics}.
\newblock Springer-Verlag, New York, 1991.

\bibitem[BH91]{Barbosa1}
H.~J.~C. Barbosa and T.~J.~R. Hughes.
\newblock The finite element method with {L}agrange multipliers on the
  boundary: circumventing the {B}abu\v ska-{B}rezzi condition.
\newblock {\em Comput. Methods Appl. Mech. Engrg.}, 85(1):109--128, 1991.

\bibitem[BH92]{Barbosa2}
H.~J.~C. Barbosa and T.~J.~R. Hughes.
\newblock Boundary {L}agrange multipliers in finite element methods: error
  analysis in natural norms.
\newblock {\em Numer. Math.}, 62(1):1--15, 1992.

\bibitem[BSES05]{Sokolowski}
Z.~Belhachmi, J.~M. Sac-Ep{\'e}e, and J.~Sokolowski.
\newblock Mixed finite element methods for smooth domain formulation of crack
  problems.
\newblock {\em SIAM J. Numer. Anal.}, 43(3):1295--1320 (electronic), 2005.

\bibitem[CDF05]{Fukushima2005}
V.~Cayol, P.~Durand, and Y.~Fukushima.
\newblock Finding realistic dike models from interferometric synthetic aperture
  radar data: The february 2000 eruption at piton de la fournaise.
\newblock {\em Journal of Geophysical Research: Solid Earth}, 110(B3), 2005.

\bibitem[CFL14]{Court}
S.~Court, M.~Fourni{\'e}, and A.~Lozinski.
\newblock A fictitious domain approach for the {S}tokes problem based on the
  extended finite element method.
\newblock {\em Internat. J. Numer. Methods Fluids}, 74(2):73--99, 2014.

\bibitem[CHM12]{Choi2010}
Y.~J. Choi, M.~A. Hulsen, and H.~E.~H. Meijer.
\newblock Simulation of the flow of a viscoelastic fluid around a stationary
  cylinder using an extended finite element method.
\newblock {\em Comput. \& Fluids}, 57:183--194, 2012.

\bibitem[Cia97]{Ciarlet97}
P.~G. Ciarlet.
\newblock {\em Mathematical elasticity. {V}ol. {II}}, volume~27 of {\em Studies
  in Mathematics and its Applications}.
\newblock North-Holland Publishing Co., Amsterdam, 1997.
\newblock Theory of plates.

\bibitem[CLR08]{Chahine}
E.~Chahine, P.~Laborde, and Y.~Renard.
\newblock Crack tip enrichment in the xfem using a cutoff function.
\newblock {\em Int. J. Numer. Meth. Engng}, 75(6):629--646, 2008.

\bibitem[CMR83]{Crouch}
S.~L. Crouch, Starfield~A. M., and J.~Rizzo.
\newblock Boundary element methods in solid mechanics.
\newblock {\em J. Appl. Mech}, 50(3):704--705, 1983.

\bibitem[DGL]{SuperLU}
J.~W. Demmel, J.~R. Gilbert, and X.~S Li.
\newblock A general purpose library for the direct solution of large, sparse,
  nonsymmetric systems.

\bibitem[EG04]{Ern}
A.~Ern and J.-L. Guermond.
\newblock {\em Theory and practice of finite elements}, volume 159 of {\em
  Applied Mathematical Sciences}.
\newblock Springer-Verlag, New York, 2004.

\bibitem[FB10]{reviewXfem}
T.-P. Fries and T.~Belytschko.
\newblock The extended/generalized finite element method: an overview of the
  method and its applications.
\newblock {\em Internat. J. Numer. Methods Engrg.}, 84(3):253--304, 2010.

\bibitem[GW08]{Gerstenberger2008}
A.~Gerstenberger and W.~A. Wall.
\newblock An extended finite element method/{L}agrange multiplier based
  approach for fluid-structure interaction.
\newblock {\em Comput. Methods Appl. Mech. Engrg.}, 197(19-20):1699--1714,
  2008.

\bibitem[HR09]{HaslR}
J.~Haslinger and Y.~Renard.
\newblock A new fictitious domain approach inspired by the extended finite
  element method.
\newblock {\em SIAM J. Numer. Anal.}, 47(2):1474--1499, 2009.

\bibitem[LPRS05]{Laborde}
P.~Laborde, J.~Pommier, Y.~Renard, and M.~Sala\"{u}n.
\newblock High-order extended finite element method for cracked domains.
\newblock {\em Int. J. Numer. Meth. Engng}, 64(3):354--381, 2005.

\bibitem[LYMS14]{Liu2014}
W.~Liu, Q.~D. Yang, S.~Mohammadizadeh, and X.~Y. Su.
\newblock An efficient augmented finite element method for arbitrary cracking
  and crack interaction in solids.
\newblock {\em Internat. J. Numer. Methods Engrg.}, 99(6):438--468, 2014.

\bibitem[MDB99]{MoesD}
N.~Mo{\"e}s, J.~Dolbow, and T.~Belytschko.
\newblock A finite element method for crack growth without remeshing.
\newblock {\em Internat. J. Numer. Methods Engrg.}, 46(1):131--150, 1999.

\bibitem[MGB02]{MoesG}
N.~Mo{\"e}s, A.~Gravouil, and T.~Belytschko.
\newblock Non-planar 3d crack growth by the extended finite element and level
  sets, part i: Mechanical model.
\newblock {\em Internat. J. Numer. Methods Engrg.}, 53(11):2549--2568, 2002.

\bibitem[ONK{\etalchar{+}}11]{Ozawa}
S.~Ozawa, T.~Nishimura, T.~Kobayashi, M.~Tobita, and T.~Imakiire.
\newblock Coseismic and postseismic slip of the 2011 magnitude-9 tohoku-oki
  earthquake.
\newblock {\em Nature}, 475:373--376, 2011.

\bibitem[PDD{\etalchar{+}}83]{Pollard}
D.~D Pollard, P.~T. Delaney, W.~A. Duffield, E.~T. Endo, and Okamura~A. T.
\newblock Surface deformation in volcanic rift zones.
\newblock {\em Tectonophysics}, 94(1-4):541--584, 1983.

\bibitem[RP]{Getfem}
Y.~Renard and J.~Pommier.
\newblock {\em An open source generic C++ library for finite element methods}.

\bibitem[SBCB03]{Stazi}
F.~L. Stazi, E.~Budyn, J.~Chessa, and T.~Belytschko.
\newblock An extended finite element method with high-order elements for curved
  cracks.
\newblock {\em Comput. Mech.}, 31(1-2):38--48, 2003.

\bibitem[SCMB01]{Stolarska}
M.~Stolarska, D.~L. Chopp, N.~Mo\"{e}s, and T.~Belytschko.
\newblock Modelling crack growth by level sets in the extended finite element
  method.
\newblock {\em Int. J. Numer. Meth. Engng}, 51(8):943--960, 2001.

\bibitem[SMMB00]{Sukumar}
N.~Sukumar, N.~Mo\"{e}s, B.~Moran, and T.~Belytschko.
\newblock Extended finite element method for three dimensional crack modelling.
\newblock {\em Int. J. Numer. Meth. Engng}, 48(11):1549--1570, 2000.

\bibitem[Sta01]{cauchycrack}
Georgios~E. Stavroulakis.
\newblock {\em Inverse and crack identification problems in engineering
  mechanics}, volume~46 of {\em Applied Optimization}.
\newblock Kluwer Academic Publishers, Dordrecht, 2001.

\bibitem[Ste95]{Stenberg}
R.~Stenberg.
\newblock On some techniques for approximating boundary conditions in the
  finite element method.
\newblock {\em J. Comput. Appl. Math.}, 63(1-3):139--148, 1995.
\newblock International Symposium on Mathematical Modelling and Computational
  Methods Modelling 94 (Prague, 1994).

\bibitem[Tem83]{Temam}
R.~Temam.
\newblock {\em Probl\`emes math\'ematiques en plasticit\'e}, volume~12 of {\em
  M\'ethodes Math\'ematiques de l'Informatique [Mathematical Methods for
  Information Science]}.
\newblock Gauthier-Villars, Montrouge, 1983.

\bibitem[WCKd12]{Wauthier}
C.~Wauthier, V.~Cayol, F.~Kervyn, and N.~d'Oreye.
\newblock Magma sources involved in the 2002 nyiragongo eruption, as inferred
  from an insar analysis.
\newblock {\em Journal of Geophysical Research}, 117(B05411), 2012.

\end{thebibliography}

\end{document}